\documentclass[11pt,french]{smfart}
\usepackage{amsfonts, amssymb, amsmath, latexsym, enumerate, epsfig, color, mathrsfs, fancyhdr, pdfsync, eurosym, endnotes, stmaryrd}
\usepackage[all]{xy}
\usepackage{smfthm}

\usepackage[frenchb,english]{babel}

\usepackage{raccourcis}

\usepackage[mac]{inputenc}
\usepackage[T1]{fontenc}

\usepackage{smfhref}

\begin{document}
\setlength{\baselineskip}{0.54cm}	

\title{Espaces de Berkovich sur~$\Z$ : \'etude locale}
\alttitle{Berkovich spaces over~$\Z$: local study}
\author{J\'er\^ome Poineau}
\address{Institut de recherche math\'ematique avanc\'ee, 7, rue Ren\'e Descartes, 67084 Strasbourg, France}
\email{jerome.poineau@math.unistra.fr}
\thanks{L'auteur est membre du projet jeunes chercheurs \og Berko \fg\ de l'Agence Nationale de la Recherche.}

\date{\today}

\altkeywords{Berkovich spaces, global analytic geometry, Weierstra{\ss} theorems, noetherianity, regularity, excellence, analytic continuation, coherence}

\begin{abstract}
\selectlanguage{french}
Nous \'etudions les propri\'et\'es locales des espaces de Berkovich sur~$\Z$. \`A l'aide de th\'eor\`emes de Weierstra{\ss}, nous montrons que les anneaux locaux de ces espaces sont noeth\'eriens, r\'eguliers dans le cas des espaces affines et excellents. Nous prouvons \'egalement que le faisceau structural est coh\'erent. Nos m\'ethodes s'adaptent \`a d'autres anneaux de base (corps valu\'es, anneaux de valuation discr\`ete, anneaux d'entiers de corps de nombres, etc.) et traitent de fa\c{c}on unifi\'ee espaces complexes et $p$-adiques.
\end{abstract}

\begin{altabstract}
\selectlanguage{english}
We investigate the local properties of Berkovich spaces over~$\Z$. Using Weierstra{\ss} theorems, we prove that the local rings of those spaces are noetherian, regular in the case of affine spaces and excellent. We also show that the structure sheaf is coherent. Our methods apply over other base rings (valued fields, discrete valuation rings, rings of integers of number fields, etc.) and provide a unified treatment of complex and $p$-adic spaces.
\end{altabstract}

\maketitle

\section*{Introduction}

Les fondations de la g\'eom\'etrie analytique complexe \'etaient \'etablies depuis longtemps lorsque John Tate a propos\'e, en 1961, la d\'efinition d'une g\'eom\'etrie analytique $p$-adique poss\'edant des propri\'et\'es analogues (\textit{cf.}~\cite{Tate}). Ce sujet a, par la suite, suscit\'e un vif int\'er\^et et plusieurs g\'eom\'etries $p$-adiques (ou sur d'autres corps valu\'es complets) ont vu le jour, chacune poss\'edant ses sp\'ecificit\'es. Mentionnons les th\'eories de Michel Raynaud, o\`u les espaces sont vus comme des fibres g\'en\'eriques de sch\'emas formels (\textit{cf.}~\cite{tableronde}), et de Roland Huber, qui permet de retrouver les points du topos (\textit{cf}.~\cite{Huber1} et~\cite{Huber2}). 

Dans ce texte, nous nous int\'eresserons \`a une autre de ces g\'eom\'etries : celle d\'evelopp\'ee par Vladimir G.~Berkovich (\textit{cf.}~\cite{rouge} et~\cite{bleu}). Outre ses nombreux succ\`es dans des domaines vari\'es (programme de Langlands, int\'egration motivique, dynamique, th\'eorie de Hodge $p$-adique, etc.), elle poss\`ede la particularit\'e int\'eressante de permettre de d\'efinir des espaces analytiques sur tout anneau de Banach. En particulier, les espaces de Berkovich sur le corps des nombres complexes~$\C$ (muni de la valeur absolue usuelle~$|.|_{\infty}$) existent et ne sont autres que les espaces analytiques complexes usuels. Mieux encore, on peut d\'efinir des espaces de Berkovich sur l'anneau des entiers~$\Z$ (\'egalement muni de la valeur absolue usuelle~$|.|_{\infty}$) et l'on obtient alors des espaces naturellement fibr\'es en espaces analytiques complexes et $p$-adiques. Pourtant, en dehors du cas des espaces d\'efinis sur un corps valu\'e complet, bien peu de r\'esultats sont connus. Signalons tout de m\^eme l'ouvrage~\cite{asterisque} o\`u le cas de la droite de Berkovich sur~$\Z$ est trait\'e de fa\c{c}on approfondie.

Dans cet article, nous nous proposons de contribuer \`a l'\'etude locale des espaces de Berkovich de dimension sup\'erieure sur~$\Z$ (ou, plus g\'en\'eralement, sur un anneau d'entiers de corps de nombres), espaces que nous imaginons comme des fibrations en espaces affines au-dessus du spectre analytique~$\Mc(\Z)$. En g\'eom\'etrie analytique complexe, cette \'etude repose de fa\c{c}on cruciale sur le th\'eor\`eme de division de Weierstra{\ss}, un th\'eor\`eme qui permet de diviser une fonction analytique par une autre, au voisinage d'un point, de fa\c{c}on \`a obtenir un reste polynomial. Si, dans ce dernier cadre, tous les points sont d\'efinis sur le corps de base~$\C$, il en va tout autrement dans notre situation, o\`u le corps r\'esiduel~$\Hs(x)$ d'un point~$x$ peut \^etre une extension non triviale du corps~$\Hs(b)$ sur lequel sa fibre est d\'efinie. Dans ce texte, nous d\'emontrons un th\'eor\`eme de division de Weierstra{\ss} pour les points dits rigides des fibres (ceux en lesquels l'extension $\Hs(x)/\Hs(b)$ est alg\'ebrique). 

Le th\'eor\`eme de division de Weierstra{\ss} est l'ingr\'edient essentiel permettant d'acc\'eder \`a la structure des anneaux locaux. Nous poursuivons notre travail en d\'emontrant que ceux de l'espace~$\E{n}{\Z}$ sont noeth\'eriens et r\'eguliers, puis que le faisceau structural est coh\'erent, g\'en\'eralisant ainsi en toute dimension les r\'esultats de~\cite{asterisque}. Nous les compl\'etons encore en d\'emontrant que les anneaux locaux sont excellents. Pour des espaces analytiques g\'en\'eraux sur~$\Z$, ou un anneau d'entiers de corps de nombres, ces r\'esultats entra\^{\i}nent imm\'ediatement que les anneaux locaux sont excellents et que le faisceau structural est coh\'erent. Ces anneaux locaux \'etant \'egalement hens\'eliens (\textit{cf.}~\cite{asterisque}), ils v\'erifient l'approximation d'Artin. Nous donnons un exemple concret d'application de ce r\'esultat.

Signalons pour finir que, bien que notre int\'er\^et r\'esid\^at principalement dans l'\'etude des espaces analytiques sur un anneau d'entiers de corps de nombres, nous nous sommes efforc\'es, lorsque cela ne compliquait pas outre mesure les preuves, d'\'enoncer les r\'esultats dans une g\'en\'eralit\'e maximale en t\^achant en particulier de limiter les conditions impos\'ees \`a l'anneau de Banach~$\As$ sur lequel les espaces sont d\'efinis. Ainsi le th\'eor\`eme de division de Weierstra{\ss} vaut-il sur un anneau de Banach quelconque (avec cependant une condition technique lorsque la fibre contenant le point~$x$ est d\'efinie sur un corps~$\Hs(b)$ imparfait et trivialement valu\'e). Quant aux r\'esultats de noeth\'erianit\'e, r\'egularit\'e et coh\'erence pour l'espace affine analytique~$\E{n}{\As}$, nous les d\'emontrons lorsque l'anneau de Banach~$\As$ appartient \`a une classe qui contient non seulement les anneaux d'entiers de corps de nombres mais aussi les corps valu\'es et les anneaux de valuation discr\`ete. 

Pr\'ecisons qu'\`a l'exception de celles utilis\'ees pour d\'emontrer l'excellence, nos m\'ethodes, inspir\'ees de celles de la g\'eom\'etrie analytique complexe, ne n\'ecessitent aucune connaissance \textit{a priori} de la g\'eom\'etrie des espaces de Berkovich sur un corps d\'epassant celle des r\'esultats classiques en dimension~$1$. En particulier, nous obtenons de nouvelles preuves des r\'esultats de noeth\'erianit\'e et r\'egularit\'e pour les anneaux locaux des espaces~$\E{n}{k}$, o\`u~$k$ est un corps valu\'e complet, ainsi que de la coh\'erence du faisceau structural, par des m\'ethodes qui traitent de fa\c{c}on unifi\'ee corps archim\'ediens et corps ultram\'etriques.

\medskip

Ajoutons quelques mots concernant la structure du texte. Dans la premi\`ere section, nous rappelons quelques notations et d\'efinitions, ainsi que des r\'esultats \'el\'ementaires, ayant trait aux espaces de Berkovich sur un anneau de Banach, tir\'es de~\cite{rouge} et~\cite{asterisque}. Ainsi que nous l'avons expliqu\'e plus haut, nous souhaitons \'etudier les points rigides des fibres. Pour ce faire, un moyen naturel consiste \`a les envoyer, par un morphisme fini, sur des points rationnels. C'est pourquoi nous consacrons les sections~$2$ et~$3$ \`a \'etablir quelques propri\'et\'es d'alg\`ebres et de morphismes finis, moyennant deux conditions techniques~$(D)$ et~$(N)$, par des arguments adapt\'es de~\cite{asterisque}. Dans les deux sections suivantes, nous \'etudions les conditions introduites : nous montrons que la condition~$(D)$ est toujours v\'erifi\'ee et donnons des crit\`eres pratiques pour que la condition~$(N)$ le soit. Nous \'etablissons en passant le fait que les corps r\'esiduels~$\kappa(x)$ des points des espaces de Berkovich sont des corps hens\'eliens. \`A la section~$6$, nous regroupons les r\'esultats obtenus pour relier l'anneau local en un point rigide \`a l'anneau local en un point rationnel, puis, \`a la section~$7$, d\'emontrons le th\'eor\`eme de division de Weierstra{\ss} annonc\'e. Enfin, dans les trois derni\`eres sections, nous utilisons ce r\'esultat pour \'etudier de fa\c{c}on pr\'ecise le faisceau structural de l'espace affine sur certains anneaux de Banach (corps valu\'es, anneaux de valuation discr\`ete, anneaux d'entiers de corps de nombres, notamment). Nous d\'emontrons que les germes de ce faisceau sont des anneaux locaux noeth\'eriens et r\'eguliers, excellents si l'anneau de base est de caract\'eristique nulle, puis que le faisceau lui-m\^eme est coh\'erent et noeth\'erien au sens de M.~Kashiwara.

\bigskip

\noindent {\bf Remerciements}

Nous remercions Antoine Chambert-Loir et Antoine Ducros pour leurs remarques et conseils.

\section{Rappels}

Fixons un anneau de Banach~$(\As,\|.\|)$ et un entier~$n\in\N$. Au chapitre~1 de~\cite{rouge}, V.~Berkovich explique comment d\'efinir un espace localement annel\'e~$\E{n}{\As}$ (not\'e \'egalement~$\Ms(\As)$ si $n=0$), qu'il appelle \textbf{espace affine analytique} de dimension~$n$ sur~$\As$. Ensemblistement, ses points sont les semi-normes multiplicatives sur l'anneau de polyn\^omes $\As[\bT]=\As[T_{1},\dots,T_{n}]$ dont la restriction \`a~$\As$ est born\'ee par la norme donn\'ee~$\|.\|$. Pour tout point~$x$ de cet espace, associ\'e \`a une semi-norme multiplicative~$|.|_{x}$, on d\'efinit un corps r\'esiduel~$\Hs(x)$ de la fa\c{c}on suivante. L'ensemble $\p_{x} = \{f\in \As[\bT]\,|\, |f|_{x} = 0\}$ est un id\'eal premier de~$\As[\bT]$ et~$|.|_{x}$ induit une valeur absolue sur le corps $\Frac(\As[\bT]/\p_{x})$. Son compl\'et\'e~$\Hs(x)$ sera le corps r\'esiduel voulu. Remarquons que nous disposons d'une \og application d'\'evaluation \fg\ naturelle $f \in \As[\bT] \mapsto f(x) \in \Hs(x)$.

V.~Berkovich munit alors l'ensemble~$\E{n}{\As}$ ainsi construit de la topologie la plus faible qui rende continues les applications $|.|_{x} \in \E{n}{\As} \mapsto |f|_{x} \in\R$, pour $f \in \As[\bT]$. Le fait que les corps r\'esiduels~$\Hs(x)$ soient munis d'une valeur absolue permet de d\'efinir une notion de convergence uniforme, puis de fonction analytique. Pr\'ecis\'ement, les fonctions analytiques sur un ouvert~$U$ de~$\E{n}{\As}$ sont les applications
\[f : U \to \bigsqcup_{x\in U} \Hs(x)\]
telles que, pour tout $x\in U$, on ait $f(x) \in\Hs(x)$ et qui sont localement limites uniformes de fractions rationnelles (c'est-\`a-dire d'\'el\'ements de l'anneau total des fractions de~$\As[\bT]$) sans p\^oles. Ces constructions d\'efinissent un faisceau, not\'e~$\Os_{\E{n}{\As}}$.\footnote{D\`es le chapitre~2 de~\cite{rouge}, V.~Berkovich se restreint au cas o\`u l'anneau de Banach~$(\As,\|.\|)$ appartient \`a une certaine classe : celle des alg\`ebres affino\"{\i}des sur un corps valu\'e~$k$. Les d\'efinitions de l'ensemble~$\E{n}{\As}$ et de sa topologie sont inchang\'ees. En revanche, le faisceau structural ne co\"{\i}ncide avec celui que nous venons de d\'ecrire que lorsque l'alg\`ebre affino\"{\i}de est r\'eduite.}

Signalons que nous nous permettrons encore d'\'ecrire $\Os_{\E{n}{\As}}(V)$ lorsque~$V$ n'est pas ouvert. En g\'en\'eral, il faudrait le d\'efinir comme l'anneau des sections continues de l'espace \'etal\'e correspondant au-dessus de~$V$. Si~$V$ est une partie compacte de~$\E{n}{\As}$, cela revient \`a poser 
\[\Os_{\E{n}{\As}}(V) = \varinjlim_{U \supset V} \Os_{\E{n}{\As}}(U),\] 
o\`u~$U$ d\'ecrit l'ensemble des voisinages ouverts de~$V$. 

Soit~$x$ un point de~$\E{n}{\As}$. Nous dirons qu'un \'el\'ement~$f$ de~$\Os_{\E{n}{\As},x}$ est \mbox{\textbf{d\'efini}} sur une partie~$U$ de~$\E{n}{\As}$ contenant~$x$ s'il poss\`ede un ant\'ec\'edent par le morphisme $\Os_{\E{n}{\As}}(U) \to \Os_{\E{n}{\As},x}$.

Signalons qu'\`a la fin du premier chapitre de~\cite{rouge}, V.~Berkovich propose une d\'efinition g\'en\'erale d'espace analytique sur~$\As$.\footnote{Dans le cas o\`u~$\As=\C$, cette d\'efinition redonne les espaces analytiques complexes usuels. Dans le cas o\`u~$k$ est un corps ultram\'etrique complet, en revanche, elle est plus restrictive que celle que V.~Berkovich propose par la suite. Toutefois, elle permet de retrouver les analytifi\'ees de vari\'et\'es alg\'ebriques et, plus g\'en\'eralement, les bons espaces sans bord.} Un espace localement annel\'e~$(V,\Os_{V})$ est appel\'e mod\`ele local d'un espace analytique sur~$\As$ s'il existe un entier~$n$, un ouvert~$U$ de~$\E{n}{\As}$ et un faisceau~$\Is$ d'id\'eaux de type fini de~$\Os_{U}$ tels que $(V,\Os_{V})$ soit isomorphe au support du faisceau~$\Os_{U}/\Is$, muni du faisceau~$\Os_{U}/\Is$. Un \textbf{espace analytique sur~$\As$} est un espace localement annel\'e qui poss\`ede un recouvrement par des ouverts isomorphes \`a un mod\`ele local.

\medskip

Nous allons maintenant rappeler divers notions et r\'esultats tir\'es de~\cite{asterisque}. Nous y renvoyons le lecteur int\'eress\'e par des pr\'ecisions suppl\'ementaires.

Soit~$n\in\N$. Soit~$V$ une partie compacte de~$\E{n}{\As}$. On note~$\Ks(V)$ l'anneau des fractions rationnelles sans p\^oles sur~$V$ et~$\Bs(V)$ son compl\'et\'e pour la norme uniforme~$\|.\|_{V}$ sur~$V$. Posons $\Bs(V)^\dag = \varinjlim_{W\supset V} \Bs(W)$, o\`u~$W$ d\'ecrit l'ensemble des voisinages compacts de~$V$ dans~$\E{n}{\As}$. Remarquons que, pour tout point~$x$ de~$\E{n}{\As}$, nous avons
\[\Bs(\{x\})^\dag =  \varinjlim_{W\ni x} \Bs(W) \simeq \Os_{\E{n}{\As},x}.\]

Soit~$x$ un point de~$\E{n}{\As}$. Nous dirons qu'un \'el\'ement~$f$ de~$\Os_{\E{n}{\As},x}$ est \mbox{\textbf{$\Bs$-d\'efini}} (resp. \textbf{$\Bs^\dag$-d\'efini}) sur un voisinage compact~$V$ de~$x$ s'il poss\`ede un ant\'ec\'edent par le morphisme $\Bs(V) \to \Os_{\E{n}{\As},x}$ (resp. $\Bs^\dag(V) \to \Os_{\E{n}{\As},x}$).

Rappelons quelques d\'efinitions tir\'ees de~\cite{asterisque}, \S 1.2. On dit qu'une partie compacte~$V$ de~$\E{n}{\As}$ est \textbf{rationnelle} s'il existe un entier~$p$, des polyn\^omes $P_{1},\ldots,P_{p},Q$ de $\As[T_{1},\ldots,T_{n}]$ ne s'annulant pas simultan\'ement sur~$V$ et des nombres r\'eels \mbox{$r_{1},\ldots,r_{p}>0$} tels que
\[V = \bigcap_{1\le i\le p} \left\{ x\in \E{n}{\As}\, \big|\, |P_{i}(x)|\le r_{i}\, |Q(x)|\right\}.\]
On dit qu'une partie compacte~$V$ de~$\E{n}{\As}$ est \textbf{spectralement convexe} si le morphisme naturel
\[\varphi : \Mc(\Bs(V)) \to \E{n}{\As}\]
induit un hom\'eomorphisme entre $\Mc(\Bs(V))$ et~$V$ et si le morphisme induit
\[\varphi^{-1}\big(\mathring{V}\big) \to \mathring{V}\]
est un isomorphisme d'espace annel\'es. 

\begin{prop}
Soit~$V$ un compact rationnel de~$\E{n}{\As}$. Alors~$V$ est spectralement convexe et, pour toute partie compacte~$U$ de~$V$, le morphisme naturel 
\[\Bs(U)\quad (\textrm{dans } \E{n}{\As}) \to \Bs(U)\quad (\textrm{dans } \Bs(V))\]
est un isomorphisme. 
\end{prop}

Exposons maintenant quelques r\'esultats du paragraphe~\S~2.1 de~\cite{asterisque} (en nous contentant de la dimension~$1$). Soit $t>0$. D\'efinissons l'alg\`ebre $\As\la |T|\le t\ra$ comme l'alg\`ebre constitu\'ee des s\'eries de la forme $\sum_{n\in\N} a_{n} T^n$, o\`u $(a_{n})_{n\ge 0}$ est une suite d'\'el\'ements de~$\As$ telle que la s\'erie $\sum_{n\in\N} \|a_{n}\| t^n$ converge. Cette alg\`ebre est compl\`ete pour la norme d\'efinie par 
\[\left\|\sum_{n\ge 0} a_{n} T^n\right\|_{t} = \sum_{n\ge 0} \|a_{n}\| t^n.\]

Soient $s,t\in\R$ v\'erifiant $0< s\le t$. D\'efinissons l'alg\`ebre $\As\la s\le |T|\le t\ra$ comme l'alg\`ebre constitu\'ee des s\'eries de la forme $\sum_{n\in\Z} a_{n} T^n$, o\`u $(a_{n})_{n\ge 0}$ est une suite d'\'el\'ements de~$\As$ telle que la famille $(\|a_{n}\| \max(s^n,t^n))_{n\in\Z}$ est sommable. Cette alg\`ebre est compl\`ete pour la norme d\'efinie par 
\[\left\|\sum_{n\in\Z} a_{n} T^n\right\|_{s,t} = \sum_{n\in\Z} \|a_{n}\| \max(s^n,t^n).\]
Nous prolongeons cette d\'efinition en posant $\As\la 0\le  |T|\le t\ra = \As\la |T|\le t\ra$ et \mbox{$\|.\|_{0,t} = \|.\|_{t}$}.

Posons~$X=\E{1}{\As}$ (avec variable~$T$), $B=\Mc(\As)$ et notons $\pi : X \to B$ le morphisme de projection. Pour toute partie~$V$ de~$B$ et tous $s,t\in\R$, d\'efinissons
\begin{align*}
\overset{\circ}{D}_{V}(t) &= \{x\in X\, |\, \pi(x)\in V, |T(x)|< t\},\\
\overline{D}_{V}(t) &= \{x\in X\, |\, \pi(x)\in V, |T(x)|\le t\},\\
\overset{\circ}{C}_{V}(s,t) &= \{x\in X\, |\, \pi(x)\in V, s < |T(x)|< t\},\\
\overline{C}_{V}(s,t) &= \{x\in X\, |\, \pi(x)\in V, s\le |T(x)|\le t\}.
\end{align*}
Lorsque~$V=B$, nous supprimerons l'indice dans la notation. Si $V$~est un singleton~$\{b\}$, nous omettrons les accolades.

Comme on s'y attend, les spectres des alg\`ebres $\As\la |T|\le t\ra$, pour $t>0$, et \mbox{$\As\la s\le |T|\le t\ra$}, pour $t\ge s >0$ ou $t>s=0$, sont canoniquement isomorphes \`a $\overline{D}(t)$ et $\overline{C}(s,t)$ respectivement. 

Indiquons maintenant une relation entre norme en tant que s\'erie et norme uniforme. Pour \'eviter de distinguer les cas entre disques et couronnes, introduisons la relation suivante sur~$\R_{+}$ : nous posons $s\prec t $ si $s<t$ ou $s=0$.

Rappelons \'egalement que l'anneau de Banach~$(\As,\|.\|)$ est dit \textbf{uniforme} si sa norme est \og power-multiplicative \fg, c'est-\`a-dire qu'elle commute \`a l'op\'eration d'\'el\'evation \`a une puissance enti\`ere ou, de mani\`ere \'equivalente, qu'elle co\"{\i}ncide avec la norme uniforme sur~$\Mc(\As)$. Il est souvent commode de supposer que l'anneau de Banach~$(\As,\|.\|)$ est uniforme et c'est une hypoth\`ese sous laquelle nous nous sommes constamment plac\'es dans~\cite{asterisque}. Remarquons qu'elle impose \`a l'anneau~$\As$ d'\^etre r\'eduit et permet d'identifier canoniquement~$\As[\bT]$ \`a un sous-anneau de~$\Os_{\E{n}{\As}}(\E{n}{\As})$ (\textit{cf.}~\cite{asterisque}, lemme~1.1.25). Cette hypoth\`ese montre son utilit\'e lorsque l'on cherche \`a d\'ecrire explicitement des anneaux de fonctions.

\begin{prop}\label{prop:comparaisonnormes}
Supposons que l'anneau de Banach~$(\As,\|.\|)$ est uniforme. Soient $s,u\in\R_{+}$ et $t,v \in \R_{+}^*$ tels que $s\prec u \le v <t$. Pour tout \'el\'ement~$f$ de $\As\la s\le |T|\le t\ra$, nous avons
\[\|f\|_{u,v} \le \left(\frac{s}{u-s} + \frac{t}{t-v}\right) \|f\|_{\overline{C}(s,t)},\]
avec la convention que $s/(u-s)=0$ si $s=0$.
\end{prop}
\begin{proof}
Il s'agit de majorer la norme en tant que s\'erie par la norme uniforme, autrement dit, de majorer les coefficients d'une s\'erie en fonction de sa norme uniforme sur un disque ou une couronne. On se ram\`ene au cas o\`u~$\As$ est un corps valu\'e en consid\'erant une fibre sur laquelle le maximum de la fonction est atteint. On utilise ensuite soit la description explicite de la norme sur les disques ou les couronnes, si le corps est ultram\'etrique, soit la formule de Cauchy, s'il est archim\'edien. Nous renvoyons \`a la proposition~2.1.3 de~\cite{asterisque} pour les d\'etails (et la g\'en\'eralisation en dimension sup\'erieure).
\end{proof}

Nous pouvons d\'eduire de ce r\'esultat une description de l'anneau~$\Bs$ d'une couronne compacte.

\begin{prop}\label{prop:Bcouronne}
Supposons que l'anneau de Banach~$(\As,\|.\|)$ est uniforme. Soient $s,t\in\R$ tels que $0\le s\le t$. Alors le morphisme naturel
\[\varinjlim_{s'\prec s\le t < t'} \As\la s'\le |T|\le t'\ra \to \Bs(\overline{C}(s,t))^\dag\]
est un isomorphisme.
\end{prop}
\begin{proof}
Il suffit de montrer que ce morphisme est surjectif. Soit~$f$ un \'el\'ement de $\Bs(\overline{C}(s,t))^\dag$. Il existe $s',t'\in\R$ v\'erifiant $s'\prec s\le t < t'$ tels que $f\in \Bs(\overline{C}(s',t'))$. Par d\'efinition, nous pouvons \'ecrire la fonction~$f$ comme limite uniforme de fractions $P_{n}/Q_{n}$, avec $P_{n},Q_{n}\in\As[T]$, les polyn\^omes~$Q_{n}$ ne s'annulant pas sur~$\overline{C}(s',t')$. 

Comme indiqu\'e pr\'ecedemment, cette couronne n'est autre que le spectre de l'alg\`ebre $\As\la s'\le |T|\le t'\ra$. D'apr\`es le corollaire~1.2.4 de~\cite{rouge}, les polyn\^omes~$Q_{n}$ sont inversibles dans cette alg\`ebre. Par cons\'equent, $(P_{n}/Q_{n})_{n\ge 0}$ est une suite d'\'el\'ements de $\As\la s'\le |T|\le t'\ra$ qui converge vers~$f$ pour la norme uniforme sur~$\overline{C}(s',t')$. La proposition pr\'ec\'edente assure qu'elle converge dans $\As\la s''\le |T|\le t''\ra$ pour la norme~$\|.\|_{s'',t''}$, quels que soient $s'',t''\in\R$ v\'erifiant $s'\prec s'' \prec s$ et $t < t'' < t'$. On en d\'eduit le r\'esultat voulu.
\end{proof}

Si~$W$ est une partie compacte et spectralement convexe de~$\E{n}{\As}$, nous avons muni l'anneau~$\Bs(W)$ de la norme uniforme sur $W = \Mc(\Bs(W))$. Par cons\'equent, dans ce cas, l'anneau de Banach~$\Bs(W)$ est toujours uniforme.

\begin{coro}
Soit~$V$ une partie compacte de~$B$ qui poss\`ede un syst\`eme fondamental de voisinages spectralement convexes. Soient $s,t\in\R$ tels que \mbox{$0\le s\le t$}. Alors le morphisme naturel
\[\varinjlim_{W\supset V,s'\prec s\le t < t'} \Bs(W)\la s'\le |T|\le t'\ra \to \Bs(\overline{C}_{V}(s,t))^\dag,\]
o\`u~$W$ d\'ecrit l'ensemble des voisinages compacts de~$V$ dans~$B$, est un isomorphisme.
\end{coro}

Remarquons que le r\'esultat vaut, en particulier, lorsque la partie~$V$ est un compact rationnel ou un point. 

\begin{coro}\label{cor:dvptpoint}
Soit~$b$ un point de~$B$. Notons~$0_{b}$ le point~$0$ de la fibre~$X_{b}$. Alors le morphisme naturel
\[\varinjlim_{V\ni b,t>0} \Bs(V)\la |T|\le t\ra \to \Os_{X,0_{b}},\]
o\`u~$V$ d\'ecrit l'ensemble des voisinages compacts de~$b$ dans~$B$, est un isomorphisme.
\end{coro}

\begin{coro}\label{cor:dvptdisque}
Soit~$b$ un point de~$B$. Soit~$r\ge 0$. Alors le morphisme naturel
\[\varinjlim_{V\ni b,t>r} \Bs(V)\la |T|\le t\ra \to \Os_{X}(\overline{D}_{b}(r)),\]
o\`u~$V$ d\'ecrit l'ensemble des voisinages compacts de~$b$ dans~$B$, est un isomorphisme.
\end{coro}
\begin{proof}
Seule la surjectivit\'e pose probl\`eme. Soit $f \in \Os_{X}(\overline{D}_{b}(r))$. Il existe un voisinage~$V$ de~$b$ et un nombre r\'eel $t>r$ tel que $f \in \Os_{X}(\overline{D}_{V}(t))$. On utilise le corollaire pr\'ec\'edent pour trouver un \'el\'ement~$g$ de $\Bs(W)\la |T|\le t_{0}\ra$, avec $W \subset V$ et $t_{0}>0$, dont le germe en~$0_{b}$ est~$f$. En utilisant la description explicite des sections sur $\overline{D}_{c}(t)$, pour tout $c\in W$, on montre que $g \in \Hs(c)\la |T|\le t\ra$. On en d\'eduit que $g \in \Bs(W)\la |T|\le t\ra$.
\end{proof}

En appliquant ce corollaire fibre \`a fibre et en identifiant les coefficients des d\'eveloppements, nous obtenons le r\'esultat suivant. Pour toute partie compacte~$W$ de~$B$ et tout nombre r\'eel~$t>0$, nous noterons $\Os(W)\la |T|\le t\ra$ l'alg\`ebre constitu\'ee des s\'eries de la forme  $\sum_{n\in\N} a_{n} T^n$, o\`u $(a_{n})_{n\ge 0}$ est une suite d'\'el\'ements de~$\Os(W)$ telle que la s\'erie $\sum_{n\in\N} \|a_{n}\|_{W} t^n$ converge.

\begin{coro}
Soit~$V$ une partie de~$B$. Soit~$r\ge 0$. Alors le morphisme naturel
\[\varinjlim_{W\supset V,t>r} \Os(W)\la |T|\le t\ra \to \Os_{X}(\overline{D}_{V}(r)),\]
o\`u~$W$ d\'ecrit l'ensemble des voisinages compacts de~$V$ dans~$B$, est un isomorphisme.
\end{coro}

Dans l'\'etude qui suit, il est superflu d'exiger que l'anneau de Banach~$(\As,\|.\|)$ soit uniforme. En cas de besoin, nous pourrions de toute fa\c{c}on nous y ramener par le lemme suivant, dont la d\'emonstration ne pr\'esente aucune difficult\'e.

\begin{lemm}\label{lem:separecomplete}
Notons~$\hat{\As}$ le s\'epar\'e compl\'et\'e de l'anneau~$\As$ pour la semi-norme spectrale~$\|.\|_{\infty}$ sur~$\Ms(\As)$. Alors le morphisme naturel
\[\Mc(\hat{\As},\|.\|_{\infty}) \to \Mc(\As,\|.\|)\]
est un isomorphisme d'espaces annel\'es.
\end{lemm}

\begin{center}{\rule{4cm}{0.2mm}}\end{center}

\medskip

\textsl{Pour toute la suite du texte, nous fixons un anneau de Banach $(\As,\|.\|)$. Nous notons $B=\Mc(\As)$, $X=\E{1}{\As}$ et $\pi : X \to B$ le morphisme de projection. Pour tout $n\in\N$, nous notons $X_{n} = \E{n}{\As}$.}

\section{Th\'eor\`eme de Weierstra{\ss} global}

Dans cette section, nous reprenons, en les pr\'ecisant, les r\'esultats de~\cite{asterisque}, \S 5.2. Les preuves ne demandant que des modifications mineures, nous ne nous y attarderons pas. Insistons sur le fait que nous ne requ\'erons aucune hypoth\`ese d'uniformit\'e sur l'anneau de Banach~$(\As,\|.\|)$.

\medskip

Soient $d\in\N$ et $G \in \As[T]$ un polyn\^ome unitaire de degr\'e~$d$. 

\begin{theo}[Th\'eor\`eme de division de Weierstra{\ss} global]\label{theo:Wglobal}
Il existe un nombre r\'eel~$v>0$ v\'erifiant la propri\'et\'e suivante : pour toute $\As$-alg\`ebre de Banach~$\As'$ telle que le morphisme structural $\As\to \As'$ diminue les normes, pour tout nombre r\'eel~$w\ge v$ et tout \'el\'ement~$F$ de $\As'\of{\la}{|T|\le w}{\ra}$, il existe un unique couple $(Q,R) \in \As'\of{\la}{|T|\le w}{\ra}^2$ tel que 
\begin{enumerate}[i)]
\item $F=QG+R$;
\item $R$ soit un polyn\^ome de degr\'e strictement inf\'erieur \`a~$d$.
\end{enumerate}
En outre, il existe une constante~$C>0$ ind\'ependante de~$w$, $\As'$ et~$F$, telle que l'on ait les in\'egalit\'es
\[\left\{{\renewcommand{\arraystretch}{1.3}\begin{array}{rcl}
\|Q\|_{w} &\le& C\|F\|_{w}\ ;\\
\|R\|_{w} &\le& C\|F\|_{w}.
\end{array}}\right.\]
\end{theo}

Pour la suite de cette section, fixons une $\As$-alg\`ebre de Banach~$\As'$ telle que le morphisme structural $\As\to \As'$ diminue les normes. Nous munissons l'alg\`ebre quotient $\As'[T]/(G(T))$ de la semi-norme r\'esiduelle $\|.\|_{\As',w,\textrm{r\'es}}$ induite par la norme~$\|.\|_{\As',w}$ sur~$\As'[T]$. Par
d\'efinition, quel que soit~$F$ dans~$\As'[T]/(G(T))$, nous avons
\[\|F\|_{\As',w,\textrm{r\'es}}= \inf\left\{ \max_{0\le i\le e} \|a_{i}\| w^i,\ \sum_{i=0}^e a_{i}\, T^i=F \mod G, e\in\N \right\}.\]

Si~$\As'$ est une $\As$-alg\`ebre de la forme~$\Bs(U)$, o\`u~$U$ est une partie compacte de~$\Mc(\As)$, munie de la norme spectrale~$\|.\|_{U}$, nous noterons $\|.\|_{U,w,\textrm{r\'es}}$ la norme $\|.\|_{\Bs(U),w,\textrm{r\'es}}$.

\begin{coro}\label{cor:resnorme}
Pour tout nombre r\'eel~\mbox{$w\ge v$}, les propri\'et\'es suivantes sont satisfaites : 
\begin{enumerate}[i)]
\item la semi-norme~$\|.\|_{\As',w,\textrm{r\'es}}$ d\'efinie sur le quotient~$\As'[T]/(G(T))$ est une norme ;
\item l'anneau~$\As'[T]/(G(T))$ est complet pour la norme~$\|.\|_{\As',w,\textrm{r\'es}}$.
\end{enumerate}
\end{coro}

Puisque le polyn\^ome~$G$ est unitaire et de degr\'e~$d$, l'application
\[n : \begin{array}{ccc}
\As'^d & \to & \As'[T]/(G(T))\\
(a_{0},\ldots,a_{d-1}) & \mapsto & \disp \sum_{i=0}^{d-1} a_{i}\, T^i 
\end{array}\]
est bijective. Nous noterons~$\|.\|_{\As'}$ la norme d\'efinie sur~$\As'^d$ en prenant le maximum des normes des coordonn\'ees. Nous pouvons alors d\'efinir une norme~$\|.\|_{\As',\textrm{div}}$ sur~$\As'[T]/(G(T))$ par la formule
\[\|.\|_{\As',\textrm{div}} = \|n^{-1}(.)\|_{\As'}.\]

Si~$\As'$ est une $\As$-alg\`ebre de la forme~$\Bs(U)$, o\`u~$U$ est une partie compacte de~$\Mc(\As)$, munie de la norme spectrale~$\|.\|_{U}$, nous noterons $\|.\|_{U,\textrm{div}}$ la norme $\|.\|_{\Bs(U),\textrm{div}}$.

\begin{coro}\label{cor:eqnormedivres}
Pour tout nombre r\'eel~\mbox{$w\ge v$} les normes~$\|.\|_{\As',\textrm{div}}$ et~$\|.\|_{\As',w,\textrm{r\'es}}$ d\'efinies sur~$\As'[T]/(G(T))$ sont \'equi\-va\-lentes. Plus pr\'ecis\'ement, pour tout \'el\'ement~$F$ de $\As'[T]/(G(T))$, nous avons
\[ \|F\|_{\As',w,\textrm{r\'es}} \le \|F\|_{\As',\textrm{div}} \le C\, \|F\|_{\As',w,\textrm{r\'es}}.\]
\end{coro}

\section{Morphismes finis}\label{section:morphismesfinis}

Soit $G\in\As[T]$ un polyn\^ome unitaire non constant. Consid\'erons le morphisme d'anneaux de Banach $\As \to \As[T]/(G(T))$ et le morphisme d'espaces annel\'es $\varphi : Z_{G} \to B$, o\`u~$Z_{G}$ d\'esigne le ferm\'e de Zariski de~$X$ d\'efini par~\mbox{$G=0$}, qu'il induit. C'est un morphisme fini dont nous allons \'etablir quelques propri\'et\'es. Nous suivons ici~\cite{asterisque}, \S~5.3.

Dans~\cite{asterisque}, d\'efinition~5.2.5, nous avons introduit une condition appel\'ee~$(R_{G})$. Nous la rempla\c{c}ons ici par la condition suivante. 

\begin{defi}
Soit~$U$ une partie compacte de~$B$. On dit que~$U$ satisfait la {\bf condition~$\boldsymbol{(N_{G})}$} si elle est spectralement convexe et s'il existe un nombre r\'eel~$v>0$ tel que, pour tout~$w\ge v$, la semi-norme~$\|.\|_{U,w,\textrm{r\'es}}$ soit \'equivalente \`a la norme spectrale sur~$\Bs(U)[T]/(G(T))$. 
\end{defi}

Dans~\cite{asterisque}, \S~5.3, nous ne nous sommes en r\'ealit\'e servi de~$(R_{G})$ que pour d\'emontrer~$(N_{G})$ (\textit{cf.}~proposition~5.2.7) et cette derni\`ere condition suffit dans tous les raisonnements.

\begin{rema}\label{rem:NG}
La condition~$(N_{G})$ \'enonce une \'equivalence entre deux normes sur un module de type fini sur l'anneau de Banach~$\Bs(U)$. L'on peut y penser comme une g\'en\'eralisation du r\'esultat classique d'\'equivalence des normes pour les espaces vectoriels de dimension finie sur un corps valu\'e complet~$K$ de valuation non triviale (que l'on ne peut appliquer que si la norme spectrale est bien une norme sur $K[T]/(G(T))$, autrement dit, lorsque le polyn\^ome~$G$ est sans facteurs multiples).

L'on peut \'egalement consid\'erer le cas o\`u l'alg\`ebre~$\Bs(U)$ est une alg\`ebre affino\"{\i}de. Une alg\`ebre finie sur une alg\`ebre affino\"{\i}de \'etant elle-m\^eme affino\"{\i}de, on retrouve alors un r\'esultat connu, dans le cas des alg\`ebres r\'eduites (\textit{cf.}~\cite{rouge}, proposition~2.1.4 (ii) et~\cite{BGR}, th\'eor\`eme~6.2.4/1).
\end{rema}

\begin{rema}\label{rem:degre1}
Remarquons, d\`es \`a pr\'esent, que la condition~$(N_{G})$ est satisfaite si le polyn\^ome~$G$ est de degr\'e~$1$, puisque l'on dispose alors d'un isomorphisme admissible $\Bs(U) \simeq \Bs(U)[T]/(G(T))$. Nous donnerons \`a la section~\ref{sec:N} des crit\`eres plus g\'en\'eraux.
\end{rema}

Nous reprenons ici la proposition~5.3.3 de~\cite{asterisque}.

\begin{prop}\label{prop:Bfini}
Soit~$U$ une partie compacte de~$B$ qui satisfait la condition~$(N_{G})$. Les normes~$\|.\|_{U,\textrm{div}}$, $\|.\|_{U,w,\textrm{r\'es}}$, pour $w\ge v$, et la norme spectrale sur $\Bs(U)[T]/(G(T))$ sont alors \'equivalentes. De plus, nous avons un isomorphisme admissible naturel
\[\Bs(U)[T]/(G(T)) \xrightarrow[]{\sim} \Bs(\varphi^{-1}(U)),\]
o\`u~$\Bs(\varphi^{-1}(U))$ est calcul\'e en consid\'erant~$\varphi^{-1}(U)$ comme partie de~$X$.
\end{prop}

\begin{rema}
Dans la preuve de la proposition~5.3.3 de~\cite{asterisque}, on consid\`ere en r\'ealit\'e~$\varphi^{-1}(U)$ comme partie de~$Z_{G} \simeq \Mc(\As[T]/(G(T)))$. Cela ne change rien \`a la preuve.
\end{rema}

Dans~\cite{asterisque}, d\'efinition~5.3.5, nous avons introduit une condition appel\'ee~$(I_{G})$. Nous la modifions ici de la fa\c{c}on suivante.

\begin{defi}
Notons $G(b)(T) = \prod_{i=1}^r h_{i}(T)^{n_{i}}$ la d\'ecomposition en produit de polyn\^omes irr\'eductibles et unitaires de~$G(b)(T)$ dans~$\Hs(b)[T]$. On dit que le point~$b$ de~$B$ satisfait la {\bf condition~$\boldsymbol{(D_{G})}$} s'il existe des polyn\^omes unitaires $H_{1},\ldots,H_{r}$ \`a coefficients dans~$\Os_{B,b}$ tels que
\begin{enumerate}[i)]
\item $G = \prod_{i=1}^r H_{i}$ dans $\Os_{B,b}[T]$ ;
\item quel que soit $i\in \cn{1}{r}$, $H_{i}(b) = h_{i}^{n_{i}}$.
\end{enumerate}
\end{defi}

Comme pr\'ec\'edemment, signalons que nous n'avons utilis\'e la condition~$(I_{G})$ dans~\cite{asterisque}, \S~5.3 que pour d\'emontrer~$(D_{G})$ (\textit{cf.} la discussion qui suit la d\'efinition~5.3.5) et que cette derni\`ere condition suffit dans tous les raisonnements. 

\begin{rema}
G\'eom\'etriquement, la condition~$(D_{G})$ signifie que tous les points du ferm\'e de Zariski~$Z_{G}$ d\'efini par \mbox{$G=0$} au-dessus de~$b$ sont d\'efinis au voisinage de~$b$.

Nous montrerons \`a la section~\ref{sec:D} que la condition~$(D_{G})$ est toujours v\'erifi\'ee.
\end{rema}

\'Enon\c{c}ons maintenant l'analogue du th\'eor\`eme~5.3.8 de~\cite{asterisque}. La d\'emonstration \'etant en tout point identique, nous ne la reprendrons pas.

\begin{theo}\label{theo:morphismefinipoint}
Soit~$b$ un point de~$B$. Supposons que le point~$b$ v\'erifie la condition~$(D_{G})$ et poss\`ede un syst\`eme fondamental de voisinages compacts qui satisfont la condition~$(N_{G})$. Alors le morphisme
\[\alpha_{b} : \begin{array}{ccc}
\Os_{B,b}^p & \to & (\varphi_{*}\Os_{Z_{G}})_{b}\\
(a_{0},\ldots,a_{p-1}) & \mapsto & \disp \sum_{i=0}^p a_{i}\, T^i
\end{array}\]
est un isomorphisme de~$\Os_{B,b}$-modules.
\end{theo}

\section{Hens\'elianit\'e et condition~$(D)$}\label{sec:D}

Nous nous int\'eressons ici \`a des propri\'et\'es d'hens\'elianit\'e dans les espaces analytiques sur~$\As$. Commen\c{c}ons par rappeler un r\'esultat (\textit{cf.}~\cite{asterisque}, corollaire~2.5.2 ; l'hypoth\`ese d'uniformit\'e sur l'anneau de Banach~$\As$ qui y figure est superflue, ce dont on se convainc \`a l'aide du lemme~\ref{lem:separecomplete}).

\begin{theo}
Soient~$(Z,\Os_{Z})$ un espace analytique sur~$\As$ et~$z$ un point de~$Z$. L'anneau local~$\Os_{Z,z}$ est hens\'elien.
\end{theo}

Une l\`eg\`ere modification de la preuve conduit au r\'esultat suivant.

\begin{theo}
Soit~$b$ un point de~$B=\Mc(\As)$. Le corps r\'esiduel~$\kappa(b)$ est un corps hens\'elien.
\end{theo}
\begin{proof}
Si le point~$b$ est archim\'edien, le corps~$\kappa(b)$ est \'egal \`a~$\R$ ou~$\C$ ; il est donc hens\'elien.

Supposons maintenant que le point~$b$ est ultram\'etrique. Il suffit de montrer que l'anneau de valuation~$\kappa(b)^o$ est hens\'elien. Soient~$P$ un polyn\^ome unitaire \`a coefficients dans~$\kappa(b)^o$ et~$f$ un \'el\'ement de~$\kappa(b)^o$ tel que $|P(f)(b)|<1$ et $|P'(f)(b)|=1$. 

Consid\'erons un voisinage compact~$V$ de~$b$ dans~$B$ sur lequel le polyn\^ome~$P$ se rel\`eve en un polyn\^ome \`a coefficients dans~$\Bs(V)$ unitaire et l'\'el\'ement~$f$ se rel\`eve en un \'el\'ement de~$\Bs(V)$. Nous choisissons de tels relev\'es et les notons identiquement.

Notons $M=\|f\|_{V}$ et $N=\|P(f)\|_{V}$. Quitte \`a restreindre~$V$, nous pouvons supposer que $N<1$ et que~$P'(f)$ est inversible sur~$V$. Notons $m=\|P'(f)^{-1}\|_{V}$.

Il existe un polyn\^ome~$Q[T_{1},T_{2}]$ \`a coefficients dans~$\Bs(V)$ tel que, pour tous~$u$ et~$v$ dans~$\Bs(V)$, on ait
\[P(u+v) = P(u) + P'(u)v + v^2Q(u,v).\]
Notons~$C$ le maximum des normes des coefficients du polyn\^ome~$Q$.

Il existe un \'el\'ement~$\lambda$ de l'intervalle~$\of{[}{0,1}{]}$ tel que $\|2\|_{V} \le 2^\lambda$. D'apr\`es le th\'eor\`eme d'Ostrowski, pour tout entier~$n$, nous avons alors $\|n\|_{V}\le n^\lambda$. Le lemme~1.3.2 de~\cite{asterisque} permet d'en d\'eduire que
\[\forall u,v\in\Bs(V), \forall c\in V,\ |(u+v)(c)| \le 2^\lambda\max(|u(c)|,|v(c)|).\]
Notons $r$ le nombre de coefficients du polyn\^ome~$Q$ et~$d$ son degr\'e total. L'in\'egalit\'e pr\'ec\'edente entra\^{\i}ne
\[\forall u,v\in\Bs(V), \forall c\in V,\ |Q(u+v)(c)| \le 2^{\lambda r} C \max(1,|u(c)|^d,|v(c)|^d).\]

Pour tout~$\eps>0$, quitte \`a r\'etr\'ecir le voisinage~$V$, nous pouvons rendre les quantit\'es~$M$, $m$, $C$ et~$2^\lambda$ inf\'erieures \`a~$1+\eps$. Au cours de cette op\'eration, le nombre r\'eel~$N$ ne peut que diminuer. Fixons un nombre r\'eel~$K$ dans l'intervalle $\of{]}{N,1}{[}$. Nous pouvons choisir le voisinage~$V$ de fa\c{c}on que l'on ait $Nm<1$ et $Nm^2\, 2^{\lambda r} C \max(1,M^d) \le K$. 

Soit~$g$ un \'el\'ement de~$\Bs(V)$. Nous avons
\[P(f+P(f)g) = P(f)P'(f) \left( \frac{1}{P'(f)} + g + g^2\, \frac{P(f)}{P'(f)}\, Q(f,P(f)g) \right).\]
Posons
\[R(g) = g^2\, \frac{P(f)}{P'(f)}\, Q(f,P(f)g).\]

Si $\|g\|_{V} \le m$, nous avons, pour tout point~$c$ de~$V$
\[\renewcommand{\arraystretch}{1.5}{\begin{array}{rcl}
|R(g)(c)| & \le & |g(c)|^2\, Nm\, 2^{\lambda r} C \max(1,|f(c)|^d,|P(f)(c)|^d |g(c)|^d)\\
&\le&  |g(c)|\, Nm^2\, 2^{\lambda r} C \max(1,M^d,N^d m^d)\\
&\le& K|g(c)|
\end{array}}\]
et, en particulier,
\[\|R(g)\|_{V} \le K \|g\|_{V}.\]

En utilisant cette derni\`ere in\'egalit\'e et le fait que~$\Bs(V)$ est complet, on montre que la s\'erie $\sum_{n\ge 0} R^{\circ n}(-P'(f)^{-1})$ (o\`u $R^{\circ n}$ d\'esigne la puissance $n^\textrm{\`eme}$ pour la composition) converge vers un \'el\'ement~$s$ de~$\Bs(V)$. Cet \'el\'ement satisfait l'\'egalit\'e $s-R(s)=-P'(f)^{-1}$ et donc $P(f+P(f)s) = 0$.

Calculons maintenant au point~$b$ qui, rappelons-le, est suppos\'e ultram\'etrique. D'apr\`es les in\'egalit\'es qui pr\'ec\`edent, nous avons $|s(b)| = |P'(f)^{-1}(b)|=1$, d'o\`u $|P(f)s(b)|<1$ et $|(f+P(f)s)(b)|\le 1$. L'\'el\'ement $h = f+P(f)s$ de~$\Bs(V)$ d\'efini donc un \'el\'ement de~$\kappa(b)^o$ qui co\"{\i}ncide avec~$f$ modulo~$\kappa(b)^{oo}$ et annule le polyn\^ome~$P$.
\end{proof}

\begin{coro}
Soient~$(Z,\Os_{Z})$ un espace analytique sur~$\As$ et~$z$ un point de~$Z$. Le corps r\'esiduel~$\kappa(z)$ est un corps hens\'elien.
\end{coro}

\begin{coro}\label{cor:impliqueDG}
Soient~$b$ un point de~$B=\Mc(\As)$ et~$G(T)$ un polyn\^ome unitaire non constant \`a coefficients dans~$\Os_{B,b}$. Alors le point~$b$ satisfait la condition~$(D_{G})$.
\end{coro}
\begin{proof}
Notons~$e$ l'exposant caract\'eristique du corps~$\kappa(b)$. \'Ecrivons la d\'ecomposition de~$G(T)$ en produit de facteurs irr\'eductibles dans~$\kappa(b)[T]$ sous la forme
\[G(T) = \prod_{i=1}^r G_{i}(T)^{n_{i}},\]
o\`u les~$G_{i}$ sont des polyn\^omes unitaires et irr\'eductibles deux \`a deux distincts et les~$n_{i}$ des entiers.

Puisque l'anneau~$\Os_{B,b}$ est hens\'elien, il existe des polyn\^omes $H_{1},\dots,H_{r}$ unitaires \`a coefficients dans~$\Os_{B,b}$ tels que
\begin{enumerate}[i)]
\item $\disp G(T) = \prod_{i=1}^r H_{i}(T)$ dans~$\Os_{B,b}[T]$ ;
\item pour tout~$i$, $H_{i}(T) = G_{i}(T)^{n_{i}}$ dans~$\kappa(b)[T]$.
\end{enumerate}

Soit $i\in\cn{1}{r}$. Il existe un polyn\^ome irr\'eductible et s\'eparable~$g_{i}$ \`a coefficients dans~$\kappa(b)$ et un entier~$a_{i}$ tel que $G_{i}(T) = g_{i}(T^{e^{a_{i}}})$. Puisque le corps~$\kappa(b)$ est hens\'elien, le polyn\^ome~$g_{i}(T)$ est encore irr\'eductible dans~$\Hs(b)[T]$ (\textit{cf.}~\cite{algcom56}, VI, \S~8, exercices~14a et~12b ou ~\cite{bleu}, proposition~2.4.1). Par cons\'equent, il existe un polyn\^ome~$h_{i}$ \`a coefficients dans~$\Hs(b)$ et un entier~$m_{i}\ge 1$ tels que $g_{i}(T^{e^{a_{i}}}) = h_{i}(T)^{m_{i}}$ dans $\Hs(b)[T]$. 

Finalement, l'\'ecriture
\[G(T) = \prod_{i=1}^r h_{i}(T)^{m_{i} n_{i}}\]
est la d\'ecomposition en produit de facteurs irr\'eductibles du polyn\^ome~$G(T)$ dans~$\Hs(b)[T]$. On en d\'eduit le r\'esultat voulu.
\end{proof}

\section{Condition~$(N)$}\label{sec:N}

Soient $d\in\N^*$ et $G \in \As[T]$ un polyn\^ome unitaire de degr\'e~$d$. Nous noterons 
\[G = \sum_{k=0}^{d} g_{k}T^k = T^d + \sum_{k=0}^{d-1} g_{k}T^k.\]

Dans cette section, nous pr\'esentons diff\'erents r\'esultats permettant d'assurer que la condition~$(N_{G})$ est satisfaite.

\subsection{Condition de finitude du bord analytique}

Nous commen\c{c}ons par \'enoncer un crit\`ere particuli\`erement adapt\'e aux espaces ultram\'etriques. On y retrouve l'hypoth\`ese sur le polyn\^ome~$G$ mentionn\'ee \`a la remarque~\ref{rem:NG}.

\begin{defi}
Soit $n\in\N$. Soit~$V$ une partie compacte et spectralement convexe de~$\E{n}{\As}$. On dit qu'une partie ferm\'ee~$\Gamma$ de~$V$ est un \textbf{bord analytique} de~$V$ si elle v\'erifie la condition suivante :
\[\forall f\in\Bs(V),\, \|f\|_{V} = \|f\|_{\Gamma}.\]

On appelle \textbf{bord de Shilov} de~$V$ le plus petit bord analytique, pour la relation d'inclusion, de~$V$, s'il existe.
\end{defi}

\begin{prop}\label{prop:normeresuniformegeneral}
Soit~$U$ une partie compacte et spectralement convexe de~$B$. Supposons qu'il existe des parties $V_{1},\dots,V_{r}$ de~$U$ telles que
\begin{enumerate}[i)]
\item pour tout $i\in\cn{1}{r}$, $V_{i}$ est compacte et satisfait la condition~$(N_{G})$ ;
\item la r\'eunion $\disp \Gamma_{U} = \bigcup_{1 \le i\le r} V_{i}$ est un bord analytique de~$U$.
\end{enumerate}
Alors~$U$ satisfait \'egalement  la condition~$(N_{G})$. 

Notons $\varphi : \Mc(\Bs(U)[T]/(G(T))) \to U$ le morphisme de projection. Alors la partie $\Gamma_{U,G} = \varphi^{-1}(\Gamma_{U})$ est un bord analytique de $\Mc(\Bs(U)[T]/(G(T)))$.
\end{prop}
\begin{proof}
Le corollaire~\ref{cor:resnorme} assure qu'il existe un nombre r\'eel~$v$ tel que, pour tout~$w\ge v$, la semi-norme~$\|.\|_{U,w,\textrm{r\'es}}$ soit une norme sur~$\Bs(U)[T]/(G(T))$ et qu'elle soit \'equivalente \`a la norme~$\|.\|_{U,\textrm{div}}$. Notons~$\|.\|_{U,\infty}$ la norme spectrale associ\'ee. C'est la norme uniforme sur 
\[\left\{\left.x\in\E{1}{\As}\, \right|\, \pi(x)\in U,\, G(x)=0\right\}.\]
Pour montrer que les deux normes sont \'equivalentes, et donc montrer que la condition~$(R_{G})$ est satisfaite, il suffit de montrer qu'il existe une constante~$D\in\R$ telle que, pour tout \'el\'ement~$F$ de~$\Bs(U)[T]/(G(T))$, nous avons
\[ \|F\|_{U,\textrm{div}} \le D\, \|F\|_{U,\infty}. \]

Soit $i\in\cn{1}{r}$. Les normes~$\|.\|_{i,\textrm{div}}$ et~$\|.\|_{i,\infty}$ sur $\Bs(V_{i})[T]/(G(T))$ sont \'equivalentes : il existe un nombre r\'eel~$D_{i}>0$ tel que, pour tout \'el\'ement~$f$ de $\Bs(V_{i})[T]/(G(T))$, on ait
\[ \|f\|_{i,\textrm{div}} \le D_{i}\, \|f\|_{i,\infty}. \]

Soit~$F$ un \'el\'ement de~$\Bs(U)[T]/(G(T))$. Puisque le polyn\^ome~$G$ est unitaire et de degr\'e~$p$, l'\'el\'ement~$F$ poss\`ede un unique repr\'esentant dans~$\Bs(U)[T]$ de la forme
\[F_{0}(T) = \sum_{k=0}^{p-1} a_{k}\, T^k,\]
avec~$a_{0},\ldots,a_{p-1}\in\Bs(U)$. Par d\'efinition, nous avons
\[ \|F\|_{U,\textrm{div}} = \max_{0\le k\le p-1} (\|a_{k}\|_{U}),\]
o\`u~$\|.\|_{U}$ d\'esigne la norme spectrale sur~$\Bs(U)$.

Il existe $j\in\cn{0}{p-1}$ et $i\in\cn{1}{r}$ tels que $\|F\|_{U,\textrm{div}} = \|a_{j}\|_{V_{i}}$, o\`u~$\|.\|_{V_{i}}$ d\'esigne la norme spectrale sur~$\Bs(V_{i})$. Nous avons alors
\[ \|F\|_{U,\textrm{div}} \le D_{i}\, \|F\|_{i,\infty} \le \max_{1\le i\le r} (D_{i})\, \|F\|_{U,\infty}. \]
Ceci d\'emontre la premi\`ere partie du r\'esultat.

Remarquons maintenant que
\[ \|F\|_{U,\infty} \le \|F\|_{U,\textrm{div}} = \max_{1\le i\le r} (\|F\|_{i,\textrm{div}}) \le \max_{1\le i\le r} (D_{i})\, \max_{1\le i\le r} (\|F\|_{i,\infty}),\]
d'o\`u l'on tire
\[\|F\|_{U,\infty} \le \max_{1\le i\le r} (D_{i})\, \max_{x\in \Gamma_{U,G}} (|F(x)|).\]
En \'ecrivant la m\^eme in\'egalit\'e pour~$F^n$, avec $n\in \N^*$, puis en prenant la racine~$n^\textrm{\`eme}$ et en faisant tendre~$n$ vers l'infini, on montre que
\[\|F\|_{U,\infty} \le \max_{x\in \Gamma_{U,G}} (|F(x)|).\]
Par cons\'equent, la partie~$\Gamma_{U,G}$ est un bord analytique de $\Bs(U)[T]/(G(T))$. 
\end{proof}

\begin{coro}\label{cor:normeresuniforme}
Soit~$U$ une partie compacte et spectralement convexe de~$B$. Supposons qu'elle poss\`ede un bord analytique~$\Gamma_{U}$ fini tel que, pour tout point~$\gamma$ de~$\Gamma_{U}$, le polyn\^ome~$G(\gamma)(T)$ soit sans facteurs multiples. Alors elle satisfait la condition~$(N_{G})$. 

En outre, avec les notations de la proposition~\ref{prop:normeresuniformegeneral}, la partie~$\Gamma_{U,G}$ est un bord analytique fini de $\Mc(\Bs(U)[T]/(G(T)))$.
\end{coro}

Introduisons maintenant des notations. Nous noterons~$B_{\textrm{um}}$ la partie ultram\'etrique de~$B$, c'est-\`a-dire l'ensemble des points~$b$ de~$B$ en lesquels la valeur absolue canonique sur le corps~$\Hs(b)$ est ultram\'etrique. Soit~$V$ une partie compacte de~$B_{\textrm{um}}$. Pour tout nombre r\'eel~$r>0$, d\'efinissons l'alg\`ebre $\Bs(V)\{|T|\le r\}$ comme l'alg\`ebre constitu\'ee des s\'eries \`a coefficients dans~$\Bs(V)$ de la forme 
\[\sum_{n\in\N} a_{n} T^n,\] 
telles que la s\'erie $\sum_{n\in\N} \|a_{n}\| r^n$ converge. Cette alg\`ebre est compl\`ete pour la norme d\'efinie par 
\[\left\|\sum_{n\in\N} a_{n} T^n\right\|_{V,r,\textrm{um}} = \max_{n\in\N} (\|a_{n}\|_{V}\, r^n).\]

Pour tous $r,s\in\R$ v\'erifiant $0<r\le s$, d\'efinissons l'alg\`ebre \mbox{$\Bs(V)\{r \le |T|\le s\}$} comme l'alg\`ebre constitu\'ee des s\'eries \`a coefficients dans~$\Bs(V)$ de la forme
\[\sum_{n\in\Z} a_{n}\, T^n\]
telles que la famille~$(\|a_{n}\|_{V}\, \max(r^n,s^n))_{n\in\Z}$ soit sommable. Cette alg\`ebre est compl\`ete pour la norme d\'efinie par
\[\left\|\sum_{n\in\Z} a_{n}\, T^n\right\|_{V,r,s,\textrm{um}} = \max_{n\in\Z} (\|a_{n}\|_{V}\, \max(r^n,s^n)).\]
Nous prolongeons cette d\'efinition en posant \mbox{$\Bs(V)\{0\le  |T|\le s\} = \Bs(V)\{|T|\le s\}$} et \mbox{$\|.\|_{V,0,s,\textrm{um}} = \|.\|_{V,s,\textrm{um}}$}.

\begin{lemm}\label{lem:Bsum}
Soit~$V$ une partie compacte et spectralement convexe de~$B_{\textrm{um}}$. Soient~$r,s\in \R$ v\'erifiant $0 = r < s$ ou $0<r\le s$. Le morphisme 
\[\As[T] \to \Bs(\overline{C}_{V}(r,s))\] 
induit un isomorphisme d'alg\`ebres norm\'ees
\[\Bs(V)\{r\le |T|\le s\} \xrightarrow[]{\sim} \Bs(\overline{C}_{V}(r,s)).\]
\end{lemm}
\begin{proof}
On reprend le raisonnement de la preuve du lemme~5.6.3 de~\cite{asterisque}.
\end{proof}

Dans la section~\ref{section:Weierstrass}, nous utiliserons la condition~$(N_{G})$ au-dessus d'une partie de $X=\E{1}{\As}$ et pour un polyn\^ome~$G$ de la forme~$P(S)-T$, o\`u~$T$ d\'esigne la variable sur~$\E{1}{\As}$. Nous allons donc maintenant concentrer nos efforts sur cette situation. 

Introduisons quelques notations pour les lemniscates. Pour toute partie~$W$ de~$B$, tout polyn\^ome $P\in\Bs(W)[T]$ et tous nombres r\'eels~$r$ et~$s$, d\'efinissons
\[\overline{D}_{W}(P;s) = \left\{x\in X_{W}\, \big|\,  |P(T)(x)| \le s\right\}\]
et
\[\overline{C}_{W}(P;r,s) = \left\{x\in X_{W}\, \big|\, r\le |P(T)(x)| \le s\right\}.\]

\begin{prop}\label{prop:couronneum}
Soit~$U$ une partie compacte et spectralement convexe de~$B_{\textrm{um}}$. Supposons qu'elle poss\`ede un bord analytique fini~$\Gamma_{U}$. Soit~$P(S)$ un polyn\^ome unitaire non constant \`a coefficients dans~$\Bs(U)$. Soient $r,s\in \R$ v\'erifiant \mbox{$0=r<s$} ou $0<r\le s$. Alors la couronne $\overline{C}_{U}(r,s)$ satisfait la condition~$(N_{P(S)-T})$.

En outre, la couronne $\overline{C}_{U}(r,s)$ poss\`ede un bord analytique fini, ainsi que la lemniscate $\overline{C}_{U}(P;r,s)$
\end{prop}
\begin{proof}
Soit~$\gamma\in\Gamma_{U}$. L'alg\`ebre 
\[\Bs(\overline{C}_{\gamma}(r,s))[S]/(Q(S)-T) = \Hs(\gamma)\{s^{-1}T, r T^{-1}\}[S]/(Q(S)-T)\] 
n'est autre que l'alg\`ebre affino\"{\i}de de la lemniscate $\overline{C}_{\gamma}(P,r,s)$. C'est donc une alg\`ebre~$\Hs(\gamma)$-affino\"{\i}de r\'eduite. D'apr\`es la proposition~2.1.4~(ii) de~\cite{rouge}, sa norme r\'esiduelle est \'equivalente \`a sa norme spectrale, autrement dit la couronne $\overline{C}_{\gamma}(t,s)$ satisfait la condition~$(N_{P(S)-T})$. 

D'apr\`es le lemme~\ref{lem:Bsum}, la partie $\bigcup_{\gamma\in\Gamma_{U}} \overline{C}_{\gamma}(r,s)$ est un bord analytique de~$\overline{C}_{U}(r,s)$  et l'on conclut par la proposition~\ref{prop:normeresuniformegeneral}.

Pour~$\gamma\in\Gamma_{U}$, notons~$\Gamma_{\gamma}$ le bord de Shilov de la couronne $\overline{C}_{\gamma}(t,s)$. L'ensemble fini $\bigcup_{\gamma\in\Gamma_{U}} \Gamma_{\gamma}$ est alors un bord analytique de~$\Gamma_{U}$. L'assertion finale d\'ecoule de la proposition~\ref{prop:normeresuniformegeneral}.
\end{proof}

\begin{coro}\label{cor:lemniscateum}
Soit~$U$ une partie compacte et spectralement convexe de~$B_{\textrm{um}}$. Supposons qu'elle poss\`ede un bord analytique fini~$\Gamma_{U}$. Soient~$P(S)$ et~$Q(T)$ des polyn\^omes unitaires et non constants \`a coefficients dans~$\Bs(U)$. Soient $r,s\in \R$ v\'erifiant \mbox{$0=r<s$} ou $0<r\le s$. Alors la lemniscate $\overline{C}_{U}(Q;r,s)$ satisfait la condition~$(N_{P(S)-T})$.
\end{coro}
\begin{proof}
D'apr\`es la preuve de la proposition pr\'ec\'edente, la partie $\bigcup_{\gamma\in\Gamma_{U}} \overline{C}_{\gamma}(Q;r,s)$ est un bord analytique de~$\overline{C}_{U}(Q;r,s)$. Nous pouvons donc appliquer le m\^eme raisonnement que dans cette preuve afin de conclure.
\end{proof}

Afin d'aller plus loin, d\'emontrons un r\'esultat sur la topologie de~$X$.

\begin{prop}\label{prop:voisinages}
Soient~$x$ un point de~$X$ et~$U$ un voisinage de~$x$. Alors, il existe un voisinage~$V$ de~$b=\pi(x)$ dans~$B$, un polyn\^ome $P(T)\in \As[T]$ \`a coefficient dominant inversible et deux nombres r\'eels $r,s \ge 0$ v\'erifiant \mbox{$r\prec |P(T)(x)| < s$} tels que
\[ \{y\in X_{V}\, |\, r \le |P(T)(y)| \le s\} \subset U\]
et
\[\pi(\{y\in X_{V}\, |\, r \le |P(T)(y)| \le s\})=V.\]
\end{prop}
\begin{proof}
Nous pouvons supposer que~$U$ est ouvert. Il existe un polyn\^ome $P(T)\in\Hs(b)[T]$ et deux nombres r\'eels $r,s\ge 0$ tels que $r\prec |P(T)(x)| < s$ et
\[\{y\in X_{b}\, |\, r \le |P(T)(y)| \le s\} \subset U\cap X_{b}.\]
Les coefficients de~$P(T)$ sont limites de quotients d'\'el\'ements de~$\As$ dont les d\'enominateurs ne s'annulent pas en~$b$. Nous pouvons donc supposer que~$P(T)$ est \`a coefficients dans l'anneau total des fractions de~$\As$. Nous pouvons \'egalement supposer que son coefficient dominant ne s'annule pas en~$b$. Soit~$q\in \As$ un d\'enominateur commun \`a tous les coefficients de~$P(T)$. Puisqu'il ne s'annule pas en~$b$, nous avons
\[\{y\in X_{b}\, |\, r \le |P(T)(y)| \le s\} = \{y\in X_{b}\, |\, |q(b)|r \le |qP(T)(y)| \le |q(b)|s\}.\]
Par cons\'equent, nous pouvons supposer que~$P(T)$ est \`a coefficients dans~$\As$.

D'apr\`es~\cite{asterisque}, corollaire~1.1.12, la partie
\[K = \{y\in X\, |\, r \le |P(T)(y)| \le s\}\]
est un voisinage compact du point~$x$ dans~$X$. Par cons\'equent, $K\setminus U$ est une partie compacte de~$X$. Sa projection est une partie compacte de~$B$ qui \'evite le point~$b$. On obtient le r\'esultat annonc\'e en prenant pour~$V$ un voisinage de~$b$ dans  $B \setminus \pi(K\setminus U)$.
\end{proof}

\begin{coro}\label{cor:projouverte}
Pour tout~$n\in\N$, le morphisme de projection 
\[X_{n} = \E{n}{\As} \to \Mc(\As)=B\] 
est ouvert.
\end{coro}

\begin{defi}
Un point~$b$ de~$B$ est dit \textbf{ultram\'etrique typique} s'il appartient \`a l'int\'erieur de~$B_{\textrm{um}}$ et s'il poss\`ede un syst\`eme fondamental de voisinages compacts, spectralement convexes et admettant un bord analytique fini.
\end{defi}

\begin{prop}\label{prop:umtypiquerecurrence}
Soit~$b$ un point ultram\'etrique typique de~$B$. Soit~$n\in\N$. Tout point de $(X_{n})_{b} = \E{n}{\Hs(b)}$ est ultram\'etrique typique.
\end{prop}
\begin{proof}
Soit $x\in (X_{n})_{b}$. D'apr\`es la proposition~\ref{prop:voisinages}, il poss\`ede un syst\`eme fondamental de voisinages de la forme $\overline{C}_{U}(P;r,s)$, o\`u~$U$ est un voisinage compact et spectralement convexe de~$b$ dans~$B_{\textrm{um}}$ qui poss\`ede un bord analytique fini, $P$~un polyn\^ome unitaire \`a coefficients dans~$\Bs(U)$ et $r,s\in\R$. On conclut alors par la proposition~\ref{prop:couronneum}.
\end{proof}

\begin{prop}\label{prop:umtypiquedroite}
Soit~$b$ un point ultram\'etrique typique de~$B$. Soit~$P$ un polyn\^ome unitaire non constant \`a coefficients dans~$\As$. Tout point de $(X_{n})_{b}$ poss\`ede un syst\`eme fondamental de voisinages compacts et spectralement convexes qui satisfont la condition~$(N_{P(S)-T})$.
\end{prop}
\begin{proof}
On utilise la proposition~\ref{prop:voisinages} et le corollaire~\ref{cor:lemniscateum}.
\end{proof}

\subsection{Condition de s\'eparabilit\'e}

Int\'eressons-nous maintenant \`a une autre condition impliquant~$(N_{G})$ : il s'agit de la condition~$(R_{G})$ que nous avons introduite \`a la d\'efinition~5.2.5 de~\cite{asterisque}. 

\begin{defi}
Soit~$U$ une partie compacte de~$B$. On dit que~$U$ satisfait la {\bf condition~$\boldsymbol{(R_{G})}$} si elle est spectralement convexe et poss\`ede un bord analytique~$\Gamma_{U}$ sur lequel la fonction $|\textrm{R\'es}(G,G')|$ est born\'ee inf\'erieurement par un nombre r\'eel $m_{U}>0$.
\end{defi}

\'Enon\c{c}ons \`a pr\'esent un lemme classique.

\begin{lemm}
Notons $G(T) = T^d + \sum_{k=0}^{d-1} g_{k} T^k$. Pour tout point~$b$ de~$B$ et toute racine~$\alpha$ de~$G(b)(T)$ dans une cl\^oture alg\'ebrique de~$\Hs(b)$, nous avons
\[|\alpha| \le \max\Big(1, \sum_{k=0}^{d-1} |g_{k}(b)| \Big).\]
\end{lemm}

\begin{prop}
Toute partie compacte et spectralement convexe de~$B$ qui v\'erifie la condition~$(R_{G})$ v\'erifie aussi la condition~$(N_{G})$.
\end{prop}
\begin{proof}
Soit~$U$ une partie compacte et spectralement convexe de~$B$ qui v\'erifie la condition~$(R_{G})$. Consid\'erons le nombre r\'eel~$v$ dont il est question dans le th\'eor\`eme~\ref{theo:Wglobal} appliqu\'e \`a $B=\Bs(U)$. Posons $v' = \max(v,1,\sum_{k=0}^{d-1} \|g_{k}\|)$, avec les notations du lemme pr\'ec\'edent. On reprend alors la preuve de la proposition~5.2.7 de~\cite{asterisque}. (On y consid\'erait  $v' = \max\big(v,\max_{0\le k\le d-1}(\|g_{k}\|^{1/(d-k)})\big)$, ce qui n'est en fait correct que lorsque la partie~$U$ est contenue dans la partie ultram\'etrique de~$B$.)

En deux mots, il faut montrer que la norme d'un polyn\^ome~$F$ calcul\'ee en prenant le maximum des normes des coefficients de son reste~$R$ dans la division euclidienne par~$G$ est major\'ee, \`a une constante multiplicative pr\`es, par sa norme sur le ferm\'e de Zariski~$Z_{G}$ d\'efini par l'\'equation $G=0$. Il est possible d'exprimer le reste~$R$ en fonction de ses valeurs aux points de~$Z_{G}$, et d'en d\'eduire la majoration souhait\'ee, pour peu que l'on connaisse ses valeurs en suffisamment de points. C'est ce qu'assure la condition~$(R_{G})$, en demandant que les points de~$Z_{G}$ correspondent \`a des racines simples de~$G$.
\end{proof}

Comme pr\'ec\'edemment, nous allons maintenant nous restreindre au cas o\`u~$G$ est de la forme~$P(S)-T$.

Si~$Q$ est un polyn\^ome \`a coefficients dans un anneau de Banach, nous noterons~$\|Q\|_{\infty}$ le maximum des normes de ses coefficients.

\begin{lemm}\label{lem:approxdisque}
Soient~$b$ un point de~$B$. Soit~$Q_{0}$ un polyn\^ome s\'eparable et unitaire de degr\'e~$d\ge 1$ \`a coefficients dans~$\Hs(b)$. Il existe $r,\eta >0$ v\'erifiant la propri\'et\'e suivante : pour tout nombre r\'eel $t\in\of{]}{0,r}{]}$, pour tout polyn\^ome~$Q$ \`a coefficients dans~$\Hs(b)$ et de degr\'e~$d$ tel que $\|Q-Q_{0}\|_{\infty} \le \eta$, le disque~$\overline{D}_{b}(t)$ satisfait la condition~$(N_{Q(S)-T})$.
\end{lemm}
\begin{proof}
Par hypoth\`ese, le polyn\^ome $R(T) = \textrm{R\'es}_{S}(Q_{0}(S)-T,Q_{0}'(S))$ de~$\Hs(b)[T]$ ne s'annule pas en~$0$. Il existe donc un nombre r\'eel~$r>0$ tel que~$R(T)$ soit minor\'e par~$|R(0)|/2$ en tout point du disque~$\overline{D}_{b}(r)$. Par cons\'equent, il existe un nombre r\'eel~$\eta>0$ tel que, pour tout polyn\^ome~$Q$ \`a coefficients dans~$\Hs(b)$ et de degr\'e~$d$ tel que $\|Q-Q_{0}\|_{\infty} \le \eta$, le polyn\^ome $\textrm{R\'es}_{S}(Q(S)-T,Q'(S))$ soit minor\'e par~$|R(0)|/4$ sur~$\overline{D}_{b}(r)$. Le disque~$\overline{D}_{b}(r)$ satisfait alors la condition~$(R_{Q(S)-T})$, et donc la condition~$(N_{Q(S)-T})$.
\end{proof}

\begin{lemm}\label{lem:lemniscatesNPS-T}
Soit~$b$ un point de~$B$. Soit $P(S)$ un polyn\^ome unitaire de degr\'e~$d\ge 1$ \`a coefficients dans~$\As$ dont l'image dans~$\Hs(b)[S]$ est s\'eparable. Soit~$Q(T)$ un polyn\^ome unitaire \`a coefficients dans~$\Os_{B,b}$.

Soient~$s,s'\in\R$ tels que $0 < s' < s$. Il existe $s_{1},s_{2} \in \R$ v\'erifiant \mbox{$s' < s_{1} < s_{2} < s$} et un voisinage~$V$ de~$b$ dans~$B$ sur lequel les coefficients de~$Q$ sont d\'efinis tels que, pour tout $v\in \of{[}{s_{1},s_{2}}{]}$ et tout voisinage compact et spectralement convexe~$U$ de~$b$ dans~$V$, la lemniscate~$\overline{D}_{U}(Q;v)$ satisfasse la condition~$(R_{P(S)-T})$ et donc~$(N_{P(S)-T})$.

Soient~$r,r',s,s'\in\R$ tels que $0< r < r' < s' <s$. Il existe $r_{1},r_{2},s_{1},s_{2} \in \R$ v\'erifiant $r < r_{2} < r_{1} < r'$ et $s' < s_{1} < s_{2} < s$ et un voisinage~$V$ de~$b$ dans~$B$  sur lequel les coefficients de~$Q$ sont d\'efinis tels que, pour tout $u\in \of{[}{r_{2},r_{1}}{]}$, tout $v\in \of{[}{s_{1},s_{2}}{]}$ et tout voisinage compact et spectralement convexe~$U$ de~$b$ dans~$V$, la lemniscate~$\overline{C}_{U}(Q;u,v)$ satisfasse la condition~$(R_{P(S)-T})$ et donc~$(N_{P(S)-T})$.
\end{lemm}
\begin{proof}
Nous ne traiterons que le premier cas, le second se d\'emontrant par les m\^emes m\'ethodes. Quitte \`a remplacer~$\As$ par~$\Bs(W)$ o\`u~$W$ d\'esigne un voisinage compact et rationnel de~$b$, nous pouvons supposer que~$Q$ est \`a coefficients dans~$\As$.

Consid\'erons le polyn\^ome $R(T) = \textrm{R\'es}_{S}(P(S)-T,P'(S)) \in \As[T]$. Par hypoth\`ese, le polyn\^ome~$R(b)(T)$ n'est pas nul. Il existe donc $s_{1},s_{2}\in\R$ qui v\'erifient les in\'egalit\'es de l'\'enonc\'e et tels que~$R(b)(T)$ ne s'annule pas sur le compact $\overline{C}_{b}(Q;s_{1},s_{2})$. On en d\'eduit qu'il existe un voisinage~$V$ de~$b$ tel que~$R(T)$ ne s'annule pas sur $\overline{C}_{V}(Q;s_{1},s_{2})$. 
 
Soit~$U$ un voisinage compact et spectralement convexe de~$b$ dans~$V$. Pour tout $v\in \of{[}{s_{1},s_{2}}{]}$, la lemniscate~$\overline{D}_{U}(Q;v)$ poss\`ede un bord analytique contenu dans $\overline{C}_{U}(Q;s_{1},s_{2})$. Par cons\'equent, la lemniscate~$\overline{D}_{U}(Q;v)$ satisfait la condition~$(R_{P(S)-T})$.
\end{proof}

\begin{prop}\label{prop:NPS-T}
Soit~$b$ un point de~$B$. Soit $P(S)$ un polyn\^ome unitaire non constant \`a coefficients dans~$\As$ dont l'image dans~$\Hs(b)[S]$ est s\'eparable. Alors, tout point de~$X_{b}$ poss\`ede un syst\`eme fondamental de voisinages compacts et spectralement convexes qui satisfont la condition~$(R_{P(S)-T})$ et donc~$(N_{P(S)-T})$.
\end{prop}
\begin{proof}
Soient~ $x$ un point de~$X_{b}$ et~$U$ un voisinage ouvert de~$x$ dans~$X$. D'apr\`es la proposition~\ref{prop:voisinages}, il existe un voisinage compact~$V$ de~$b$ dans~$B$, un polyn\^ome unitaire~$Q(T)$ \`a coefficients dans~$\Os_{B,b}$ et deux nombres r\'eels~$r$ et~$s$ v\'erifiant $0\le r < s$ tels que $r \prec |Q(x)| < s$ et $\overline{C}_{V}(Q;r,s) \subset U$. Le r\'esultat d\'ecoule alors du lemme~\ref{lem:lemniscatesNPS-T}.
\end{proof}

\section{Endomorphismes de la droite}\label{sec:endodroite}

Soient~$d\in\N^*$ et~$P$ un polyn\^ome de degr\'e~$d$ \`a coefficients dans~$\As$ dont le coefficient dominant est inversible. Le morphisme naturel
\[\As[T] \to \As[T,S]/(P(S)-T) \xrightarrow[]{\sim} \As[S],\]
qui envoie~$T$ sur~$P(S)$, induit un morphisme continu~$\varphi$ de la droite~$X$ dans elle-m\^eme. 

Nous pouvons appliquer dans ce contexte certains des r\'esultats de la partie pr\'ec\'edente. Expliquons comment obtenir un analogue de la proposition~\ref{prop:Bfini}.

\begin{prop}\label{prop:Blemniscate}
Soit~$U$ une partie compacte de~$X$ (avec variable~$T$) qui satisfait la condition~$(N_{P(S)-T})$. Les normes~$\|.\|_{U,\textrm{div}}$ et  $\|.\|_{U,w,\textrm{r\'es}}$, pour $w\ge v$, sont alors \'equivalentes \`a la norme spectrale sur \mbox{$\Bs(U)[S]/(P(S)-T)$}. De plus, nous avons un isomorphisme admissible naturel
\[\Bs(U)[S]/(P(S)-T) \xrightarrow[]{\sim} \Bs(\varphi^{-1}(U)).\]

En particulier, pour tous $r,s\ge 0$ et toute partie compacte et spectralement convexe~$W$ de~$X$ tels que~$\overline{C}_{W}(r,s)$ satisfait la condition~$(N_{P(S)-T})$, nous avons un isomorphisme admissible naturel
\[\Bs(\overline{C}_{W}(r,s))[S]/(P(S)-T) \xrightarrow[]{\sim} \Bs(\overline{C}_{W}(P;r,s)).\]
\end{prop}
\begin{proof}
Consid\'erons le plan $X_{2} = \E{2}{\As}$ avec variables~$S$ et~$T$ et notons $\pi_{T} : X_{2} \to X$ sa projection sur la variable~$T$. Notons~$Z$ le ferm\'e de Zariski de~$\pi_{T}^{-1}(U)$ d\'efini par $P(S)=T$. D'apr\`es la proposition~\ref{prop:Bfini}, nous avons un isomorphisme admissible $\Bs(U)[S]/(P(S)-T) \xrightarrow[]{\sim} \Bs(Z)$. 

Notons $\pi_{S} : X_{2}=\E{2}{\As} \to \E{1}{\As}$ le morphisme de projection sur la variable~$S$. Nous avons un isomorphisme admissible $\Bs(\pi_{S}(Z))[T]/(T-P(S)) \xrightarrow[]{\sim} \Bs(Z)$ (ce~que l'on d\'emontre soit directement, soit en invoquant de nouveau la proposition~\ref{prop:Bfini} et la remarque~\ref{rem:degre1}). On conclut en remarquant que l'on a \'egalement un isomorphisme admissible $\Bs(\pi_{S}(Z)) \simeq \Bs(\pi_{S}(Z))[T]/(T-P(S))$ et que~$\pi_{S}(Z)$ n'est autre que~$\varphi^{-1}(U)$.
\end{proof}

\begin{coro}\label{cor:approximationlemniscate}
Pour tout $s\ge 0$ et toute partie compacte et spectralement convexe~$W$ de~$B$ tels que~$\overline{D}_{W}(s)$ satisfait la condition~$(N_{P(S)-T})$, l'anneau~$\Bs(W)[S]$ est dense dans $\Bs(\overline{D}_{W}(P;s))$ pour la norme uniforme.

Pour tous $r,s\in\R$ v\'erifiant $0<r\le s$ et toute partie compacte et spectralement convexe~$W$ de~$B$ tels que~$\overline{C}_{W}(r,s)$ satisfait la condition~$(N_{P(S)-T})$, l'anneau~$\Bs(W)[S,P(S)^{-1}]$ est dense dans $\Bs(\overline{C}_{W}(P;r,s))$ pour la norme uniforme.
\end{coro}

Nous pouvons maintenant reprendre les r\'esultats de la fin de la section~\ref{section:morphismesfinis} \textit{mutatis mutandis}, en utilisant les r\'esultats du corollaire~\ref{cor:impliqueDG} et des propositions~\ref{prop:umtypiquedroite} et~\ref{prop:NPS-T}. 

\begin{theo}\label{theo:finiPS-Tseparable}
Soit~$V$ une partie de~$B$ telle qu'en tout point~$b$ qui n'est pas ultram\'etrique typique, le polyn\^ome $P(b)(T)$ soit s\'eparable. Alors, le morphisme
\[\begin{array}{ccc}
\Os_{X_{V}}^d & \to & \varphi_{*}\Os_{X_{V}}\\
(a_{0},\ldots,a_{d-1}) & \mapsto & \disp \sum_{i=0}^d a_{i}\, T^i
\end{array}\]
est un isomorphisme de~$\Os_{X_{V}}$-modules. 

En particulier, pour toute partie~$U$ de~$V$, le morphisme naturel
\[\Os_{X}(U)[S]/(P(S)-T) \to \Os_{X}(\varphi^{-1}(U))\]
est un isomorphisme.
\end{theo}
\begin{proof}
D'apr\`es le corollaire~\ref{cor:impliqueDG} et la proposition~\ref{prop:NPS-T}, tout point de~$X_{V}$ satisfait la condition~$(D_{P(S)-T})$ et poss\`ede un syst\`eme fondamental de voisinages compacts qui satisfont la condition~$(N_{P(S)-T})$.
\end{proof}

\'Enon\c{c}ons un cas particulier dont nous nous servirons par la suite.

\begin{coro}\label{cor:Olemniscate}
Soit~$V$ une partie de~$B$ telle qu'en tout point~$b$ qui n'est pas ultram\'etrique typique, le polyn\^ome $P(b)(T)$ soit s\'eparable. Soient~$r,s\ge 0$. Le morphisme naturel
\[\Os_{X}(\overline{C}_{V}(r,s))[S]/(P(S)-T) \to \Os_{X}(\overline{C}_{V}(P;r,s))\]
est un isomorphisme. 
\end{coro}

\section{Th\'eor\`eme de Weierstra{\ss} local}\label{section:Weierstrass}

Ainsi que nous l'avons expliqu\'e dans l'introduction, l'objectif principal de ce texte est de d\'emontrer un th\'eor\`eme de division Weierstra{\ss} donnant acc\`es \`a la structure des anneaux locaux des espaces de Berkovich affines. En g\'eom\'etrie analytique complexe, o\`u ce th\'eor\`eme est connu depuis longtemps et joue un r\^ole majeur, tous les points sont d\'efinis sur~$\C$ et l'on se ram\`ene par translation, \`a traiter le cas du point~$0$. Dans~\cite{asterisque}, nous nous sommes d\'ej\`a int\'eress\'e \`a l'analogue de cette situation et avons d\'emontr\'e un th\'eor\`eme de division de Weierstra{\ss} au voisinage du point~$0$ d'une fibre~$X_{b}$, avec $b\in B$ (\textit{cf.}~th\'eor\`eme~2.2.3). Cependant les fibres~$X_{b}$ contiennent en g\'en\'eral bien d'autres types de points pour lesquels il serait utile de disposer d'un tel th\'eor\`eme.
 
\begin{defi}
Soit~$b$ un point de~$B$. Soit~$n\in\N$. Un point~$x$ de~$\E{n}{\As}$ (avec variables $T_{1},\dots,T_{n}$) au-dessus de~$b$ est dit \textbf{rigide} si~$\Hs(x)$ est une extension finie de~$\Hs(b)$, autrement dit, si $T_{1}(x),\dots,T_{n}(x)$ sont alg\'ebriques sur~$\Hs(b)$.

Un point~$x$ de~$\E{n}{\As}$ au-dessus de~$b$ est dit \textbf{rigide \'epais} si~$\kappa(x)$ est une extension finie de~$\kappa(b)$, autrement dit, si $T_{1}(x),\dots,T_{n}(x)$ sont alg\'ebriques sur~$\kappa(b)$. Un point rigide qui n'est pas \'epais est appel\'e point \textbf{rigide maigre}.
\end{defi}

\begin{defi}
Soient~$b$ un point de~$B$ et~$x$ un point rigide de~$X_{b}$. Il existe un unique polyn\^ome irr\'eductible et unitaire $P(T) \in \Hs(b)[T]$ qui s'annule au point~$x$. Nous l'appellerons \textbf{polyn\^ome minimal} du point~$x$. 

Remarquons que~$x$ est un point rigide \'epais si, et seulement si, $P(T)\in\kappa(b)[T]$.
\end{defi}

Le th\'eor\`eme qui suit est un th\'eor\`eme de division de Weierstra{\ss} au voisinage d'un point rigide~$x$ quelconque d'une fibre~$X_{b}$, avec~$b\in B$. Lorsque le point~$x$ est \'epais, les coefficients de son polyn\^ome minimal~$P$ sont d\'efinis au voisinage de~$b$ et l'on peut utiliser les r\'esultats des sections pr\'ec\'edentes pour \'etudier le morphisme d\'efini par~$P$ d'un voisinage de~$X_{b}$ dans lui-m\^eme. Lorsque le point~$x$ est maigre, en revanche, son polyn\^ome minimal~$P$ n'est d\'efini que sur~$\Hs(b)$ et il nous faudra utiliser des polyn\^omes auxiliaires proches de~$P$ mais qui s'\'etendent au voisinage de~$b$.

\begin{theo}[de division de Weierstra{\ss}]\label{theo:Weierstrassapproche}
Soit~$b$ un point de~$B$. Soit $P(S) \in \Hs(b)[S]$ un polyn\^ome irr\'eductible et unitaire de degr\'e~$d$. Si le polyn\^ome~$P$ est ins\'eparable et si le corps r\'esiduel~$\Hs(b)$ est trivialement valu\'e, supposons que le point~$b$ est ultram\'etrique typique. 

Notons~$x$ le point rigide de la fibre~$X_{b}$ d\'efini par l'\'equation~$P=0$. Soit~$G$ un \'el\'ement de l'anneau local~$\Os_{X,x}$. Supposons que son image dans~$\Os_{X_{b},x}$ n'est pas nulle et notons~$n$ sa valuation $P$-adique.

Alors, pour tout $F \in \Os_{X,x}$, il existe un unique couple $(Q,R) \in {\Os_{X,x}}^2$ tel que
\begin{enumerate}[i)]
\item $F = QG + R$ ;
\item $R \in \Os_{B,b}[S]$ est un polyn\^ome de degr\'e strictement inf\'erieur \`a~$nd$.
\end{enumerate}
\end{theo}
\begin{proof}
Si le degr\'e~$d$ du polyn\^ome~$P$ est nul, le r\'esultat est \'evident. Nous supposerons donc que $d\ge 1$. Il existe $\alpha_{0},\ldots,\alpha_{d-1} \in \Hs(b)$ tels que 
\[ P(S) = S^d + \sum_{i=0}^{d-1} \alpha_{i} S^i. \] 
Soit $F\in \Os_{X,x}$. 

Dans un premier temps, nous supposerons que le corps~$\Hs(b)$ est parfait ou que sa valeur absolue n'est pas triviale. Nous indiquons \`a la fin de la preuve les modifications \`a apporter dans les cas restant. 

Il existe un nombre r\'eel $r>0$ et un voisinage~$U$ de $X_{b} \cap \{|P|\le r\}$ dans~$X$ sur lequel~$F$ et~$G$ soient d\'efinis. Il existe un \'el\'ement inversible~$H$ de~$\Os_{X_{b},x}$ tel que $G=P^n H$. Quitte \`a diminuer~$r$, nous pouvons supposer que~$H$ et~$H^{-1}$ sont d\'efinis sur $X_{b} \cap \{|P|\le r\}$ dans~$X_{b}$. Fixons un nombre r\'eel $s\in \of{]}{0,r}{[}$. Soit~$\eps>0$. Nous imposerons plus tard d'autres conditions (qui ne d\'ependront que de~$s$ et de~$P$) sur cet \'el\'ement.

$\bullet$ Supposons que le polyn\^ome~$P$ est s\'eparable. D'apr\`es le lemme~\ref{lem:approxdisque}, quitte \`a diminuer~$r$, nous pouvons trouver un polyn\^ome~$P_{\eps}(S)$ unitaire de degr\'e~$d$ \`a coefficients dans~$\Os_{B,b}$ v\'erifiant $\|P_{\eps}(b) - P\|_{\infty} \le \eps$ et tel que le disque~$\overline{D}_{b}(r)$ satisfasse la condition~$(N_{P_{\eps}(S)-T})$. 

Nous pouvons en outre supposer que le polyn\^ome~$P_{\eps}(b)(S)$ est s\'eparable. Dans ce cas, d'apr\`es la proposition~\ref{prop:NPS-T}, tout point de~$\overline{D}_{b}(r)$ poss\`ede un syst\`eme fondamental de voisinages compacts et spectralement convexes qui satisfont la condition~$(N_{P_{\eps}(S)-T})$.

$\bullet$ Supposons que la valeur absolue associ\'ee au point~$b$ est ultram\'etrique, mais pas triviale. Nous pouvons trouver un polyn\^ome~$P_{\eps}(S)$ unitaire de degr\'e~$d$ \`a coefficients dans~$\Os_{B,b}$ tel que $\|P_{\eps}(b) - P\|_{\infty} \le \eps$. Nous pouvons supposer que le polyn\^ome $P_{\eps}(b)(S)$ est s\'eparable. Dans ce cas, d'apr\`es la proposition~\ref{prop:NPS-T}, tout point de~$\overline{D}_{b}(r)$ poss\`ede un syst\`eme fondamental de voisinages compacts et spectralement convexes qui satisfont la condition~$(N_{P_{\eps}(S)-T})$.

Notons~$\eta_{r}$ l'unique point du bord de Shilov du disque~$\overline{D}_{b}(r)$. Puisque~$T(\eta_{r})$ est transcendant sur~$\Hs(b)$ et que le polyn\^ome $P_{\eps}(b)(S)$ est s\'eparable, le polyn\^ome $P_{\eps}(b)(S) - T(\eta_{r}) \in \Hs(\eta_{r})[S]$ l'est aussi. Le disque~$\overline{D}_{b}(r)$ satisfait donc la condition~$(N_{P_{\eps}(S)-T})$, en vertu du corollaire~\ref{cor:normeresuniforme}.

$\bullet$ Supposons finalement que le polyn\^ome~$P$ est ins\'eparable et le corps~$\Hs(b)$ trivialement valu\'e. Nous pouvons alors relever~$P$ en un polyn\^ome \`a coefficients dans~$\Os_{B,b}$ que nous noterons~$P_{\eps}$. Par hypoth\`ese, le point~$b$ est ultram\'etrique typique. D'apr\`es la proposition~\ref{prop:couronneum}, le disque~$\overline{D}_{b}(r)$ satisfait la condition~$(N_{P_{\eps}(S)-T})$. D'apr\`es la proposition~\ref{prop:umtypiquedroite}, tout point de~$\overline{D}_{b}(r)$ poss\`ede un syst\`eme fondamental de voisinages compacts et spectralement convexes qui satisfont la condition~$(N_{P_{\eps}(S)-T})$.

\medskip

Quitte \`a diminuer~$\eps$, nous pouvons supposer que $X_{b} \cap \{|P_{\eps}|\le r\}$ est un voisinage du point~$x$ dans~$X_{b}\cap U$ et que~$H$ et~$H^{-1}$ sont d\'efinis sur $X_{b} \cap \{|P_{\eps}|\le r\}$ dans~$X_{b}$. Un argument de compacit\'e montre qu'il existe un voisinage compact~$V$ de~$b$ dans~$B$ tel que $X_{V}  \cap \{|P_{\eps}|\le r\}$ soit contenu dans~$U$. Quitte \`a restreindre~$V$, nous pouvons supposer que $P_{\eps}(S)\in \Bs(V)[S]$.

Consid\'erons l'anneau de Banach~$A=\Bs(V)\la |T|\le r\ra$. D'apr\`es le th\'eor\`eme~\ref{theo:Wglobal} et le corollaire~\ref{cor:resnorme}, il existe un nombre r\'eel~$v>0$ tel que, pour tout~$w\ge v$ et toute $A$-alg\`ebre de Banach~$A'$ telle que le morphisme structural $A\to A'$ diminue les normes, la semi-norme~$\|.\|_{w,\textrm{r\'es}}$ d\'efinie sur le quotient~$A'[S]/(P_{\eps}(S)-T)$ soit une norme. En outre, d'apr\`es le corollaire~\ref{cor:eqnormedivres}, il existe une constante~$C>0$ tel que pour toute telle alg\`ebre~$A'$ et tout \'el\'ement~$F$ de $A'[S]/(P_{\eps}(S)-T)$, nous ayons
\[ \|F\|_{A',w,\textrm{r\'es}} \le \|F\|_{A',\textrm{div}} \le C\, \|F\|_{A',w,\textrm{r\'es}}.\] 
Lorsque l'alg\`ebre~$A'$ sera de la forme $A' = \Bs(W)\la |T|\le t\ra$, nous noterons $\|.\|_{W,t,w,\textrm{r\'es}}$ et $\|.\|_{W,t,\textrm{div}}$ les normes $\|.\|_{A',w,\textrm{r\'es}}$ et $\|.\|_{A',\textrm{div}}$ respectivement.

Par choix de~$P_{\eps}$, le disque~$\overline{D}_{b}(r)$ satisfait la condition~$(N_{P_{\eps}(S)-T})$. Il existe donc un nombre r\'eel~$v'>0$ et une constante~$D>0$ tel que, pour tout $w'\ge v'$ et tout \'element~$f$ de \mbox{$\Bs(\overline{D}_{b}(r))[S]/(P_{\eps}(S)-T)$}, on ait
\[ \|f\|_{w',\textrm{r\'es}} \le D \|f\|_{X_{b} \cap \{|P_{\eps}|\le r\}},\]
o\`u $\|.\|_{w',\textrm{r\'es}}$ d\'esigne la norme r\'esiduelle sur $\Bs(\overline{D}_{b}(r))[S]/(P_{\eps}(S)-T)$ induite par la norme $\|.\|_{w'}$ sur $\Bs(\overline{D}_{b}(r))[S]$, l'alg\`ebre $\Bs(\overline{D}_{b}(r))$ \'etant munie de la norme spectrale.

Soit $w_{0} \ge \max(v,v')$. Dor\'enavant toutes les normes r\'esiduelles seront calcul\'ees avec ce rayon~$w_{0}$.

En utilisant la proposition~\ref{prop:comparaisonnormes} pour comparer les normes sur $\Bs(\overline{D}_{b}(r))$ et $\Hs(b)\la |T|\le s\ra$, on montre que, pour tout~$f$ dans \mbox{$\Hs(b)\la |T|\le r\ra[S]/(P_{\eps}(S)-T)$}, nous avons
\[ \|f\|_{b,s,w_{0},\textrm{r\'es}} \le \frac{r}{r-s}\, D \|f\|_{X_{b} \cap \{|P_{\eps}|\le r\}}. \]

Le corollaire~\ref{cor:Olemniscate} assure que la fonction~$H^{-1}$ poss\`ede un repr\'esentant dans \mbox{$\Os_{X_{b}}(\overline{D}_{b}(r))[S]/(P_{\eps}(S)-T)$}. Par cons\'equent, elle est approchable uniform\'ement sur \mbox{$X_{b} \cap \{|P_{\eps}|\le r\}$} par des \'el\'ements de~$\Os_{B,b}[S]$. Quitte \`a restreindre~$V$, nous pouvons donc supposer qu'il existe $K \in \Bs(V)\la |T|\le r\ra [S]/(P_{\eps}(S)-T)$ tel que 
\[ \|K(b)G(b)-P^n\|_{X_{b} \cap \{|P_{\eps}|\le r\})} \le \frac{1}{2}\, C^{-1}\, D^{-1}\, \left(1 - \frac{s}{r}\right) s^n. \]

D'apr\`es les corollaires~\ref{cor:Olemniscate} et~\ref{cor:dvptdisque}, quitte \`a restreindre~$V$ de nouveau, nous pouvons supposer que les germes des fonctions~$F$ et~$G$ au voisinage de~$x$ poss\`edent \'egalement des repr\'esentants dans $\Bs(V)\la |T|\le r\ra [S]/(P_{\eps}(S)-T)$.

\medskip

Soit~$W$ un voisinage compact de~$b$ dans~$V$. Tout \'el\'ement~$\varphi$ appartenant \`a \mbox{$\Bs(W)\la |T|\le s\ra [S]/(P_{\eps}(S)-T)$} peut s'\'ecrire de fa\c{c}on unique sous la forme
\[\varphi = \sum_{i=0}^{d-1} (\alpha_{i}(\varphi)T^n + \beta_{i}(\varphi)) S^i,\]
o\`u les~$\alpha_{i}(\varphi)$ sont des \'el\'ements de $\Bs(W)\la |T|\le s\ra$ et les~$\beta_{i}(\varphi)$ des \'el\'ements de $\Bs(W)[T]$ de degr\'e strictement inf\'erieur \`a~$n$. Posons
\[\alpha(\varphi) = \sum_{i=0}^{d-1} \alpha_{i}(\varphi) S^i \textrm{ et } \beta(\varphi) = \sum_{i=0}^{d-1} \beta_{i}(\varphi) S^i.\]
Remarquons que~$\beta(\varphi)$ est un polyn\^ome en~$S$ de degr\'e inf\'erieur \`a~$nd-1$. Nous avons
\[\varphi = \alpha(\varphi)T^n + \beta(\varphi).\]
Remarquons que
\[\|\alpha(\varphi)\|_{W,s,\textrm{div}} \le s^{-n}\, \|\varphi\|_{W,s,\textrm{div}},\]
d'o\`u l'on tire
\[\|\alpha(\varphi)\|_{W,s,w_{0},\textrm{r\'es}} \le C\, s^{-n}\, \|\varphi\|_{W,s,w_{0},\textrm{r\'es}}.\]

Consid\'erons, \`a pr\'esent, l'endomorphisme 
\[ A_{W} :
{\renewcommand{\arraystretch}{1.2}\begin{array}{ccc}
\Bs(W)\la |T|\le s\ra [S]/(P_{\eps}(S)-T) & \to &\Bs(W)\la |T|\le s\ra [S]/(P_{\eps}(S)-T)\\
\varphi & \mapsto & \alpha(\varphi)\, KG + \beta(\varphi)
\end{array}}.\]

Quel que soit $\varphi \in \Bs(W)\la |T|\le s\ra [S]/(P_{\eps}(S)-T)$, nous avons 
\[{\renewcommand{\arraystretch}{1.4}\begin{array}{rcl}
\|A_{W}(\varphi)-\varphi\|_{W,s,w_{0},\textrm{r\'es}} &=& \|\alpha(\varphi)\, (KG-T^n)\|_{W,s,w_{0},\textrm{r\'es}}\\
&\le&  \|\alpha(\varphi)\|_{W,s,w_{0},\textrm{r\'es}}\, \|KG-P_{\eps}^n\|_{W,s,w_{0},\textrm{r\'es}}\\
&\le& C s^{-n}\,  \|KG-P_{\eps}^n\|_{W,s,w_{0},\textrm{r\'es}}\,\|\varphi\|_{W,s,w_{0},\textrm{r\'es}}.
\end{array}}\]

Or 
\[{\renewcommand{\arraystretch}{2}\begin{array}{rcl}
\|K(b)G(b)-P^n\|_{b,s,w_{0},\textrm{r\'es}} &\le&  \disp \frac{r}{r-s}\, D \|K(b)G(b)-P^n\|_{X_{b} \cap \{|P_{\eps}|\le s\} }\\
&\le&\disp \frac{1}{2}\, C^{-1}\, s^n
\end{array}}\]
et
\[{\renewcommand{\arraystretch}{1.6}\begin{array}{rcl}
\|P^n-P_{\eps}(b)^n\|_{b,s,w_{0},\textrm{r\'es}} &\le& \disp\|P-P_{\eps}(b)\|_{b,s,w_{0},\textrm{r\'es}}\, \left\|\sum_{i=0}^{n-1} P^i P_{\eps}(b)^{n-1-i}\right\|_{b,s,w_{0},\textrm{r\'es}}\\
&\le& \|P-P_{\eps}(b)\|_{b,s,\textrm{div}} \cdot n \max(\|P\|_{b,s,w_{0},\textrm{r\'es}},\|P_{\eps}(b)\|_{b,s,w_{0},\textrm{r\'es}})^{n-1}\\
&\le& n\,  \|P-P_{\eps}(b)\|_{b,s,\textrm{div}}\, \max(\|P\|_{b,s,\textrm{div}},\|P_{\eps}(b)\|_{b,s,\textrm{div}})^{n-1}\\
&\le& \disp n\, \eps  \max_{0\le i\le d-1} (|\alpha_{i}|+\eps)^{n-1}.
\end{array}}\]
Quitte \`a diminuer~$\eps$, nous pouvons donc supposer que 
\[\|K(b)G(b)-P_{\eps}(b)^n\|_{b,\textrm{r\'es}} < C^{-1}\, s^n\]
puis, quitte \`a restreindre le voisinage~$V$ de~$b$, que 
\[ \|KG-P_{\eps}^n\|_{W,\textrm{r\'es}} < C^{-1}\, s^n.\]

Sous ces conditions, l'endomorphisme $A_{W} = \textrm{Id} + (A_{W}-\textrm{Id})$ est inversible. Les \'el\'ements $Q=\alpha(A_{V}^{-1}(F))K$ et $R=\beta(A_{V}^{-1}(F))$ v\'erifient $QG+R = F$. Ceci prouve l'existence d'un couple $(Q,R)$ v\'erifiant les propri\'et\'es de l'\'enonc\'e.

D\'emontrons \`a pr\'esent l'unicit\'e. Soit un couple $(Q',R') \in {\Os_{X,x}}^2$ v\'erifiant les propri\'et\'es de l'\'enonc\'e. Nous pouvons choisir le voisinage~$U$ de~$x$ au d\'ebut de la preuve de fa\c{c}on que~$Q'$ et~$R'$ soient d\'efinis sur~$U$. D'apr\`es les corollaires~\ref{cor:Olemniscate} et~\ref{cor:dvptdisque}, les germes des fonctions~$Q'$ et~$R'$ au voisinage de~$x$ poss\`edent des repr\'esentants dans $\Bs(W)\la |T|\le r\ra [S]/(P_{\eps}(S)-T)$, pour un certain voisinage compact~$W$ de~$b$ dans~$V$. L'injectivit\'e du morphisme~$A_{W}$ permet alors de conclure.
\end{proof}

\begin{rema}
Nous avons d\'ej\`a \'enonc\'e, dans~\cite{asterisque}, un r\'esultat de division de Weierstra{\ss} partiel au voisinage des points rigides \'epais, sous certaines conditions (\textit{cf.}~th\'eor\`eme~5.4.4). La d\'emonstration en est malheureusement incorrecte (d\^u \`a une r\'eduction liminaire inconsid\'er\'ee). Le th\'eor\`eme propos\'e ici est, de toute fa\c{c}on, plus g\'en\'eral.
\end{rema}

\'Enon\c{c}ons maintenant un corollaire tr\`es utile.

\begin{coro}\label{cor:pointpasepais}
Soit~$b$ un point de~$B$. Si son corps r\'esiduel~$\Hs(b)$ est imparfait et trivialement valu\'e, supposons que le point~$b$ est ultram\'etrique typique. Soit~$x$ un point de~$X_{b}$ qui ne soit pas un point rigide \'epais. Soit~$f$ un \'el\'ement de~$\Os_{X,x}$ tel que~$f(x)=0$. Alors~$f$ est nul dans~$\Os_{X_{b},x}$.
\end{coro}
\begin{proof}
Supposons que~$x$ soit un point de type~2, 3, 4 ou de type~1 non rigide. Dans ce cas, l'anneau local~$\Os_{X_{b},x}$ est un corps et le r\'esultat est imm\'ediat.

Il nous reste \`a traiter le cas o\`u~$x$ est un point rigide maigre. Consid\'erons son polyn\^ome minimal unitaire $P(S)\in\Hs(b)[S]$. Notons~$d$ son degr\'e. Supposons, par l'absurde, que~$f$ n'est pas nul dans~$\Os_{X_{b},x}$. Notons~$n$ sa valuation $P$-adique dans~$\Os_{X_{b},x}$. Appliquons le th\'eor\`eme~\ref{theo:Weierstrassapproche} avec $F=S^{nd}$ et $G=f$. Il assure qu'il existe un polyn\^ome~$R$ \`a coefficients dans~$\Os_{B,b}$ et de degr\'e strictement inf\'erieur \`a~$nd$ tel que $F=Qf+R$. En particulier, le polyn\^ome non nul~$F-R$ s'annule en~$x$. On en d\'eduit que le point~$x$ est \'epais, ce qui contredit l'hypoth\`ese.
\end{proof}

Une fois le th\'eor\`eme de division de Weierstra{\ss} d\'emontr\'e, nous pouvons en d\'eduire, par un raisonnement classique, le th\'eor\`eme de pr\'eparation de Weierstra{\ss}. Remarquons qu'il ne vaut que pour les points rigides \'epais.

\begin{theo}[de pr\'eparation de Weierstra{\ss}]\label{theo:preparationWeierstrass}
Soit~$b$ un point de~$B$. Soit $P(S) \in \kappa(b)[S]$ un polyn\^ome unitaire de degr\'e~$d$ dont l'image dans~$\Hs(b)[S]$ est irr\'eductible. Si le polyn\^ome~$P$ est ins\'eparable et si le corps r\'esiduel~$\Hs(b)$ est trivialement valu\'e, supposons que le point~$b$ est ultram\'etrique typique. 

Notons~$x$ le point rigide de la fibre~$X_{b}$ d\'efini par l'\'equation~$P=0$. Soit~$G$ un \'el\'ement de l'anneau local~$\Os_{X,x}$. Supposons que son image dans~$\Os_{X_{b},x}$ n'est pas nulle et notons~$n$ sa valuation $P$-adique.

Alors il existe un unique couple $(\Omega,E) \in {\Os_{X,x}}^2$ tel que :
\begin{enumerate}[i)]
\item $\Omega \in \Os_{B,b}[S]$ est un polyn\^ome unitaire de degr\'e~$nd$ v\'erifiant $\Omega(b)(S) = P(S)^n$ dans $\Hs(b)[S]$ ;
\item $E$ est inversible dans~$\Os_{X,x}$ ;
\item $G = \Omega E$.
\end{enumerate}
\end{theo}
\begin{proof}
Supposons qu'un tel couple $(\Omega,E)\in {\Os_{X,x}}^2$ existe et \'ecrivons \mbox{$\Omega = P^n + R$}, o\`u $R(S)\in\Os_{B,b}[S]$ est un polyn\^ome de degr\'e strictement inf\'erieur \`a~$nd$. Nous avons alors $P^n = E^{-1} G - R$ dans~$\Os_{X,x}$ et le th\'eor\`eme de division de Weierstra{\ss} assure l'unicit\'e des \'el\'ements~$E$ et~$R$, et donc $\Omega$.

Pour d\'emontrer l'existence, appliquons le th\'eor\`eme de division de Weierstra{\ss} \`a~$P^n$ et~$G$. On en d\'eduit qu'il existe des \'el\'ements~$Q$ de~$\Os_{X,x}$ et~$R$ de~$\Os_{B,b}[S]$, $R$ \'etant de degr\'e strictement inf\'erieur \`a~$nd$, tels que $P^n = QG+R$. On v\'erifie que~$Q$ est inversible dans~$\Os_{X,x}$ et que les \'el\'ements~$E=Q^{-1}$ et $\Omega = P^n-R$ conviennent.
\end{proof}

Une autre cons\'equence classique du th\'eor\`eme de division de Weierstra{\ss} a trait aux morphismes finis. Nous pouvons, en effet, maintenant, g\'en\'eraliser le th\'eor\`eme~\ref{theo:finiPS-Tseparable} en rel\^achant l'hypoth\`ese sur le polyn\^ome~$P$. Cette partie est classique, du moins en ce qui concerne les espaces analytiques complexes, et se trouve r\'edig\'ee dans~\cite{GR}, I, \S~2, par exemple. Nous renvoyons le lecteur d\'esireux de lire les d\'etails dans le cadre plus g\'en\'eral o\`u nous nous sommes plac\'es \`a~\cite{asterisque}, \S~5.5. Nous nous contenterons ici d'indiquer les \'enonc\'es des r\'esultats.

Reprenons le cadre d\'ecrit \`a la section~\ref{sec:endodroite}. Soient~$d\in\N^*$ et~$P$ un polyn\^ome de degr\'e~$d$ \`a coefficients dans~$\As$ dont le coefficient dominant est inversible. Consid\'erons le morphisme~$\varphi$ de la droite affine~$X$ dans elle-m\^eme induit par le polyn\^ome~$P$.

Il peut \^etre d\'ecrit de la fa\c{c}on suivante. Notons $X_{2} = \E{2}{\As}$, avec variables~$S$ et~$T$. Consid\'erons le ferm\'e de Zariski~$Z$ de~$X_{2}$ d\'efini par le polyn\^ome $G(S,T) = P(S)-T$ et identifions-le \`a la droite affine~$X$ (avec variable~$S$). Le morphisme~$\varphi$ est alors la compos\'ee du plongement de~$X$ dans~$X_{2}$ et de la projection sur la deuxi\`eme coordonn\'ee~$T$.

\begin{theo}
Soient~$b$ un point de~$B$ et~$x$ un point de~$X_{b}$ (avec variable~$T$). Si son corps r\'esiduel~$\Hs(x)$ est imparfait et trivialement valu\'e, supposons que le point~$x$ est ultram\'etrique typique. Notons $\varphi^{-1}(x) = \{y_{1},\dots,y_{t}\} \subset Z$. Soit $(f_{1},\dots,f_{t}) \in \prod_{i=1}^t \Os_{X_{2},y_{i}}$. Alors, il existe un unique \'el\'ement $(r,q_{1},\dots,q_{t})$ de $\Os_{X,x}[S] \times \prod_{i=1}^t \Os_{X_{2},y_{i}}$ v\'erifiant les propri\'et\'es suivantes :
\begin{enumerate}[i)]
\item pour tout $i\in\cn{1}{t}$, nous avons $f_{i} = q_{i}G + r$ dans~$\Os_{X_{2},y_{i}}$ ;
\item le polyn\^ome~$r$ est de degr\'e strictement inf\'erieur \`a~$d$.
\end{enumerate}
\end{theo}

\begin{theo}\label{theo:finiPS-T}
Supposons que tout point~$b$ de~$B$ dont le corps r\'esiduel~$\Hs(b)$ est de caract\'eristique non nulle et trivialement valu\'e est ultram\'etrique typique. Alors, le morphisme
\[\begin{array}{ccc}
\Os_{X}^d & \to & \varphi_{*}\Os_{X}\\
(a_{0},\ldots,a_{d-1}) & \mapsto & \disp \sum_{i=0}^d a_{i}\, T^i
\end{array}\]
est un isomorphisme de~$\Os_{X}$-modules. 

En particulier, pour toute partie~$U$ de~$X$, le morphisme naturel
\[\Os_{X}(U)[S]/(P(S)-T) \to \Os_{X}(\varphi^{-1}(U))\]
est un isomorphisme.
\end{theo}

\begin{coro}
Supposons que tout point~$b$ de~$B$ dont le corps r\'esiduel~$\Hs(b)$ est de caract\'eristique non nulle et trivialement valu\'e est ultram\'etrique typique. Supposons que le faisceau~$\Os_{X}$ est coh\'erent. Alors, pour toute partie~$U$ de~$X$ et tout faisceau coh\'erent~$\Fs$ sur~$\varphi^{-1}(U)$, le faisceau~$(\varphi_{U})_{*}\Fs$ est coh\'erent.
\end{coro}

Remarquons que le r\'esultat du th\'eor\`eme~\ref{theo:finiPS-T} s'\'etend en dimension sup\'erieure.

\begin{coro}\label{cor:finidimsup}
Supposons que tout point~$b$ de~$B$ dont le corps r\'esiduel~$\Hs(b)$ est de caract\'eristique non nulle et trivialement valu\'e est ultram\'etrique typique. Soit~$n\in\N^*$. Pour tout~$m\in \cn{1}{n}$, soit $P_{m} \in \As[T_{1},\dots,T_{m}]$ un polyn\^ome qui, vu comme polyn\^ome en la variable~$T_{m}$ n'est pas constant et poss\`ede un coefficient dominant inversible. Consid\'erons l'endomorphisme de $\As[T_{1},\dots,T_{n}]$ qui, pour tout $m\in\cn{1}{n}$, envoie~$T_{m}$ sur~$P_{m}$. Notons $\psi : \E{n}{\As} \to \E{n}{\As}$ le morphisme entre espaces analytiques associ\'e.

Alors, pour toute partie~$U$ de~$\E{n}{\As}$, le morphisme naturel
\[\Os_{\E{n}{\As}}(U)[S_{1},\dots,S_{n}]/(P_{1}(S_{1})-T_{1},\dots,P_{n}(S_{1},\dots,S_{n})-T_{n}) \to \Os_{\E{n}{\As}}(\varphi^{-1}(U))\]
est un isomorphisme.
\end{coro}

\begin{rema}
En nous pla\c{c}ant au voisinage d'un point de~$\E{n}{\As}$, nous pourrions rel\^acher les conditions sur les polyn\^omes~$P_{m}$.
\end{rema}

\section{Premi\`eres propri\'et\'es des anneaux locaux}

Dans cette partie, nous revenons \`a notre objectif initial : l'\'etude des anneaux locaux des espaces affines analytiques sur~$\Z$ et les anneaux d'entiers de corps de nombres. Nous d\'emontrons ici qu'ils sont noeth\'eriens et r\'eguliers. Les preuves que nous proposons valent en fait pour les anneaux locaux des espaces~$\E{n}{\As}$, o\`u~$\As$ appartient \`a une classe d'anneaux de Banach plus g\'en\'erale, qui contient notamment les corps valu\'es complets (archim\'ediens ou non) et les anneaux de valuation discr\`ete.

Bien entendu, ces r\'esultats sont d\'ej\`a connus pour les espaces de Berkovich sur un corps valu\'e ultram\'etrique complet (\textit{cf.}~\cite{bleu}, th\'eor\`emes~2.1.4 et~2.2.1). Le cas d'un point rationnel se d\'emontre par des m\'ethodes classiques de g\'eom\'etrie rigide (\textit{cf.}~\cite{BGR}, proposition~7.3.2/7), qui reposent \textit{in fine} sur un th\'eor\`eme de division de Weierstra{\ss} global : il vaut pour les alg\`ebres de Tate, qui sont des alg\`ebres de fonctions sur des disques de rayon strictement positif, et se d\'emontre par des arguments de r\'eduction (\textit{cf.}~\cite{BGR}, th\'eor\`eme~5.2.1/2). Pour traiter le cas d'un point~$x$ quelconque, V.~Berkovich effectue le changement de base $k \to \Hs(x)$, op\'eration qui rend le point~$x$ rationnel. Un r\'esultat de L.~Gruson (\textit{cf.}~\cite{Gruson}) montre que ce changement de base est fid\`element plat et permet de conclure. Il semble difficile d'adapter ces arguments pour les espaces de Berkovich globaux (en particulier le dernier, puisque le changement de base $\As \to \Hs(x)$ n'est pas plat en g\'en\'eral). C'est pourquoi nous avons choisi de mettre en {\oe}uvre une m\'ethode purement locale, bas\'ee sur les th\'eor\`emes de Weierstra{\ss} d\'emontr\'es \`a la section pr\'ec\'edente. Notre m\'ethode pr\'esente, de surcro\^{\i}t, l'avantage d'unifier les approches complexe et $p$-adique.

\medskip

Nous d\'emontrerons les r\'esultats de noeth\'erianit\'e et r\'egularit\'e annonc\'es en supposant que les anneaux locaux de l'espace de base~$B$ sont assez simples. Pr\'ecisons.

\begin{defi}
Soit~$C$ une partie compacte de~$B$. Soit~$\Us$ un syst\`eme fondamental de voisinages compacts et spectralement convexes de~$C$. On dit qu'un id\'eal~$I$ de~$\Bs(C)^\dag$ est \textbf{$\Bs$-fortement de type fini} relativement \`a~$\Us$ s'il existe des \'el\'ements $f_{1},\ldots, f_{p}$ de~$I$ qui sont $\Bs$-d\'efinis sur tous les \'el\'ements de~$\Us$ et v\'erifient la propri\'et\'e suivante : pour tout voisinage compact~$V$ de~$C$, il existe une famille  de nombres r\'eels strictement positifs $(K_{U,V})_{U\in\Us}$ telle que, pour tout \'el\'ement~$f$ de~$I$ $\Bs$-d\'efini sur~$V$ et tout \'el\'ement~$U$ de~$\Us$ contenu dans~$\overset{\circ}{V}$, il existe des \'el\'ements $a_{1},\ldots,a_{p}$ de~$\Bs(U)$ tels que
\[\begin{cases}
f = a_{1}f_{1} + \dots +a_{p} f_{p} \textrm{ dans } \Bs(U)\ ;\\
\forall i \in \cn{1}{p}, \|a_{i}\|_{U} \le K_{U,V} \|f\|_{V}.
\end{cases}\]
Une famille $(f_{1},\ldots,f_{p})$ v\'erifiant les propri\'et\'es pr\'ec\'edentes est appel\'ee \textbf{\mbox{$\Bs$-syst\`eme} de g\'en\'erateurs fort} de l'id\'eal~$I$ relativement \`a~$\Us$. On dit \'egalement qu'elle \textbf{engendre $\Bs$-fortement} l'id\'eal~$I$ relativement \`a~$\Us$. 
\end{defi}

\begin{defi}
Soit~$C$ une partie de~$B$. Nous dirons qu'un syst\`eme fondamental de voisinages~$\Us$ de~$C$ est \textbf{fin} s'il contient un syst\`eme fondamental de voisinages de tous ses \'el\'ements.
\end{defi}

\begin{defi}
Soit~$b$ un point de~$B$. Nous dirons qu'un anneau local noeth\'erien~$\Os_{B,b}$ de dimension~$n$ est \textbf{fortement r\'egulier} s'il existe des \'el\'ements $f_{1},\dots,f_{n}$ de~$\m_{b}$ et un syst\`eme fondamental fin~$\Us$ de voisinages compacts et spectralement convexes de~$b$ dans~$B$ tels que la famille $(f_{1},\dots,f_{n})$ engendre $\Bs$-fortement l'id\'eal maximal~$\m_{b}$ relativement \`a~$\Us$.

Dans ces conditions, nous dirons que l'anneau local~$\Os_{B,b}$ est un \textbf{corps fort} (resp. un \textbf{anneau fortement de valuation discr\`ete}) si $n=0$ (resp. $n=1$). 
\end{defi}

\begin{rema}
Un corps fort est un corps et la condition suppl\'ementaire que nous imposons s'apparente au principe du prolongement analytique. Remarquons que, si~$B$ est localement connexe et si le principe du prolongement analytique y vaut, alors tout anneau local~$\Os_{B,b}$ qui est un corps est un corps fort.

Un anneau fortement de valuation discr\`ete est un anneau de valuation discr\`ete et la condition suppl\'ementaire est l'analogue de la condition~$(U)$ introduite dans~\cite{asterisque}, d\'efinition~2.2.9. Remarquons que si la condition est v\'erifi\'ee pour une uniformisante, alors elle l'est pour toutes. Cette condition nous semble naturelle \`a imposer dans un cadre analytique : elle permet, par exemple, de montrer qu'une s\'erie dont tous les coefficients sont multiples d'une uniformisante~$\pi$ est elle-m\^eme multiple de~$\pi$.
\end{rema}

\begin{defi}\label{defi:debase}
Un \textbf{anneau de Banach}~$(\As,\|.\|)$ est dit \textbf{de base} si, pour tout point~$b$ de son spectre $B=\Mc(\As)$, il existe un syst\`eme fondamental fin de voisinages compacts et spectralement convexes~$\Us_{b}$ v\'erifiant les propri\'et\'es suivantes :
\begin{enumerate}[i)]
\item l'anneau local~$\Os_{B,b}$ est un anneau local noeth\'erien, fortement r\'egulier relativement \`a~$\Us_{b}$ et de dimension inf\'erieure \`a~$1$, c'est-\`a-dire un corps fort ou un anneau fortement de valuation discr\`ete ;
\item si $\Hs(b)$ est de caract\'eristique non nulle et trivialement valu\'e, tout \'el\'ement de~$\Us_{b}$ est contenu dans~$B_{\textrm{um}}$ et poss\`ede un bord analytique fini (ce qui entra\^{\i}ne que le point~$b$ est ultram\'etrique typique).
\end{enumerate}
\end{defi}

\begin{rema}\label{rem:liste}
Dans la suite, nous allons nous int\'eresser aux points de~$\E{n}{\As}$ sur un anneau de Banach de base au sens de la d\'efinition qui pr\'ec\`ede. Indiquons quelques exemples importants d'anneaux de Banach $(\As,\|.\|)$ qui v\'erifient cette propri\'et\'e et o\`u nos r\'esultats valent donc inconditionnellement : 
\begin{enumerate}[i)]
\item un corps~$k$ muni d'une valeur absolue (archim\'edienne ou non) pour laquelle il est complet ;
\item un corps~$k$ muni de la norme $\max(|.|_{0},|.|)$, o\`u~$|.|$ est une valeur absolue sur~$k$ (\textit{cf.}~\cite{BerkovichW0});
\item un anneau de valuation discr\`ete muni d'une norme associ\'ee \`a la valuation ;
\item un anneau d'entiers de corps de nombres~$A$ muni de la norme~$\max_{\sigma}(|\sigma(.)|_{\infty})$, o\`u~$\sigma$ d\'ecrit l'ensemble des plongements complexes de~$K=\Frac(A)$ (\textit{cf.}~\cite{asterisque}, \S~3.1), par exemple l'anneau~$\Z$ muni de la valeur absolue usuelle~$|.|_{\infty}$.
\end{enumerate} 
\end{rema}

Dans un premier temps, nous allons \'etudier la propri\'et\'e de g\'en\'eration forte et montrer qu'elle passe des anneaux locaux de~$B=\Mc(\As)$ \`a ceux de~$X=\E{1}{\As}$. Nous commencerons par quelques r\'esultats sur les fonctions d\'efinies au voisinage de couronnes.

\begin{prop}\label{prop:fortcouronne}
Soient~$b$ un point de~$B$ et~$\Us$ un syst\`eme fondamental fin de voisinages compacts et spectralement convexes de~$b$. Supposons que l'id\'eal~$\m_{b}$ de~$\Os_{B,b}$ poss\`ede un $\Bs$-syst\`eme de g\'en\'erateurs forts $(f_{1},\ldots,f_{p})$ relativement \`a~$\Us$. 

Soient~$r$ et~$s$ deux nombres r\'eels v\'erifiant $0\le r\le s$ et $s>0$. Notons~$\Vs$ l'ensemble des couronnes compactes de la forme $\overline{C}_{U}(u,v)$, avec $U\in\Us$ et \mbox{$0\le u \prec r\le s < v$}. C'est un syst\`eme fondamental fin de voisinages compacts et spectralement convexes de~$\overline{C}_{b}(r,s)$ et la famille $(f_{1},\ldots,f_{p})$ engendre $\Bs$-fortement relativement \`a~$\Vs$ l'id\'eal~$I$ de $\Bs(\overline{C}_{b}(r,s))^\dag$ form\'e des \'el\'ements qui s'annulent sur $\overline{C}_{b}(r,s)$. En outre, si $n=1$ et si $\bigcap_{m\ge 0} \m_{b}^m = (0)$, alors $\bigcap_{m\ge 0} I^m = (0)$.
\end{prop}
\begin{proof}
Soient~$V$ un voisinage compact de~$\overline{C}_{b}(r,s)$, $f$~un \'el\'ement de~$I$ qui est $\Bs$-d\'efini sur~$V$ et $\overline{C}_{U}(u,v)$ un \'el\'ement de~$\Vs$ contenu dans~$\overset{\circ}{V}$. L'int\'erieur du voisinage~$V$ contient une partie de la forme $\overline{C}_{U_{0}}(u_{0},v_{0})$, o\`u~$U_{0}$ est un voisinage compact et spectralement convexe de~$U$ dans~$B$ et $0\le u_{0} \prec u\le v < v_{0}$. D'apr\`es la proposition~\ref{prop:Bcouronne}, nous pouvons \'ecrire~$f$ sous la forme $\sum_{k\in\Z} a_{k}\, T^k$, o\`u~$(a_{k})_{k\in\Z}$ est une famille d'\'el\'ements de~$\Bs(U_{0})$ telle que la famille $\big(\|a_{k}\|_{U_{0}} \max(u_{0}^k,v_{0}^k)\big)_{k\in\Z}$ est sommable. 

Puisque~$\Us$ est fin, nous pouvons choisir un \'el\'ement~$U_{1}$ de~$\Us$ tel que \mbox{$\overset{\circ}{U_{0}} \supset U_{1} \supset \overset{\circ}{U_{1}} \supset U$}. Puisque~$f$ est un \'el\'ement de~$I$, pour tout $k\in\Z$, $a_{k}$ s'annule en~$b$. Par cons\'equent, il existe des \'el\'ements $\alpha_{1,k},\ldots,\alpha_{p,k}$ de~$\Bs(U_{1})$ tels que
\[\begin{cases}
a_{k} = \alpha_{1,k} f_{1} + \dots \alpha_{p,k} f_{p} \textrm{ dans } \Bs(U_{1})\ ;\\
\forall i \in \cn{1}{p}, \|\alpha_{i,k}\|_{U_{1}} \le K_{U_{1},U_{0}} \|a_{k}\|_{U_{0}}.
\end{cases}\]

Par cons\'equent, pour tout $i \in \cn{1}{p}$, la famille $(\|\alpha_{i,k}\|_{U_{1}} \max(u_{0}^k,v_{0}^k))_{k\in\Z}$ est sommable et la s\'erie $A_{i} = \sum_{k\in\Z} \alpha_{i,k}\, T^k$ d\'efinit donc un \'el\'ement de~$\Bs(\overline{C}_{U}(u,v))$. De plus, nous avons
\[f = \sum_{i=1}^p A_{i} f_{i} \textrm{ dans } \Bs(\overline{C}_{U}(u,v))\]
et, pour tout $i \in \cn{1}{p}$,
\[\renewcommand{\arraystretch}{1.5}{\begin{array}{ccc}
\|A_{i}\|_{\overline{C}_{U}(u,v)} &\le& \disp K_{U_{1},U_{0}} \sum_{k\in\Z}  \|a_{k}\|_{U_{1}} \max(u_{0}^k,v_{0}^k)\\
& \le&\disp K_{U_{1},U_{0}} \left(\frac{u_{0}}{u-u_{0}} + \frac{v_{0}}{v_{0}-v}\right) \|f\|_{\overline{C}_{U_{1}}(r,s)}\\
& \le & \disp K_{U_{1},U_{0}} \left(\frac{u_{0}}{u-u_{0}} + \frac{v_{0}}{v_{0}-v}\right) \|f\|_{V}, 
\end{array}}\]
d'apr\`es la proposition~\ref{prop:comparaisonnormes}.

Supposons maintenant que~$p=1$ et $\bigcap_{m\ge 0} \m_{b}^m = (0)$. Si $f \in \bigcap_{m\ge 0} I^m$, alors, pour tout~$k\in\Z$, $a_{k} \in \bigcap_{m\ge 0} (f_{1}^m)$, d'o\`u nous d\'eduisons que $f=0$.
\end{proof}

\begin{coro}\label{cor:fortlemniscate}
Soient~$b$ un point de~$B$ et~$U_{0}$ un voisinage compact de~$b$. Soient $P$ un polyn\^ome unitaire \`a coefficients dans~$\Bs(U_{0})$ et~$r$ et~$s$ deux nombres r\'eels v\'erifiant $0\le r\le s$ et $s>0$. Soit~$\Us$ un syst\`eme fondamental fin de voisinages compacts et spectralement convexes de~$b$ dans~$U_{0}$ et $u$, $v$ deux nombres r\'eels tels que $0\le u\prec r \le s < v$. Supposons que, pour tout \'el\'ement~$U$ de~$\Us$ et pour tous $u'$, $v'$ v\'erifiant $u\prec u' \prec r$ et $s < v' < v$, la couronne $\overline{C}_{U}(u',v')$ satisfait la condition~$(N_{P(S)-T})$. Notons~$\Vs$ l'ensemble des couronnes compactes de la forme $\overline{C}_{U}(P;u',v')$, avec $U\in\Us$, $u\prec u' \prec r$ et $s < v' < v$.  

Supposons que l'id\'eal~$\m_{b}$ poss\`ede un $\Bs$-syst\`eme de g\'en\'erateurs forts $(f_{1},\ldots,f_{p})$ relativement \`a~$\Us$. Alors la famille $(f_{1},\ldots,f_{p})$ est un $\Bs$-syst\`eme de g\'en\'erateurs fort relativement \`a~$\Vs$ pour l'id\'eal~$I$ de $\Bs(\overline{C}_{b}(P;r,s))^\dag$ form\'e des \'el\'ements qui s'annulent sur $\overline{C}_{b}(P;r,s)$. En outre, si $p=1$ et si $\bigcap_{m\ge 0} \m_{b}^m = (0)$, alors $\bigcap_{m\ge 0} I^m = (0)$.
\end{coro}
\begin{proof}
On utilise le corollaire pr\'ec\'edent et l'isomorphisme admissible d\'emontr\'e \`a la proposition~\ref{prop:Blemniscate}.
\end{proof}

Introduisons encore un peu de vocabulaire.

\begin{defi}
Soit~$b$ un point de~$B$. Soit~$x$ un point de~$X_{b}$. Nous dirons que le point~$x$ est \textbf{transcendant} si l'extension $\Hs(x)/\Hs(b)$ est transcendante. Nous dirons que le point~$x$ est \textbf{localement transcendant} si l'extension $\kappa(x)/\kappa(b)$ est transcendante. 

Soit~$x$ un point de~$X_{n}$, avec $n\in\N$, au-dessus de~$b$. Nous dirons que le point~$x$ est \textbf{purement localement transcendant} si les \'el\'ements $1,T_{1}(x),\dots,T_{n}(x)$ de~$\Hs(x)$ sont alg\'ebriquement ind\'ependants sur~$\kappa(b)$.
\end{defi}

\begin{rema}
Lorsque~$n=1$, un point~$x$ est localement transcendant si, et seulement si, il n'est pas rigide \'epais, c'est-\`a-dire si et seulement s'il est de type~2, 3, 4, de type~1 transcendant ou rigide maigre.

Revenons \`a un entier~$n\in\N$ quelconque. Pour $m\in\cn{0}{n}$, notons~$x_{m}\in X_{m}$, la projection de~$x$ sur ses $m$~premi\`eres coordonn\'ees. Dans ce cas, le point~$x$ est purement localement transcendant si, pour tout $m\in\cn{1}{n}$, le point~$x_{m}$ est localement transcendant sur~$x_{m-1}$. Cela ne d\'epend pas des coordonn\'ees choisies.
\end{rema}

\begin{coro}\label{cor:fortpasepais}
Soient~$b$ un point de~$B$ et~$x$ un point de~$X_{n}$, avec $n\in\N^*$, au-dessus de~$b$. Soit~$\Us$ un syst\`eme fondamental fin de voisinages compacts et spectralement convexes du point~$b$. Si~$\Hs(b)$ est de caract\'eristique non nulle et trivialement valu\'e, supposons que tout \'el\'ement de~$\Us$ est contenu dans~$B_{\textrm{um}}$ et poss\`ede un bord analytique fini. 

Supposons que l'id\'eal~$\m_{b}$ poss\`ede un $\Bs$-syst\`eme de g\'en\'erateurs forts $(f_{1},\ldots,f_{p})$ relativement \`a~$\Us$ et que le point~$x$ est purement localement transcendant. Alors il existe un syst\`eme fondamental fin de voisinages compacts et spectralement convexes~$\Vs$ du point~$x$ tel que $(f_{1},\ldots,f_{p})$ soit \'egalement un \mbox{$\Bs$-syst\`eme} de g\'en\'erateurs fort pour~$\m_{x}$ relativement \`a~$\Vs$. En outre, si $p=1$ et si $\bigcap_{m\ge 0} \m_{b}^m = (0)$, alors $\bigcap_{m\ge 0} \m_{x}^m = (0)$.

En particulier, si~$\Os_{B,b}$ est un corps fort, alors~$\Os_{X_{n},x}$ est un corps fort et si~$\Os_{B,b}$ est un anneau fortement de valuation discr\`ete d'uniformisante~$\pi$, alors~$\Os_{X_{n},x}$ est \'egalement un anneau fortement de valuation discr\`ete d'uniformisante~$\pi$.

Si~$\Hs(x)$ est de caract\'eristique non nulle et trivialement valu\'e, nous pouvons choisir~$\Vs$ de sorte que tout \'el\'ement de~$\Vs$ soit contenu dans~$(X_{n})_{\textrm{um}}$ et poss\`ede un bord analytique fini.
\end{coro}
\begin{proof}
Une r\'ecurrence permet de se ramener au cas o\`u $n=1$. Nous allons traiter le cas o\`u le point~$x$ est un point de type~2 ou~3 ou~4. Le cas o\`u c'est un point de type~1 se traite de fa\c{c}on similaire, en rempla\c{c}ant, dans le raisonnement qui suit, les couronnes par des disques.

Soit~$\Us$ un syst\`eme fondamental fin de voisinages compacts et spectralement convexes du point~$b$ par rapport auquel la famille $(f_{1},\dots,f_{p})$ engendre fortement~$\m_{b}$. Soit~$W$ un voisinage ouvert du point~$x$ dans~$X$. Il existe un polyn\^ome unitaire $P_{0}(S) \in \Hs(b)[S]$ tel que $0 < r < |P_{0}(S)(x)| < s$ et $\overline{C}_{b}(P_{0};r,s)$ soit une partie connexe de~$W$. Nous allons distinguer trois cas, selon les propri\'et\'es du corps r\'esiduel~$\Hs(b)$.

$\bullet$ Supposons que la valuation de~$\Hs(b)$ n'est pas triviale. Nous pouvons approcher les coefficients de~$P_{0}$ par des \'el\'ements de~$\Os_{B,b}$ de fa\c{c}on \`a obtenir un polyn\^ome unitaire $P(S) \in \Os_{B,b}[S]$ satisfaisant encore les propri\'et\'es \'enonc\'ees. Nous pouvons, en outre, supposer que $P(b)(S)\in\Hs(b)[S]$ est s\'eparable. D'apr\`es le corollaire~\ref{cor:projouverte}, le morphisme de projection~$\pi$ est ouvert. On en d\'eduit qu'il existe un voisinage ouvert~$U_{0}$ de~$b$ dans~$B$ sur lequel les coefficients de~$P$ sont d\'efinis et tel que $\overline{C}_{U_{0}}(P;r,s) \subset W$. 

D'apr\`es le lemme~\ref{lem:lemniscatesNPS-T}, il existe un voisinage ouvert~$U_{1}$ de~$b$ dans~$U_{0}$ et des nombres r\'eels $r_{1},r_{2},s_{1},s_{2}$ v\'erifiant $0 < r \le r_{2} < r_{1} <  |P(S)(x)| < s_{1} < s_{2} \le s$ tels que, pour tout voisinage compact et spectralement convexe~$U$ de~$b$ dans~$U_{1}$, tout $u\in\of{[}{r_{2},r_{1}}{]}$ et tout $v\in\of{[}{s_{1},s_{2}}{]}$, le disque $\overline{C}_{U}(u,v)$ satisfasse la condition~$(N_{P(S)-T})$. Remarquons que, puisque~$\overline{C}_{b}(P;r,s)$ est connexe, pour tous $u,v$ v\'erifiant $u < |P(S)(x)| < v$, la lemniscate $\overline{C}_{b}(P;u,v)$ est encore connexe. 

Notons~$\Vs_{W}$ l'ensemble des parties de la forme $\overline{C}_{U}(P;r,s)$ avec $U\in\Us$, $U\subset U_{1}$, $u\in\of{]}{r_{1},r_{2}}{[}$ et $v\in\of{]}{s_{1},s_{2}}{[}$. On v\'erifie qu'il est fin. Notons~$\Vs$ la r\'eunion des ensembles~$\Vs_{W}$, pour~$W$ d\'ecrivant l'ensemble des voisinages ouvert de~$x$ dans~$X$. C'est un syst\`eme fondamental fin de voisinages compacts et spectralement convexes de~$x$ dans~$X$.

Soient~$V$ un voisinage compact de~$x$ dans~$X$ et $\overline{C}_{U}(P;u,v)$ un \'el\'ement de~$\Vs$ contenu dans~$\overset{\circ}{V}$. Soit~$f$ un \'el\'ement de~$\Bs(V)$ tel que $f(x)=0$. D'apr\`es le corollaire~\ref{cor:pointpasepais}, l'image de~$f$ dans~$\Os_{X_{b},x}$ est nulle. Par le principe du prolongement analytique, la fonction~$f$ est nulle sur $\overline{C}_{b}(P;u,v)$, car $\overline{C}_{b}(P;u,v)$ est connexe. On conclut alors \`a l'aide du corollaire~\ref{cor:fortlemniscate}.

$\bullet$ Supposons que le corps~$\Hs(b)$ est de caract\'eristique nulle. Nous pouvons alors supposer que le polyn\^ome $P_{0}(S) \in \Hs(b)[S]$ est s\'eparable. Il est alors possible de trouver un polyn\^ome $P(S) \in \Os_{B,b}[S]$ satisfaisant les m\^emes propri\'et\'es que pr\'ec\'edemment et, en particulier, le fait que $P(b)(S)\in\Hs(b)[S]$ soit s\'eparable. Le raisonnement se poursuit identiquement.

$\bullet$ Supposons que le corps~$\Hs(b)$ est de caract\'eristique non nulle et trivialement valu\'e. Nous pouvons alors relever le polyn\^ome $P_{0}(S) \in \Hs(b)[S]$ en un polyn\^ome $P(S) \in \Os_{B,b}[S]$.  
Par hypoth\`ese, tout \'el\'ement de~$\Us$ est contenu dans~$B_{\textrm{um}}$ et poss\`ede un bord analytique fini. La proposition~\ref{prop:couronneum} permet alors d'adapter le raisonnement pr\'ec\'edent.
\end{proof}

Int\'eressons-nous maintenant aux points rigides \'epais. Le raisonnement est similaire, mais un peu plus simple.

\begin{prop}\label{prop:fortdisque}
Soient~$b$ un point de~$B$ et~$\Us$ un syst\`eme fondamental fin de voisinages compacts et spectralement convexes de~$b$. Supposons que l'id\'eal~$\m_{b}$ de~$\Os_{B,b}$ poss\`ede un $\Bs$-syst\`eme de g\'en\'erateurs forts $(f_{1},\ldots,f_{p})$ relativement \`a~$\Us$. Notons~$x$ le point~$0$ de la fibre~$X_{b}$. 

Notons~$\Vs$ l'ensemble des disques compacts de la forme $\overline{D}_{U}(v)$, avec $U\in\Us$ et $v >0$. C'est un syst\`eme fondamental fin de voisinages compacts et spectralement convexes du point~$x$ et la famille $(f_{1},\ldots,f_{p},T)$ engendre $\Bs$-fortement relativement \`a~$\Vs$ l'id\'eal maximal~$\m_{x}$ de~$\Os_{X,x}$. 

Si tout \'el\'ement de~$\Us$ est contenu dans~$B_{\textrm{um}}$ et poss\`ede un bord analytique fini, alors tout \'el\'ement de~$\Vs$ est contenu dans~$X_{\textrm{um}}$ et poss\`ede un bord analytique fini.
\end{prop}
\begin{proof}
Soient~$V$ un voisinage compact de~$x$, $f$~un \'el\'ement de~$\m_{x}$ $\Bs$-d\'efini sur~$V$ et $\overline{D}_{U}(v)$ un \'el\'ement de~$\Vs$ contenu dans~$\overset{\circ}{V}$. L'int\'erieur du voisinage~$V$ contient une partie de la forme $\overline{D}_{U}(v_{0})$ avec $v_{0}>v$. D'apr\`es la proposition~\ref{prop:Bcouronne}, nous pouvons \'ecrire~$f$ sous la forme $\sum_{k\ge 0} a_{k}\, T^k$, o\`u~$(a_{k})_{k\in\Z}$ est une famille d'\'el\'ements de~$\Bs(U)$ telle que la s\'erie $\sum_{k\ge 0} \|a_{k}\|_{U}\, v_{0}^k$ converge. 

Puisque~$f$ est un \'el\'ement de~$\m_{x}$, $a_{0}$ s'annule en~$b$. Par cons\'equent, il existe des \'el\'ements $\alpha_{1},\ldots,\alpha_{p}$ de~$\Bs(U)$ tels que
\[\begin{cases}
a_{0} = \alpha_{1} f_{1} + \dots \alpha_{p} f_{p} \textrm{ dans } \Bs(U)\ ;\\
\forall i \in \cn{1}{p}, \|\alpha_{i}\|_{U} \le K_{U,V} \|a_{0}\|_{V}.
\end{cases}\]

Or, d'apr\`es la proposition~\ref{prop:comparaisonnormes}, nous avons
\[ \|a_{0}\|_{U} \le \|f\|_{U,v} \le  \frac{v_{0}}{v_{0}-v}\, \|f\|_{\overline{D}_{U}(v_{0})} \le \frac{v_{0}}{v_{0}-v}\, \|f\|_{V} .\]
Ce raisonnement nous permet de remplacer~$f$ par $f-a_{0}$, et donc de supposer que $a_{0}=0$.

La s\'erie $\sum_{k\ge 0} \|a_{k+1}\|_{U}\, v^k$ converge et la s\'erie $\sum_{k\ge 0} a_{k+1}\, T^k$ d\'efinit donc un \'el\'ement~$g$ de~$\Bs(\overline{D}_{U}(v))$. De plus, nous avons $f=Tg$ dans $\Bs(\overline{D}_{U}(v))$ et
\[ \|g\|_{\overline{D}_{U}(v)} \le \frac{1}{v}\, \|f\|_{\overline{D}_{U}(v)} \le  \frac{1}{v}\, \|f\|_{V},\]
car le maximum de la fonction~$g$ sur toute fibre est atteint sur le cercle $\{|T|=v\}$.
On en d\'eduit le r\'esultat voulu.

L'assertion finale de l'\'enonc\'e d\'ecoule du lemme~\ref{lem:Bsum}.
\end{proof}

En reprenant, \`a pr\'esent, le raisonnement des corollaires~\ref{cor:fortlemniscate} et~\ref{cor:fortpasepais}, nous obtenons le r\'esultat suivant.

\begin{coro}\label{cor:fortepais}
Soit~$b$ un point de~$B$. Soit~$\Us$ un syst\`eme fondamental fin de voisinages compacts et spectralement convexes du point~$b$. Si~$\Hs(b)$ est de caract\'eristique non nulle et trivialement valu\'e, supposons que tout \'el\'ement de~$\Us$ est contenu dans~$B_{\textrm{um}}$ et poss\`ede un bord analytique fini. 

Supposons que l'id\'eal~$\m_{b}$ poss\`ede un $\Bs$-syst\`eme de g\'en\'erateurs forts $(f_{1},\ldots,f_{p})$ relativement \`a~$\Us$. Soit~$x$ un point rigide \'epais de~$X_{b}$. Alors il existe un syst\`eme fondamental fin de voisinages compacts et spectralement convexes~$\Vs$ du point~$x$ tel que la famille $(f_{1},\ldots,f_{p},P(T))$, o\`u $P(T)\in \Os_{B,b}[T]$ est un polyn\^ome unitaire relevant le polyn\^ome minimal du point~$x$, soit un $\Bs$-syst\`eme de g\'en\'erateurs fort pour~$\m_{x}$ relativement \`a~$\Vs$. Dans le cas o\`u~$\Hs(b)$ est de caract\'eristique non nulle et trivialement valu\'e, nous pouvons choisir~$\Vs$ de sorte que tout \'el\'ement de~$\Vs$ soit contenu dans~$X_{\textrm{um}}$ et poss\`ede un bord analytique fini. 

Soit~$n\in\N$. Soit~$x$ un point rigide \'epais de~$X_{n}$ au-dessus de~$b$. Alors il existe un syst\`eme fondamental fin de voisinages compacts et spectralement convexes~$\Vs$ du point~$x$ et des \'el\'ements $g_{1},\dots,g_{n}$ de~$\m_{x}$ tels que la famille $(f_{1},\dots,f_{p},g_{1},\dots,g_{n})$ engendre $\Bs$-fortement l'id\'eal~$\m_{x}$ relativement \`a~$\Vs$. Dans le cas o\`u~$\Hs(b)$ est de caract\'eristique non nulle et trivialement valu\'e, nous pouvons choisir~$\Vs$ de sorte que tout \'el\'ement de~$\Vs$ soit contenu dans~$(X_{n})_{\textrm{um}}$ et poss\`ede un bord analytique fini. 
\end{coro}

Approfondissons maintenant notre \'etude des points rigides \'epais.

\begin{lemm}\label{lem:divisionparpi}
Soient~$b$ un point de $B=\Mc(\As)$. Soient~$n\in\N$ et~$x$ un point rigide \'epais de $X_{n} = \E{n}{\As}$ au-dessus de~$b$. Supposons que l'anneau local~$\Os_{B,b}$ est un anneau fortement de valuation discr\`ete d'uniformisante~$\pi$. 

Alors, pour tout \'el\'ement non nul~$f$ de~$\Os_{X_{n},x}$, il existe un unique couple $(v,g)$ dans $\N\times \Os_{X_{n},x}$ v\'erifiant les propri\'et\'es suivantes :
\begin{enumerate}[i)]
\item $f = \pi^v g$ dans $\Os_{X_{n},x}$ ;
\item $g\ne 0$ dans $\Os_{(X_{n})_{b},x}$.
\end{enumerate}
\end{lemm}
\begin{proof}
Le r\'esultat se d\'emontre par r\'ecurrence sur la dimension~$n$. Pour $n=0$, c'est \'evident.

Supposons maintenant que $n\ge 1$ et que le r\'esultat est v\'erifi\'e en dimension~$n-1$. Notons $y\in X_{n-1}$ la projection de~$x$ sur ses $n-1$ premi\`eres coordonn\'ees. Notons~$x'$ le point~$0$ de la fibre de~$X_{n}$ au-dessus de~$y$. Soit $P(S) \in \kappa(y)[S]$ le polyn\^ome minimal du point rigide \'epais~$x$ au-dessus de~$y$. Relevons-le en un polyn\^ome unitaire \`a coefficients dans~$\Os_{X_{n-1},y}$ que nous noterons identiquement. D'apr\`es le th\'eor\`eme~\ref{theo:finiPS-T}, nous disposons d'un isomorphisme naturel $\Os_{X_{n},x'}[S]/(P(S)-T) \simeq \Os_{X_{n},x}$. Il suffit donc de d\'emontrer le r\'esultat pour le point~$x'$. Or tout \'el\'ement de~$\Os_{X_{n},x'}$ poss\`ede un d\'eveloppement en s\'erie convergente \`a coefficients dans~$\Os_{X_{n-1},y}$ (\textit{cf.}~corollaire~\ref{cor:dvptpoint}), ce qui permet de conclure par r\'ecurrence. (Le fait que l'anneau soit fortement de valuation discr\`ete est utilis\'e pour montrer qu'en divisant par la puissance de~$\pi$ choisie, on obtient encore une s\'erie convergente.)
\end{proof}

\begin{lemm}\label{lem:chgtvar}
Soient~$b$ un point de $B=\Mc(\As)$. Soient~$n\in\N$ et~$x$ un point rigide de $X_{n} = \E{n}{\As}$ (avec variables $T_{1},\dots,T_{n}$) au-dessus de~$b$. Soit~$f$ un \'el\'ement de~$\Os_{X_{n},x}$ dont l'image dans~$\Os_{(X_{n})_{b},x}$ n'est pas nulle. 

Alors il existe un changement de variables de la forme 
\[\left\{\begin{array}{rcl}
T_{1} &\mapsto& T_{1} + T_{n}^{u_{1}},\\
&\vdots&\\ 
T_{n-1}&\mapsto& T_{n-1} + T_{n}^{u_{n-1}},\\
T_{n} &\mapsto& T_{n},
\end{array}\right.\]
avec $u_{1},\dots,u_{n-1} \in \N$ tel que, apr\`es ce changement de variables, si l'on note $y\in X_{n-1}$ le projet\'e de~$x$ sur ses $n-1$ premi\`eres coordonn\'ees, l'image de~$f$ dans~$\Os_{(X_{n})_{y},x}$ ne soit pas nulle. 
\end{lemm}
\begin{proof}
Remarquons que le r\'esultat concerne la restriction de la fonction~$f$ \`a~$X_{b}$. Nous pouvons donc supposer que~$\As$ est un corps~$K$. Remarquons encore qu'il suffit de d\'emontrer le r\'esultat apr\`es extension des scalaires. Nous pouvons donc supposer que~$K$ est alg\'ebriquement clos. Le point~$x$ est alors un point $(\alpha_{1},\dots,\alpha_{n})$ de~$K^n$ et, d'apr\`es le corollaire~\ref{cor:dvptpoint}, la fonction~$f$ s'\'ecrit comme une s\'erie 
\[f = \sum_{(k_{1},\dots,k_{n})\in\N^n} a_{k_{1},\dots,k_{n}} (T_{1}-\alpha_{1})^{k_{1}}\dots(T_{n}-\alpha_{n})^{k_{n}}\] 
dont le rayon de convergence est strictement positif. Par hypoth\`ese, l'un des coefficients~$a_{k_{1},\dots,k_{n}}$ n'est pas nul. Soient $u_{1},\dots,u_{n-1} \in \N$. Montrer que le changement de base de l'\'enonc\'e convient revient \`a montrer que la s\'erie
\[\sum_{(k_{1},\dots,k_{n})\in\N^n} a_{k_{1},\dots,k_{n}} (\alpha_{n}^{u_{1}} - T_{n}^{u_{1}})^{k_{1}}\dots(\alpha_{n}^{u_{n-1}} - T_{n}^{u_{n-1}})^{k_{n-1}}(T_{n}-\alpha_{n})^{k_{n}}\] 
n'est pas nulle. En s'armant de courage, on d\'emontre le r\'esultat recherch\'e.
\end{proof}

\begin{rema}
Les fastidieux calculs finaux de la preuve pr\'ec\'edente peuvent \^etre \'evit\'es si l'on s'autorise \`a utiliser les r\'esultats connus sur les espaces analytiques au-dessus d'un corps.
\end{rema}

\begin{theo}\label{theo:noetherienepais}
Soit~$b$ un point de $B=\Mc(\As)$. Soit~$\Us$ un syst\`eme fondamental fin de voisinages compacts et spectralement convexes du point~$b$. Si~$\Hs(b)$ est de caract\'eristique non nulle et trivialement valu\'e, supposons que tout \'el\'ement de~$\Us$ est contenu dans~$B_{\textrm{um}}$ et poss\`ede un bord analytique fini. 

Supposons que l'anneau local~$\Os_{B,b}$ est un anneau local noeth\'erien, fortement r\'egulier relativement \`a~$\Us$ et de dimension inf\'erieure \`a~1. Soient $n\in\N$ et~$x$ un point rigide \'epais de~$\E{n}{\As}$ au-dessus de~$b$.

Alors l'anneau local~$\Os_{\E{n}{\As},x}$ est noeth\'erien et fortement r\'egulier et nous avons
\[\dim(\Os_{\E{n}{\As},x}) = \dim(\Os_{B,b})+n.\]

Si~$\Hs(x)$ est de caract\'eristique non nulle et trivialement valu\'e, il existe un syst\`eme fondamental fin~$\Vs$ de voisinages compacts et spectralement convexes de~$x$ tel que tout \'el\'ement de~$\Vs$ soit contenu dans~$(\E{n}{\As})_{\textrm{um}}$ et poss\`ede un bord analytique fini et tel que l'anneau local~$\Os_{X_{n},x}$ soit fortement r\'egulier relativement \`a~$\Vs$. 
\end{theo}
\begin{proof}
Nous allons nous contenter de d\'emontrer ce r\'esultat lorsque~$\Os_{B,b}$ est un anneau fortement de valuation discr\`ete, le cas o\`u c'est un corps fort se traitant de fa\c{c}on similaire mais plus simple. Soit~$\pi$ une uniformisante forte de~$\Os_{B,b}$. 

Nous ne traiterons pas le cas o\`u~$\Hs(b)$ est de caract\'eristique non nulle et trivialement valu\'e, ni l'assertion finale, o\`u il faut s'assurer que la propri\'et\'e d'\^etre ultram\'etrique typique reste v\'erifi\'ee \`a chaque \'etape du raisonnement. Cela n'entra\^{\i}ne aucune difficult\'e suppl\'ementaire. 

Commen\c{c}ons par montrer que~$\Os_{X_{n},x}$ est noeth\'erien. Nous allons proc\'eder par r\'ecurrence sur~$n$. Si~$n=0$, cela d\'ecoule des hypoth\`eses.

Supposons maintenant que~$n\ge 1$ et que le r\'esultat est vrai en dimension~\mbox{$n-1$}. Soit~$I$ un id\'eal de~$\Os_{X_{n},x}$. Nous pouvons supposer qu'il n'est pas nul. En appliquant le lemme~\ref{lem:divisionparpi}, on montre qu'il existe un entier~$v$ et un id\'eal~$J$ de~$\Os_{X_{n},x}$ tels que $I = \pi^{v}\, J$ et~$J$ contient un \'el\'ement~$f$ dont l'image dans~$\Os_{(X_{n})_{b},x}$ n'est pas nulle. 

Il suffit de d\'emontrer que l'id\'eal~$J$ est de type fini. Notons~$y$ la projection de~$x$ sur ses $n-1$ premi\`eres coordonn\'ees. D'apr\`es le lemme~\ref{lem:chgtvar}, quitte \`a effectuer un changement de variables, nous pouvons supposer que l'image de~$f$ dans~$\Os_{(X_{n})_{y},x}$ n'est pas nulle. Nous pouvons alors utiliser le th\'eor\`eme de division de Weierstra{\ss}~\ref{theo:Weierstrassapproche} (en choisissant pour~$B$ une alg\`ebre de la forme~$\Bs(W)$ o\`u~$W$ est un voisinage rationnel de~$y$ dans~$X_{n-1}$ assez petit et pour~$b$ le point~$y$) et diviser tout \'el\'ement de~$\Os_{X_{n},x}$ par~$f$ avec un reste polynomial. On se ram\`ene donc \`a montrer qu'un id\'eal d'un anneau de polyn\^omes sur~$\Os_{X_{n-1},y}$ est de type fini. Par r\'ecurrence, l'anneau~$\Os_{X_{n-1},y}$ est noeth\'erien et le r\'esultat s'ensuit.

Int\'eressons-nous \`a pr\'esent \`a la r\'egularit\'e et la dimension de l'anneau~$\Os_{X_{n},x}$. Si~$n=0$, les r\'esultats annonc\'es sont vrais. 

Supposons que $n\ge 1$ et qu'ils sont v\'erifi\'es en dimension~$n-1$. Notons~$x'$ le point~$0$ de la fibre de~$X_{n}$ au-dessus de~$y$. Soit $P(S) \in \kappa(y)[S]$ le polyn\^ome minimal du point rigide \'epais~$x$ au-dessus de~$y$. Relevons-le en un polyn\^ome unitaire \`a coefficients dans~$\Os_{X_{n-1},y}$ que nous noterons identiquement. D'apr\`es le th\'eor\`eme~\ref{theo:finiPS-T}, il existe un isomorphisme naturel $\Os_{X_{n},x'}[S]/(P(S)-T) \simeq \Os_{X_{n},x}$. Sur l'expression explicite de~$\Os_{X_{n},x'}$ comme anneau de s\'eries convergentes, on remarque qu'il est de dimension sup\'erieure \`a $\dim(\Os_{X_{n-1},y})+ 1 = n +1$, par hypoth\`ese de r\'ecurrence. Le m\^eme r\'esultat vaut pour~$\Os_{X_{n},x}$. 

D'apr\`es le corollaire~\ref{cor:fortepais}, l'id\'eal maximal~$\m_{x}$ est engendr\'e $\Bs$-fortement par $n+1$~\'el\'ements. On en d\'eduit que l'anneau local est de dimension~$n+1$, r\'egulier et m\^eme fortement r\'egulier.
\end{proof}

Nous arrivons finalement au th\'eor\`eme annonc\'e. Les hypoth\`eses en sont notamment v\'erifi\'ees lorsque l'on choisit pour anneau de Banach~$\As$ l'un de ceux qui figurent \`a la remarque~\ref{rem:liste}. Le r\'esultat vaut donc en particulier sur~$\Z$ et les anneaux d'entiers de corps de nombres.

\begin{theo}\label{theo:noetherien}
Soit~$b$ un point de $B=\Mc(\As)$. Soit~$\Us$ un syst\`eme fondamental fin de voisinages compacts et spectralement convexes du point~$b$. Si~$\Hs(b)$ est de caract\'eristique non nulle et trivialement valu\'e, supposons que tout \'el\'ement de~$\Us$ est contenu dans~$B_{\textrm{um}}$ et poss\`ede un bord analytique fini. 

Supposons que l'anneau local~$\Os_{B,b}$ est noeth\'erien, fortement r\'egulier relativement \`a~$\Us$ et de dimension inf\'erieure \`a~$1$. Soient $n\in\N$ et~$x$ un point de $\E{n}{\As}$ au-dessus de~$b$. Alors l'anneau local~$\Os_{\E{n}{\As},x}$ est noeth\'erien et fortement r\'egulier.

Si~$\Hs(x)$ est de caract\'eristique non nulle et trivialement valu\'e, il existe un syst\`eme fondamental fin~$\Vs$ de voisinages compacts et spectralement convexes de~$x$ tel que tout \'el\'ement de~$\Vs$ soit contenu dans~$(\E{n}{\As})_{\textrm{um}}$ et poss\`ede un bord analytique fini et tel que l'anneau local~$\Os_{X_{n},x}$ soit fortement r\'egulier relativement \`a~$\Vs$. 
\end{theo}
\begin{proof}
Comme pr\'ec\'edemment, nous ne traiterons pas le cas o\`u~$\Hs(b)$ est imparfait et trivialement valu\'e, qui n'ajoute pas de difficult\'e r\'eelle. 

Pour tout~$m\in\cn{0}{n}$, notons $x_{m}\in \E{m}{\As}$ la projection du point~$x$ sur ses $m$~premi\`eres coordonn\'ees. Quitte \`a changer l'ordre des coordonn\'ees, nous pouvons supposer qu'il existe un entier $k\in\cn{0}{n}$ v\'erifiant les propri\'et\'es suivantes :
\begin{enumerate}[i)]
\item le point~$x_{k}$ est purement localement transcendant au-dessus de~$b$ ;
\item le point~$x$ est rigide \'epais au-dessus de~$x_{k}$.
\end{enumerate}

D'apr\`es le corollaire~\ref{cor:fortpasepais}, si l'anneau local~$\Os_{B,b}$ est noeth\'erien et fortement r\'egulier de dimension inf\'erieure \`a~$1$, alors il en est de m\^eme pour l'anneau local~$\Os_{\E{k}{\As},x_{k}}$. On conclut alors par le th\'eor\`eme~\ref{theo:noetherienepais}.
\end{proof}

\begin{coro}
Soit~$\As$ un anneau de Banach de base au sens de la d\'efinition~\ref{defi:debase}. Soit~$(Z,\Os_{Z})$ un espace analytique sur~$\As$. Alors, en tout point~$z$ de~$Z$, l'anneau local~$\Os_{Z,z}$ est noeth\'erien et universellement cat\'enaire.
\end{coro}

\section{Excellence des anneaux locaux}

Dans cette partie, nous allons nous int\'eresser \`a l'excellence des anneaux locaux des espaces analytiques sur~$\Z$ et les anneaux d'entiers de corps de nombres. Pour les d\'efinitions et propri\'et\'es de base des anneaux excellents, nous renvoyons \`a~\cite{Matsumura}, \S~34. En g\'eom\'etrie analytique complexe, ce r\'esultat est bien connu et se d\'emontre g\'en\'eralement \`a l'aide de crit\`eres jacobiens (\textit{cf.}~\cite{Matsumura}, th\'eor\`eme~102 et remarque suivante). Sur un corps valu\'e complet, c'est un r\'esultat difficile, en particulier en caract\'eristique non nulle, d\^u en toute g\'en\'eralit\'e \`a A.~Ducros (\textit{cf.}~\cite{excellence}, th\'eor\`eme~2.13).

Comme dans la partie pr\'ec\'edente, nos r\'esultats valent en r\'ealit\'e pour les anneaux locaux des espaces~$\E{n}{\As}$, o\`u~$\As$ appartient \`a une classe d'anneaux de Banach plus g\'en\'erale. Celle-ci contient tous les exemples indiqu\'es \`a la remarque~\ref{rem:liste}, avec l'hypoth\`ese suppl\'ementaire que les anneaux consid\'er\'es soient de caract\'eristique nulle. 

Signalons que, cette fois-ci, nous ne red\'emontrons pas les r\'esultats d'excellence pour les anneaux locaux des espaces sur les corps valu\'es, mais les utilisons dans notre preuve (pour les corps de caract\'eristique non nulle uniquement). 

\begin{theo}
Soit~$b$ un point de $B$. Soit~$\Us$ un syst\`eme fondamental fin de voisinages compacts et spectralement convexes du point~$b$. Si~$\Hs(b)$ est de caract\'eristique non nulle et trivialement valu\'e, supposons que tout \'el\'ement de~$\Us$ est contenu dans~$B_{\textrm{um}}$ et poss\`ede un bord analytique fini. 

Supposons que l'anneau local~$\Os_{B,b}$ est un corps fort relativement \`a~$\Us$ de caract\'eristique nulle. Soient $n\in\N$ et~$x$ un point de $\E{n}{\As}$ au-dessus de~$b$. Alors l'anneau local~$\Os_{\E{n}{\As},x}$ est excellent.
\end{theo}
\begin{proof}
Nous allons d\'emontrer ce r\'esultat par r\'ecurrence sur~$n$. Si $n=0$, il est v\'erifi\'e. Supposons maintenant que~$n\ge 1$ et que le r\'esultat est vrai en dimension~$n-1$.

En proc\'edant comme dans la preuve du th\'eor\`eme~\ref{theo:noetherien}, on se ram\`ene au cas o\`u~$x$ est un point rigide \'epais au-dessus de~$b$. D'apr\`es le corollaire~\ref{cor:finidimsup} et la remarque qui le suit, l'anneau local~$\Os_{X_{n},x}$ est fini sur l'anneau local~$\Os_{X_{n},0}$ au point~$0$ de la fibre~$(X_{n})_{b}$.  Puisque l'excellence est stable par passage \`a une alg\`ebre de type fini, nous pouvons supposer que~$x=0$. L'anneau local~$\Os_{X_{n},x}$ poss\`ede alors une description explicite comme alg\`ebre de s\'eries convergentes \`a coefficients dans le corps~$\Os_{B,b}$ (\textit{cf.}~corollaire~\ref{cor:dvptpoint}).

Nous savons d\'ej\`a que~$\Os_{X_{n},x}$ est noeth\'erien et universellement cat\'enaire (car r\'egulier). Il nous reste \`a montrer que ses fibres formelles sont g\'eom\'etriquement r\'eguli\`eres. Soit~$\p$ un id\'eal premier de~$\Os_{X_{n},x}$. Nous voulons montrer que l'alg\`ebre $\widehat{\Os_{X_{n},x}} \otimes_{\Os_{X_{n},x}} \kappa(\p)$ est g\'eom\'etriquement r\'eguli\`ere sur~$\kappa(\p)$. 

Supposons tout d'abord que~$\p$ n'est pas nul. Il contient alors une s\'erie~$f$ non constante. Notons $y\in X_{n-1}$ la projection de~$x$ sur ses $n-1$ premi\`eres coordonn\'ees. D'apr\`es le lemme~\ref{lem:chgtvar}, quitte \`a changer l'ordre des variables, nous pouvons supposer que la restriction de~$f$ \`a~$(X_{n})_{y}$ n'est pas nulle. Le th\'eor\`eme de pr\'eparation de Weierstra{\ss}~\ref{theo:preparationWeierstrass} permet alors de supposer que~$f$ est un polyn\^ome en~$T_{n}$. D'apr\`es le th\'eor\`eme de division de Weierstra{\ss}, nous avons un isomorphisme
\[\Os_{X_{n},x}/(f) \simeq \Os_{X_{n-1},y}[T_n]/(f).\] 
Pour montrer que $\widehat{\Os_{X_{n},x}} \otimes_{\Os_{X_{n},x}} \kappa(\p)$ est g\'eom\'etriquement r\'eguli\`ere sur~$\kappa(\p)$, nous pouvons remplacer~$\Os_{X_{n},x}$ par $\Os_{X_{n},x}/(f)$. Par r\'ecurrence, ce dernier anneau est excellent, donc le r\'esultat est v\'erifi\'e.

Supposons maintenant que~$\p=(0)$. Notons~$K$ le corps des fractions de~$\Os_{X_{n},x}$. Il nous reste \`a montrer que l'alg\`ebre $\widehat{\Os_{X_{n},x}}\otimes_{\Os_{X_{n},x}} K$ est g\'eom\'etriquement r\'eguli\`ere sur~$K$. Remarquons d\'ej\`a qu'elle est r\'eguli\`ere, puisque c'est un localis\'e de~$\widehat{\Os_{X_{n},x}}$. Puisque nous nous sommes plac\'es en caract\'eristique nulle, cela suffit pour conclure.
\end{proof}

\begin{theo}
Soit~$b$ un point de $B$. Soit~$\Us$ un syst\`eme fondamental fin de voisinages compacts et spectralement convexes du point~$b$. Si~$\Hs(b)$ est de caract\'eristique non nulle et trivialement valu\'e, supposons que tout \'el\'ement de~$\Us$ est contenu dans~$B_{\textrm{um}}$ et poss\`ede un bord analytique fini. 

Supposons que l'anneau local~$\Os_{B,b}$ est noeth\'erien, fortement r\'egulier relativement \`a~$\Us$, de dimension inf\'erieure \`a~$1$ et de caract\'eristique nulle. Soient $n\in\N$ et~$x$ un point de $\E{n}{\As}$ au-dessus de~$b$. Alors l'anneau local~$\Os_{\E{n}{\As},x}$ est excellent.
\end{theo}
\begin{proof}
Le th\'eor\`eme qui pr\'ec\`ede traitant le cas des corps forts, nous pouvons supposer que~$\Os_{B,b}$ est un anneau fortement de valuation discr\`ete. Choi\-sis\-sons-en une uniformisante~$\pi$. Nous allons d\'emontrer le r\'esultat par r\'ecurrence sur~$n$. Si $n=0$, il est v\'erifi\'e. Supposons maintenant que~$n\ge 1$ et que le r\'esultat est vrai en dimension~$n-1$.

En proc\'edant comme pr\'ec\'edemment, on se ram\`ene au cas du point~$x=0$ de la fibre~$(X_{n})_{b}$ et \`a montrer \`a montrer que les fibres formelles de~$\Os_{X_{n},x}$ sont g\'eom\'etriquement r\'eguli\`eres. Soit~$\p$ un id\'eal premier de~$\Os_{X_{n},x}$. Nous voulons montrer que l'alg\`ebre $\widehat{\Os_{X_{n},x}} \otimes_{\Os_{X_{n},x}} \kappa(\p)$ est g\'eom\'etriquement r\'eguli\`ere sur~$\kappa(\p)$. 

Si~$\p$ contient~$\m_{b}$, alors nous pouvons remplacer l'anneau local $\Os_{X_{n},x}$ par $\Os_{X_{n},x}/\m_{b} \simeq \Os_{(X_{n})_{b},x}$. Le r\'esultat d\'ecoule alors du th\'eor\`eme~2.13 de~\cite{excellence}.

Supposons, \`a pr\'esent, que~$\p$ ne contient pas~$\m_{b}$, autrement dit, ne contient pas~$\pi$. Si~$\p$ n'est pas nul, il contient une s\'erie qui n'est pas divisible par~$\pi$. Nous pouvons alors appliquer le raisonnement du th\'eor\`eme pr\'ec\'edent, bas\'e sur les th\'eor\`emes de Weierstra{\ss}, pour faire d\'ecro\^{\i}tre la dimension et conclure par r\'ecurrence.

Finalement, le cas o\`u~$\p=(0)$ se traite comme pr\'ec\'edemment en utilisant la r\'egularit\'e de $\widehat{\Os_{X_{n},x}}$ et le fait que le corps des fractions de~$\Os_{X_{n},x}$ est de caract\'eristique nulle. 
\end{proof}

\begin{coro}
Soit~$\As$ un anneau de Banach de base au sens de la d\'efinition~\ref{defi:debase} et supposons que, pour tout point~$b$ de~$\Mc(\As)$, l'anneau local~$\Os_{\Mc(\As),b}$ est de caract\'eristique nulle. Soit~$(Z,\Os_{Z})$ un espace analytique sur~$\As$. Alors, en tout point~$z$ de~$Z$, l'anneau local~$\Os_{Z,z}$ est excellent.
\end{coro}

On peut imaginer diverses applications de ces r\'esultats. Les anneaux des espaces analytiques sur~$\Z$ \'etant hens\'eliens et excellents, ils v\'erifient l'approximation d'Artin, d'apr\`es~\cite{Popescuapproximation}. Nous pouvons par exemple appliquer ce r\'esultat \`a l'anneau local~$\Os_{0}$ au point~$0$ de la fibre de~$\E{n}{\Z}$ au-dessus de la valeur absolue triviale. Cet anneau poss\`ede une description explicite : il est constitu\'e des s\'eries de la forme 
\[\sum_{\bk\in\N^n} a_{\bk}\, \bT^\bk \in \Z\left[\frac 1 N \right] \llbracket\bT\rrbracket,\]
l'entier~$N$ d\'ependant de la s\'erie, qui poss\`edent un rayon de convergence strictement positif en toutes les places (\textit{cf.}~\cite{asterisque}, proposition~3.2.7 pour une description valant pour tous les anneaux d'entiers de corps de nombres et une preuve). On en d\'eduit que tout syst\`eme d'\'equations polynomiales \`a coefficients dans~$\Os_{0}$ (par exemple un syst\`eme \`a coefficients dans~$\Q[\bT]$) qui poss\`ede une solution dans $\hat{\Os_{0}} = \Q\llbracket \bT \rrbracket$ en poss\`ede une dans~$\Os_{0}$, et m\^eme que l'on peut trouver une solution dans~$\Os_{0}$ arbitrairement proche (pour la topologie $\bT$-adique) d'une solution formelle donn\'ee. C'est une vaste g\'en\'eralisation d'un th\'eor\`eme d'Eisenstein concernant le cas d'une \'equation en une variable.

\section{Coh\'erence du faisceau structural}

Pour finir, nous nous int\'eressons \`a la coh\'erence du faisceau structural des espaces analytiques sur~$\Z$ et les anneaux d'entiers de corps de nombres. C'est l'analogue du th\'eor\`eme fondateur de K.~Oka en g\'eom\'etrie anaytique complexe. De nouveau, la preuve que nous proposons vaut dans un cadre plus g\'en\'eral, par exemple pour les espaces analytiques sur les corps valu\'es complets (archim\'ediens ou non) et les anneaux de valuation discr\`ete.

La coh\'erence du faisceau structural est d\'ej\`a connue pour les espaces de Berkovich sur un corps valu\'e ultram\'etrique complet. On la d\'eduit ais\'ement de la noeth\'erianit\'e des alg\`ebres affino\"{\i}des, elle-m\^eme d\'ecoulant du th\'eor\`eme de division de Weierstra{\ss} global pour les alg\`ebres de Tate (\textit{cf.}~\cite{excellence}, lemme~0.1 pour une preuve d\'etaill\'ee). Cette m\'ethode ne nous semble pas pouvoir \^etre adapt\'ee pour des espaces sur un anneau de Banach quelconque. En effet, sur~$\C$ d\'ej\`a, la noeth\'erianit\'e de l'anneau des fonctions analytiques au voisinage d'un disque de dimension sup\'erieure \`a~$2$ est un r\'esultat difficile (\textit{cf.}~\cite{Frisch}, th\'eor\`eme~(I,9)), dont la preuve utilise d'ailleurs la coh\'erence du faisceau structural. C'est pourquoi nous avons choisi d'utiliser plut\^ot des m\'ethodes locales adapt\'ees de la g\'eom\'etrie analytique complexe.

\medskip

Commen\c{c}ons par quelques r\'esultats de prolongement analytique.

\begin{defi}
Soit~$(S,\Os_{S})$ un espace localement annel\'e. Soit~$s$ un point de~$S$. Nous dirons que \textbf{le principe du prolongement analytique vaut au voisinage de~$s$} si, pour tout \'el\'ement non nul~$f$ de~$\Os_{S,s}$, il existe un voisinage ouvert~$U$ de~$s$ v\'erifiant les propri\'et\'es suivantes :
\begin{enumerate}[i)]
\item $f$ est d\'efinie sur~$U$ ;
\item pour tout point~$t$ de~$U$, l'image de~$f$ dans~$\Os_{S,t}$ n'est pas nulle. 
\end{enumerate}

Nous dirons que \textbf{le principe du prolongement analytique vaut sur~$S$} s'il vaut au voisinage de tous ses points.
\end{defi}

\begin{rema}
Soit~$(S,\Os_{S})$ un espace localement annel\'e sur lequel vaut le principe du prolongement analytique. La version classique du principe du prolongement analytique est alors v\'erifi\'ee. Soient~$U$ une partie ouverte et connexe de~$S$ et~$f$ un \'el\'ement de~$\Os_{S}(U)$. S'il existe un point~$s$ de~$U$ tel que~$f$ soit nulle dans~$\Os_{S,s}$, alors~$f$ est nulle dans~$\Os_{S}(U)$.

Remarquons \'egalement que le principe du prolongement analytique vaut au voisinage de tout point~$s$ en lequel l'anneau local~$\Os_{S,s}$ est un corps.
\end{rema}

\begin{lemm}\label{lem:prolongementpurloctr}
Soit~$b$ un point de~$B$. Soit~$\Us$ un syst\`eme fondamental fin de voisinages compacts et spectralement convexes du point~$b$. Si~$\Hs(b)$ est de caract\'eristique non nulle et trivialement valu\'e, supposons que tout \'el\'ement de~$\Us$ est contenu dans~$B_{\textrm{um}}$ et poss\`ede un bord analytique fini. Supposons que~$\Os_{B,b}$ est un anneau fortement de valuation discr\`ete relativement \`a~$\Us$ et que le principe du prolongement analytique vaut au voisinage de~$b$.

Soit~$x$ un point de~$X_{n}$, avec $n\in\N$. Si le point~$x$ est purement localement transcendant, alors l'anneau local~$\Os_{X_{n},x}$ est un anneau fortement de valuation discr\`ete et le principe du prolongement analytique vaut au voisinage de~$x$. 
\end{lemm}
\begin{proof}
Soit~$\pi$ une uniformisante de~$\Os_{B,b}$. D'apr\`es le corollaire~\ref{cor:fortpasepais}, l'anneau local~$\Os_{X_{n},x}$ est un anneau fortement de valuation discr\`ete d'uniformisante~$\pi$. Pour montrer que le principe du prolongement analytique vaut au voisinage de~$x$, il suffit de montrer qu'il existe un voisinage ouvert~$U$ de~$x$ sur lequel son uniformisante~$\pi$ est d\'efinie et tel que, pour tout point~$y$ de~$U$, l'image de~$\pi$ dans~$\Os_{X_{n},y}$ n'est pas nulle. Choisissons pour~$U$ l'image r\'eciproque d'un voisinage ouvert de~$b$ sur lequel la m\^eme propri\'et\'e est v\'erifi\'ee.  On montre que l'ouvert~$U$ convient en utilisant le fait que le morphisme de projection est ouvert (\textit{cf.}~corollaire~\ref{cor:projouverte}).
\end{proof}

\begin{prop}\label{prop:prolongementfibre}
Soient~$b$ un point de~$B$ et~$x$ un point rigide \'epais de~$X_{b}$. Soit~$f$ un \'el\'ement de~$\Os_{X,x}$ dont l'image dans~$\Os_{X_{b},x}$ n'est pas nulle. Alors il existe un voisinage~$U$ de~$x$ tel que, pour tout point~$y$ de~$U$, l'image de~$f$ dans~$\Os_{X_{\pi(y)},y}$ ne soit pas nulle.
\end{prop}
\begin{proof}
Supposons tout d'abord qu'il existe $\alpha\in \Os_{B,b}$ tel que $x=\alpha$ dans~$X_{b}$. Dans ce cas, on d\'emontre le r\'esultat en utilisant la description explicite de l'anneau local (\textit{cf.}~corollaire~\ref{cor:dvptpoint}) et la description des voisinages comme des disques. Plus pr\'ecis\'ement, nous pouvons montrer l'existence d'un voisinage ouvert~$U$ de~$x$ tel que, pour tout $c\in\pi(U)$, l'ensemble des points~$y$ de $X_{c}\cap U$ o\`u~$f(y)$ n'est pas nulle est dense dans $X_{c}\cap U$.

Passons maintenant au cas o\`u le point~$x$ est un point rigide \'epais. Il est annul\'e par un polyn\^ome $P(T)\in\Os_{B,b}[T]$. Nous pouvons supposer que $P(T)\in\As[T]$. Consid\'erons l'alg\`ebre de d\'ecomposition universelle~$\As'$ du polyn\^ome~$P$ sur~$\As$. C'est une alg\`ebre finie sur~$\As$, que l'on peut munir d'une norme de Banach de fa\c{c}on que le morphisme $\As \to \As'$ soit born\'e. Consid\'erons le morphisme $ \varphi : X' = \E{1}{\As'} \to X$ associ\'e. C'est un morphisme fini au sens topologique (\textit{i.e.}~continu, ferm\'e et \`a fibres finies) et surjectif. Posons $\varphi^{-1}(x) = \{x'_{1},\dots,x'_{r}\}$. Soit $i\in\cn{1}{r}$. Il existe $\alpha'_{i}\in \As'$ tel que le point~$x'_{i}$ soit \'egal \`a~$\alpha'_{i}$ dans sa fibre au-dessus de~$\Mc(\As')$. D'apr\`es le raisonnement effectu\'e au d\'ebut de la preuve, il poss\`ede donc un voisinage ouvert~$U'_{i}$ tel que l'ensemble des points ou~$f$ n'est pas nulle soit dense dans chaque fibre. Puisque le morphisme~$\varphi$ est fini, l'image r\'eciproque de la r\'eunion des~$U'_{i}$ est un voisinage de~$x$ dans~$X$. Il satisfait la propri\'et\'e requise.
\end{proof}

\begin{coro}
Soit~$b$ un point de~$B$. Soit~$\Us$ un syst\`eme fondamental fin de voisinages compacts et spectralement convexes du point~$b$. Si~$\Hs(b)$ est de caract\'eristique non nulle et trivialement valu\'e, supposons que tout \'el\'ement de~$\Us$ est contenu dans~$B_{\textrm{um}}$ et poss\`ede un bord analytique fini. Supposons que l'anneau local~$\Os_{B,b}$ est noeth\'erien, fortement r\'egulier relativement \`a~$\Us$ et de dimension inf\'erieure \`a~$1$. 

Soit~$x$ un point de~$X_{n}$, avec $n\in\N$, au-dessus de~$b$. Supposons que le principe du prolongement analytique vaut au voisinage de~$b$. Alors il vaut au voisinage du point~$x$.
\end{coro}
\begin{proof}
Nous supposerons que l'anneau local~$\Os_{B,b}$ est un anneau fortement de valuation discr\`ete, le cas o\`u c'est un corps fort se traitant de fa\c{c}on similaire, mais plus simple.

Quitte \`a changer l'ordre des variables, nous pouvons supposer qu'il existe un entier $m\in\cn{0}{n}$ tel que, en notant~$y$ la projection de~$x$ sur ses $m$~premi\`eres coordonn\'ees, nous ayons les propri\'et\'es suivantes : $x$ est un point rigide \'epais au-dessus de~$y$ et~$y$ est purement localement transcendant au-dessus de~$b$. D'apr\`es le lemme~\ref{lem:prolongementpurloctr}, l'anneau local~$\Os_{X_{m},y}$ est un anneau fortement de valuation discr\`ete et le principe du prolongement analytique vaut au voisinage du point~$y$. Quitte \`a remplacer~$b$ par~$y$ et~$B$ par un voisinage compact rationnel de~$y$, nous pouvons supposer que le point~$x$ est rigide \'epais au-dessus de~$b$.

Si $n=0$, le r\'esultat est imm\'ediat. Supposons donc que~$n\ge 1$. Soit~$f$ un \'el\'ement de~$\Os_{X_{n},x}$. Soit~$\pi$ une uniformisante de~$\Os_{B,b}$. D'apr\`es le lemme~\ref{lem:divisionparpi}, il existe~$v\in\N$ et~$g\in\Os_{X_{n},x}$ tels que $f=\pi^v g$ dans $\Os_{X_{n},x}$ et $g\ne 0$ dans $\Os_{(X_{n})_{b},x}$. Nous voulons montrer qu'il existe un voisinage ouvert~$U$ de~$x$ sur lequel~$f$ est d\'efini et tel que~$f$ soit non nulle dans~$\Os_{X_{n},z}$, pour tout $z\in U$. Il suffit de d\'emontrer ce r\'esultat pour~$g$. D'apr\`es le lemme~\ref{lem:chgtvar}, quitte \`a effectuer un changement de variables, nous pouvons supposer que l'image de~$g$ dans~$\Os_{(X_{n})_{x'},x}$, o\`u~$x'$ est la projection de~$x$ sur ses $n-1$~premi\`eres coordonn\'ees n'est pas nulle. La proposition qui pr\'ec\`ede permet alors de conclure. 
\end{proof}

Nous allons, \`a pr\'esent, nous atteler \`a la preuve du r\'esultat de coh\'erence annonc\'e. Pour ce faire, il nous faudra imposer des conditions sur l'espace \mbox{$B=\Mc(\As)$}. Comme dans la preuve de la noeth\'erianit\'e et de la r\'egularit\'e des anneaux locaux, nous supposerons qu'en tout point~$b$ de~$B$, l'anneau local~$\Os_{B,b}$ est un corps fort ou un anneau fortement de valuation discr\`ete. Nous supposerons en outre que le principe du prolongement analytique vaut sur~$B$. Ces conditions sont v\'erifi\'ees dans le cas des espaces dont nous avons dress\'e la liste \`a la remarque~\ref{rem:liste}. 

Commen\c{c}ons par un lemme, de d\'emonstration imm\'ediate, qui nous permettra, entre autres, de traiter le cas des points o\`u l'anneau local est un corps.

\begin{lemm}\label{lem:coherentcorps}
Soient~$n\in\N$ et~$x$ un point de~$X_{n}$. Soient~$U$ un voisinage ouvert de~$x$ dans~$X_{n}$ et $f_{1},\dots,f_{p}$, avec $p\in\N^*$, des \'el\'ements de~$\Os(U)$. Notons $(e_{1},\dots,e_{p})$ la base canonique de~$\Os_{X}^p$. Notons~$\Rs$ le noyau du morphisme 
\[\sum_{i=1}^p a_{i}e_{i} \in \Os_{U}^p \mapsto \sum_{i=1}^p a_{i}f_{i} \in \Os_{U}.\]

Supposons qu'il existe un indice~$i$ tel que l'image de~$f_{i}$ dans~$\Os_{X_{n},x}$ soit inversible. Alors il existe un voisinage ouvert~$V$ de~$x$ dans~$U$ tel que, pour tout $y\in V$, la famille $(f_{j}e_{i} - f_{i} e_{j})_{1\le i < j \le p}$ de~$\Os_{X,y}^p$ engendre le germe~$\Rs_{y}$.
\end{lemm}

Le lemme suivant nous permettra de traiter le cas des points o\`u l'anneau local est de valuation discr\`ete.

\begin{lemm}\label{lem:coherentavd}
Soit~$b$ un point de~$B$. Soit~$\Us$ un syst\`eme fondamental fin de voisinages compacts et spectralement convexes du point~$b$. Si~$\Hs(b)$ est de caract\'eristique non nulle et trivialement valu\'e, supposons que tout \'el\'ement de~$\Us$ est contenu dans~$B_{\textrm{um}}$ et poss\`ede un bord analytique fini. Supposons que~$\Os_{B,b}$ est un anneau fortement de valuation discr\`ete relativement \`a~$\Us$ et que le principe du prolongement analytique vaut au voisinage de~$b$. 

Soient~$n\in\N$ et~$x$ un point de~$(X_{n})_{b}$ purement localement transcendant. Soient~$U$ un voisinage ouvert de~$x$ dans~$X_{n}$ et $f_{1},\dots,f_{p}$, avec $p\in\N^*$, des \'el\'ements de~$\Os(U)$. Notons $(e_{1},\dots,e_{p})$ la base canonique de~$\Os_{X}^p$. Notons~$\Rs$ le noyau du morphisme 
\[\sum_{i=1}^p a_{i}e_{i} \in \Os_{U}^p \mapsto \sum_{i=1}^p a_{i}f_{i} \in \Os_{U}.\]

Supposons qu'il existe un indice~$i$ tel que l'image de~$f_{i}$ dans~$\Os_{X_{n},x}$ ne soit pas nulle. Alors il existe un voisinage ouvert~$V$ de~$x$ dans~$U$ tel que, pour tout $y\in V$, la famille $(f_{j}e_{i} - f_{i} e_{j})_{1\le i < j \le p}$ de~$\Os_{X,y}^p$ engendre le germe~$\Rs_{y}$.
\end{lemm}
\begin{proof}
D'apr\`es le lemme~\ref{lem:prolongementpurloctr}, l'anneau local~$\Os_{X_{n},x}$ est un anneau de valuation discr\`ete et le principe du prolongement analytique vaut au voisinage de~$x$. Choisissons une uniformisante~$\pi$ de~$\Os_{X_{n},x}$. Quitte \`a restreindre~$U$, nous pouvons supposer que l'image de~$\pi$ n'est nulle dans aucun des anneaux locaux~$\Os_{X_{n},z}$, avec $z\in U$. Il existe un entier positif~$v$ et des \'el\'ements $g_{1},\dots,g_{p}$ de~$\Os_{X_{n},x}$, tel que l'on ait $f_{i} = \pi^v g_{i}$, pour tout~$i$, et que l'un des~$g_{i}$ soit inversible dans~$\Os_{X_{n},x}$. Quitte \`a restreindre~$U$, nous pouvons supposer que la factorisation y vaut. Remarquons que, puisque les anneaux locaux de~$X_{n}$ sont int\`egres (car r\'eguliers), le noyau du morphisme de l'\'enonc\'e ne change pas lorsque l'on remplace les~$f_{i}$ par les~$g_{i}$. Le lemme pr\'ec\'edent permet alors de conclure.
\end{proof}

Nous pouvons maintenant adapter la preuve de la coh\'erence du faisceau structural sur les espaces affines complexes. Nous nous sommes inspir\'e de celle propos\'ee dans l'ouvrage~\cite{GuRo}. 

\begin{theo}[Lemme d'Oka]\label{theo:lemmeOka}
Soit~$n\in\N^*$. Soient~$y$ un point de~$X_{n-1}$ et~$x$ un point rigide \'epais de~$(X_{n})_{y}$. Soient $p,q\in\N^*$. Soit~$\lambda$ une matrice de taille $p\times q$ \`a coefficients dans~$\Os_{X_{n-1},y}[T_{n}]_{\le \ell}$, l'ensemble des polyn\^omes de degr\'e inf\'erieur \`a~$\ell$, avec $\ell\in\N$. Soit~$\mu$ une matrice de taille $q\times 1$ \`a coefficients dans~$\Os_{X_{n-1},y}[T_{n}]_{\le m}$, avec $m\in\N$. Soit $d\ge \max(\ell,m)$. Supposons que la suite
\[ \big(\Os_{X_{n-1},y}[T_{n}]_{\le d-\ell}\big)^p \xrightarrow[]{\lambda} \big(\Os_{X_{n-1},y}[T_{n}]_{\le d}\big)^q  \xrightarrow[]{\mu} \Os_{X_{n-1},y}[T_{n}]_{\le d+m} \]
est exacte. Supposons encore que l'un des coefficients de~$\mu$ n'est pas nul dans $\Os_{(X_{n})_{y},x}$. Alors la suite
\[\Os_{X_{n},x}^p \xrightarrow[]{\lambda} \Os_{X_{n},x}^q \xrightarrow[]{\mu} \Os_{X_{n},x} \]
est exacte. 
\end{theo}
\begin{proof}
On reprend la preuve, qui repose exclusivement sur le th\'eor\`eme de division de Weierstra{\ss}, du th\'eor\`eme II.C.3 de~\cite{GuRo}. 
\end{proof}

Nous arrivons finalement au th\'eor\`eme annonc\'e. Les hypoth\`eses en sont v\'erifi\'ees lorsque l'on choisit pour anneau de Banach~$\As$ l'un de ceux qui figurent \`a la remarque~\ref{rem:liste}. Le r\'esultat vaut donc en particulier sur~$\Z$ et les anneaux d'entiers de corps de nombres.

\begin{theo}
Soit~$\As$ un anneau de Banach de base au sens de la d\'efinition~\ref{defi:debase} et supposons que le principe du prolongement analytique vaut sur \mbox{$B = \Mc(\As)$}. Alors, pour tout~$n\in\N$, le faisceau structural sur~$X_{n} = \E{n}{\As}$ est coh\'erent.
\end{theo}
\begin{proof}
Nous reprenons ici la strat\'egie de la preuve du  th\'eor\`eme IV.C.1 de~\cite{GuRo}\footnote{Le lecteur prendra garde au fait que ce que nous appelons, selon la terminologie classique, faisceau coh\'erent est appel\'e dans cet ouvrage \og faisceau d'Oka \fg.} en lui apportant les modifications n\'ecessaires.

On proc\`ede par r\'ecurrence. Si $n=0$, le r\'esultat d\'ecoule des hypoth\`eses faites sur~$B$ ainsi que du lemme~\ref{lem:coherentavd}.

Supposons que~$n\ge1$ et que le faisceau structural sur~$X_{n-1}$ est coh\'erent. Soient~$b$ un point de~$B$, $x$~un point de~$X_{n}$ au-dessus de~$b$ et~$U$ un voisinage ouvert de~$x$ dans~$X_{n}$. Soient $f_{1},\dots,f_{p} \in\Os_{X_{n}}(U)$, avec $p\in\N^*$. Consid\'erons le morphisme $\mu : \Os_{U}^p \to \Os_{U}$ associ\'e. Nous voulons montrer qu'il existe un voisinage ouvert~$V$ de~$x$ dans~$U$ sur lequel le noyau de ce morphisme est de type fini. Si tous les~$f_{i}$ sont nuls dans~$\Os_{X_{n},x}$, c'est \'evident. Nous supposerons donc que tel n'est pas le cas. Si le point~$x$ est purement localement transcendant, cela d\'ecoule du lemme~\ref{lem:coherentavd}.

Supposons maintenant que le point~$x$ n'est pas purement localement transcendant. Pour $k\in\cn{0}{n}$, notons $x_{k} \in X_{k}$ la projection du point~$x$ sur ses $k$~premi\`eres coordonn\'ees. Quitte \`a changer l'ordre des variables, nous pouvons supposer qu'il existe $m\in\cn{0}{n-1}$ tel que le point~$x_{m}$ soit purement localement transcendant au-dessus de~$b$ et que le point~$x$ soit rigide \'epais au-dessus de~$x_{m}$. Dans ce cas, l'anneau local~$\Os_{X_{m},x_{m}}$ est un corps fort ou un anneau fortement de valuation discr\`ete et le principe du prolongement analytique vaut au voisinage de~$x_{m}$.

Supposons que~$\Os_{X_{m},x_{m}}$ est un anneau fortement de valuation discr\`ete. Choi\-sis\-sons-en une uniformisante~$\pi$. Quitte \`a restreindre~$U$, nous pouvons supposer qu'elle est d\'efinie sur~$U$ et nulle dans aucun anneau local~$\Os_{X_{n},z}$, avec $z\in U$. D'apr\`es le lemme~\ref{lem:divisionparpi}, il existe un entier positif~$v$ et des \'el\'ements $g_{1},\dots,g_{p}$ de~$\Os_{X_{n},x}$, tel que l'on ait $f_{i} = \pi^v g_{i}$, pour tout~$i$, et que l'un des ~$g_{i}$ ne soit pas nul dans~$\Os_{(X_{n})_{x_{m}},x}$.  En raisonnant comme dans la preuve du lemme~\ref{lem:coherentavd}, on montre que l'on peut remplacer les~$f_{i}$ par les~$g_{i}$ et donc supposer que l'un des~$f_{i}$ n'est pas nul dans~$\Os_{(X_{n})_{x_{m}},x}$. Remarquons que cette condition est automatiquement v\'erifi\'ee dans le cas o\`u $\Os_{X_{m},x_{m}}$ est un corps. Nous supposerons d\'esormais que~$f_{1}$ n'est pas nul dans~$\Os_{(X_{n})_{x_{m}},x}$.

D'apr\`es le lemme~\ref{lem:chgtvar}, nous pouvons supposer que~$f_{1}$ poss\`ede une image non nulle dans $\Os_{(X_{n})_{x_{n-1}},x}$. Le th\'eor\`eme de division de Weierstra{\ss} permet alors, par une manipulation alg\'ebrique simple, de supposer que $f_{2},\ldots,f_{p}$ sont des polyn\^omes \`a coefficients dans~$\Os_{X_{n-1},x_{n-1}}$. Le th\'eor\`eme de pr\'eparation de Weierstra{\ss} permet de supposer que~$f_{1}$ lui-m\^eme est un tel polyn\^ome. Soit~$m\in\N$ un majorant du degr\'e de tous ces polyn\^omes. Le morphisme~$\mu$ induit alors, sur un voisinage ouvert~$V$ de~$x_{n-1}$ un morphisme
\[ \big(\Os_{V}[T_{n}]_{\le m}\big)^p \xrightarrow[]{\mu} \Os_{V}[T_{n}]_{\le 2m}.\]
D'apr\`es l'hypoth\`ese de r\'ecurrence, les faisceaux qui apparaissent ci-dessus sont coh\'erents et le noyau du morphisme est donc de type fini. On peut donc compl\'eter le morphisme en une suite exacte
\[ \Os_{V}^r \xrightarrow[]{\lambda} \big(\Os_{V}[T_{n}]_{\le m}\big)^p \xrightarrow[]{\mu} \Os_{V}[T_{n}]_{\le 2m}.\]
Quitte \`a restreindre~$U$, nous pouvons supposer que sa projection sur~$X_{n-1}$ est contenue dans~$V$. D'apr\`es la proposition~\ref{prop:prolongementfibre}, nous pouvons \'egalement supposer que, pour tout point~$y$ de~$U$, l'image de~$f_{1}$ dans~$\Os_{X_{y_{n-1}},y}$ n'est pas nulle, o\`u~$y_{n-1}$ d\'esigne la projection de~$y$ sur ses $n-1$~premi\`eres coordonn\'ees.

Soit~$y$ un point de~$U$. S'il est rigide \'epais sur~$y_{n-1}$, le th\'eor\`eme~\ref{theo:lemmeOka} assure que la suite pr\'ec\'edente induit une suite exacte
\[ \Os_{U,y}^r \xrightarrow[]{\lambda} \Os_{U,y}^p \xrightarrow[]{\mu} \Os_{U,y}.\]
Maintenant, si le point~$y$ n'est pas rigide \'epais sur~$y_{n-1}$, d'apr\`es le corollaire~\ref{cor:pointpasepais}, nous avons $f_{1}(y)\ne 0$. Le lemme~\ref{lem:coherentcorps} assure alors que le noyau du morphisme $\mu : \Os_{U,y}^p \to \Os_{U,y}$ est engendr\'e par la famille $(f_{j}e_{i} - f_{i} e_{j})_{1\le i < j \le p}$, o\`u $(e_{1},\dots,e_{p})$ d\'esigne la base canonique de~$\Os_{U}^p$. En regroupant ces r\'esultats, nous montrons que le noyau du morphisme~$\mu$ est de type fini sur~$U$, ce qui conclut la preuve. 
\end{proof}

\begin{coro}
Soit~$\As$ un anneau de Banach de base au sens de la d\'efinition~\ref{defi:debase} et supposons que le principe du prolongement analytique vaut sur~$\Mc(\As)$. Soit~$(Z,\Os_{Z})$ un espace analytique sur~$\As$. Alors le faisceau structural~$\Os_{Z}$ est coh\'erent.
\end{coro}

Suivant une suggestion de P.~Schapira, nous d\'emontrons maintenant que le faisceau structural d'un espace analytique sur~$\Z$, ou un anneau d'entiers de corps de nombres, est noeth\'erien au sens de M.~Kashiwara. Rappelons tout d'abord la d\'efinition de cette notion (\textit{cf.}~\cite{KashiwaraDmodules}, d\'efinition~A.7).

\begin{defi}
Soit~$(S,\Os_{S})$ un espace localement annel\'e. Un faisceau~$\Fs$ de $\Os_{S}$-modules est dit \textbf{noeth\'erien} s'il v\'erifie satisfait les propri\'et\'es suivantes :
\begin{enumerate}[i)]
\item $\Fs$ est un faisceau de $\Os_{S}$-modules coh\'erent ;
\item pour tout point~$s$ de~$S$, le $\Os_{S,s}$-module~$\Fs_{s}$ est noeth\'erien ;
\item pour tout ouvert~$U$ de~$S$ et toute famille $(\Fs_{i})_{i\in I}$ de faisceaux de sous-$\Os_{U}$-modules coh\'erents de~$\Fs_{U}$, le faisceau de~$\Os_{U}$-modules~$\sum_{i\in I} \Fs_{i}$ est coh\'erent. 
\end{enumerate}

Le faisceau structural~$\Os_{S}$ est dit noeth\'erien s'il est noeth\'erien sur lui-m\^eme.
\end{defi}

Nous avons d\'ej\`a d\'emontr\'e les deux premi\`eres propri\'et\'es. La derni\`ere se d\'eduit ais\'ement du r\'esultat suivant (\textit{cf.}~\cite{Demailly}, II (3.22)).

\begin{prop}
Soit~$\As$ un anneau de Banach de base au sens de la d\'efinition~\ref{defi:debase} et supposons que le principe du prolongement analytique vaut sur $B=\Mc(\As)$. 

Soit~$(Z,\Os_{Z})$ un espace analytique sur~$\As$. Soit~$U$ un ouvert de~$Z$, $\Fs$ un faisceau de $\Os_{U}$-modules coh\'erent et $(\Fs_{k})_{k\ge 0}$ une suite croissante de sous-faisceaux coh\'erents de~$\Fs$. Alors, tout point~$x$ de~$U$ poss\`ede un voisinage~$V$ sur lequel la suite stationne.
\end{prop}
\begin{proof}
Le faisceau~$\Fs$ est localement quotient d'un $\Os_{Z}$-module libre~$\Os_{Z}^q$. En tirant en arri\`ere, on se ram\`ene \`a traiter le cas o\`u~$\Fs=\Os_{Z}^q$, puis $\Fs=\Os_{Z}$. En utilisant le fait que~$\Os_{Z}$ est lui-m\^eme localement quotient du faisceau structural sur un espace affine $X_{n} = \E{n}{\As}$, on se ram\`ene finalement \`a ce dernier cas.

Nous allons proc\'eder par r\'ecurrence sur l'entier~$n$. Si~$n=0$, le r\'esultat d\'ecoule facilement des hypoth\`eses.

Supposons maintenant que $n\ge 1$ et que le r\'esultat est v\'erifi\'e pour~$X_{n-1}$. Soit~$x$ un point de~$U$. Si pour tout~$k\ge 0$, nous avons~$\Fs_{k,x}=0$, alors la suite est constamment nulle au voisinage de~$x$. Supposons donc qu'il existe~$k_{0}\ge 0$ tel que~$\Fs_{k_{0},x} \ne 0$. Soit~$f$ un \'el\'ement non nul de~$\Fs_{k_{0},x}$. Si le point~$x$ est localement transcendant sur~$b$, alors~$\Os_{X_{n},x}$ est un corps et~$f$ est inversible sur un voisinage~$V$ de~$x$ dans~$U$. Dans ce cas, la suite stationne \`a~$\Os_{V}$ sur~$V$.

Il nous reste \`a traiter le cas o\`u~$x$ n'est pas localement transcendant. Quitte \`a effectuer un changement de variables, nous pouvons supposer que le point~$x$ est rigide \'epais sur sa projection~$x_{n-1}$ sur les $n-1$ premi\`eres coordonn\'ees et que l'image de~$f$ dans~$\Os_{(X_{n})_{x_{n-1}},x}$ n'est pas nulle. En utilisant le th\'eor\`eme de pr\'eparation de {Weiestra\ss}, nous pouvons nous ramener au cas o\`u~$f$ est un polyn\^ome, dont nous noterons~$d$ le degr\'e. D'apr\`es la proposition~\ref{prop:prolongementfibre}, il existe un voisinage~$V$ de~$x$ dans~$U$ tel que, pour tout point~$y$ de~$U$, l'image de~$f$ dans~$\Os_{X_{y_{n-1}},y}$ n'est pas nulle, o\`u~$y_{n-1}$ d\'esigne la projection de~$y$ sur ses $n-1$~premi\`eres coordonn\'ees. 

Notons~$\Ps$ le $\Os_{X_{n-1}}$-module coh\'erent form\'e des polyn\^omes en~$T_{n}$ de degr\'e strictement inf\'erieur \`a~$d$. Le th\'eor\`eme de division de Weierstra{\ss} assure que, pour tout point~$y$ de~$V$ et tout entier~$k\ge k_{0}$, la fibre~$\Fs_{k,y}$ est engendr\'ee par~$f$ et $\Fs_{k,y} \cap \Ps$. On conclut en appliquant l'hypoth\`ese de r\'ecurrence \`a la suite de~$\Os_{X_{n-1}}$-modules coh\'erents $(\Fs_{k,y} \cap \Ps)_{k\ge 0}$.
\end{proof}

\begin{coro}
Soit~$\As$ un anneau de Banach de base au sens de la d\'efinition~\ref{defi:debase} et supposons que le principe du prolongement analytique vaut sur~$\Mc(\As)$. Soit~$(Z,\Os_{Z})$ un espace analytique sur~$\As$. Alors le faisceau structural~$\Os_{Z}$ est noeth\'erien.
\end{coro}

\nocite{}
\bibliographystyle{smfalpha}
\bibliography{biblio}

\end{document}